\documentclass[reqno, a4paper]{amsart}
\usepackage{lmodern}								
\usepackage[latin1]{inputenc}        
\usepackage[T1]{fontenc}            
\usepackage{setspace} 								

\usepackage{textcomp}                
\usepackage{array}                   
\usepackage{enumerate}               
\usepackage{graphicx}

\usepackage[numbers]{natbib}										

\usepackage{amsmath}     					   
\usepackage{amsfonts}                
\usepackage{amssymb}                 
\usepackage{amsthm}									 
\usepackage{microtype}								
\usepackage{ellipsis}									

\usepackage{amsxtra,amscd, mathrsfs}	
\usepackage[all]{xy}									

\numberwithin{equation}{section}
\newtheorem{thm}{Theorem}[subsection] 
\newtheorem{prop}[thm]{Proposition}
\newtheorem{cor}[thm]{Corollary}			
\newtheorem{lemma}[thm]{Lemma}		
\newtheorem{thma}{Theorem}

\theoremstyle{definition}\newtheorem{defn}[thm]{Definition}						
\theoremstyle{definition} \newtheorem{art}[thm]{}
\theoremstyle{definition} \newtheorem{ex}[thm]{Example}
\theoremstyle{definition} \newtheorem{rem}[thm]{Remark}
\theoremstyle{definition} \newtheorem{nota}[thm]{Notation}

\theoremstyle{definition} 
\theoremstyle{definition} 
\theoremstyle{plain}
\theoremstyle{plain}
\theoremstyle{plain}

\DeclareMathOperator{\an}{{an}}
\newcommand{\Xan}{\ensuremath{X^{\text{\upshape{an}}}}}

\newcommand{\Uan}{\ensuremath{U^{\text{\upshape{an}}}}}
\newcommand{\Ban}{\ensuremath{B^{\text{\upshape{an}}}}}
\newcommand{\Van}{\ensuremath{V^{\text{\upshape{an}}}}}
\newcommand{\Lan}{\ensuremath{L^{\text{\upshape{an}}}}}
\newcommand{\Tan}{\ensuremath{\mathbb{T}^{\,\text{\upshape{an}}}}}

\newcommand{\Bgen}{\ensuremath{B^{\text{gen}}}}

\newcommand{\D}{\ensuremath{\hat{D}}}
\newcommand{\E}{\ensuremath{\hat{E}}}
\newcommand{\ML}{\ensuremath{\overline{L}}}
\newcommand{\MLcan}{\ensuremath{\overline{L}^{\text{\upshape{can}}}}}
\newcommand{\MM}{\ensuremath{\overline{M}}}
\newcommand{\MH}{\ensuremath{\overline{H}}}
\newcommand{\MN}{\ensuremath{\overline{N}}}

\newcommand{\mufin}{\ensuremath{\mu_{\text{\upshape{fin}}}}}

\newcommand{\lambdator}{\ensuremath{\lambda^{\text{\upshape{tor}}}}}

\DeclareMathOperator{\can}{can}

\newcommand{\RR}{\ensuremath{\mathbb{R}}}
\newcommand{\CC}{\ensuremath{\mathbb{C}}}
\newcommand{\QQ}{\ensuremath{\mathbb{Q}}}
\newcommand{\NN}{\ensuremath{\mathbb{N}}}
\newcommand{\ZZ}{\ensuremath{\mathbb{Z}}}
\newcommand{\TT}{\ensuremath{\mathbb{T}}}
\newcommand{\KK}{\ensuremath{\mathbb{C}}}
\newcommand{\FF}{\ensuremath{\mathbb{C}}}
\newcommand{\PP}{\ensuremath{\mathbb{P}}}
\DeclareMathOperator{\quot}{Quot}

\newcommand{\Lb}{\ensuremath{\mathbb{L}}}
\newcommand{\SSS}{\ensuremath{\mathbb{S}}}

\newcommand{\NR}{\ensuremath{N_{\RR}}}
\newcommand{\MRR}{\ensuremath{M_{\RR}}}
\newcommand{\NRR}{\ensuremath{N_{\RR}\times \RR_{\geq 0}}}

\newcommand{\Kval}{\ensuremath{K^{\circ}}}
\newcommand{\Kmax}{\ensuremath{K^{\circ\circ}}}

\newcommand{\tate}{\ensuremath{K\langle x_1,\ldots,x_n\rangle}}

\DeclareMathOperator{\spec}{Spec}
\DeclareMathOperator{\spf}{Spf}
\DeclareMathOperator{\dvs}{div}
\DeclareMathOperator{\Hom}{Hom}
\DeclareMathOperator{\pic}{Pic}
\DeclareMathOperator{\DivT}{Div_{\TT}}
\DeclareMathOperator{\cyc}{cyc}

\newcommand{\XS}{\ensuremath{X_{\Sigma}}}

\newcommand{\OO}{\ensuremath{\mathcal{O}}}
\newcommand{\XX}{\ensuremath{\mathcal{X}}}
\newcommand{\YY}{\ensuremath{\mathcal{Y}}}

\newcommand{\LL}{\ensuremath{\mathcal{L}}}
\newcommand{\M}{\ensuremath{\mathcal{M}}}
\newcommand{\WW}{\ensuremath{\mathcal{W}}}

\newcommand{\A}{\ensuremath{\mathcal{A}}}

\newcommand{\Bone}{\ensuremath{B^{(1)}}}
\newcommand{\Xone}{\ensuremath{\mathcal{X}^{(1)}}}

\newcommand{\SB}{\ensuremath{\widetilde{\BBB}}}
\newcommand{\SX}{\ensuremath{\widetilde{\XXX}}}
\newcommand{\SXG}{\ensuremath{\widetilde{\XXX}_{p}}}

\DeclareMathOperator{\m}{m}

\renewcommand{\vert}{{\text{\upshape{vert}}}}
\newcommand{\hor}{{\text{\upshape{hor}}}}

\newcommand{\MMM}{\ensuremath{\mathscr{M}}}
\newcommand{\XXX}{\ensuremath{\mathscr{X}}}

\newcommand{\YYY}{\ensuremath{\mathscr{Y}}}
\newcommand{\LLL}{\ensuremath{\mathscr{L}}}
\newcommand{\WWW}{\ensuremath{\mathscr{W}}}
\newcommand{\UUU}{\ensuremath{\mathscr{U}}}
\newcommand{\BBB}{\ensuremath{\mathscr{B}}}
\newcommand{\HHH}{\ensuremath{\mathscr{H}}}
\newcommand{\VVV}{\ensuremath{\mathscr{V}}}

\newcommand{\UUUU}{\ensuremath{\mathfrak{U}}}
\newcommand{\LLLL}{\ensuremath{\mathfrak{L}}}
\newcommand{\XXXX}{\ensuremath{\mathfrak{X}}}
\newcommand{\VVVV}{\ensuremath{\mathfrak{V}}}
\newcommand{\BBBB}{\ensuremath{\mathfrak{B}}}

\newcommand{\metr}{{\| \cdot \|}}

\newcommand{\Dcal}{{\mathscr D}}

\DeclareMathOperator{\coloneq}{\mathrel{\mathop:}=}

\newcommand{\euler}{\text{e}}
\DeclareMathOperator{\val}{val}
\DeclareMathOperator{\ord}{ord}
\DeclareMathOperator{\dist}{d}
\DeclareMathOperator{\dint}{d\!}
\DeclareMathOperator{\me}{\|\cdot\|}
\DeclareMathOperator{\red}{red}
\DeclareMathOperator{\h}{h}
\DeclareMathOperator{\c1}{c_1}
\DeclareMathOperator{\trop}{trop}
\DeclareMathOperator{\psim}{\psi_{\|\cdot\|}}
\DeclareMathOperator{\thetam}{\vartheta_{\|\cdot\|}}
\DeclareMathOperator{\vol}{vol}

\newcommand{\la}{\ensuremath{\left\langle}}
\newcommand{\ra}{\ensuremath{\right\rangle}}

\newcommand{\sigmad}{\ensuremath{\sigma^{\vee}}}
\newcommand{\SSigma}{\ensuremath{\widetilde{\Sigma}}}
\DeclareMathOperator{\cone}{c}
\DeclareMathOperator{\rec}{rec}
\DeclareMathOperator{\aff}{aff}
\DeclareMathOperator{\ri}{ri}

\DeclareMathOperator{\mult}{mult}

\DeclareMathOperator{\Val}{Val}

\newcommand{\YYYY}{\ensuremath{\mathfrak{Y}}}

\makeatletter
\newcommand{\pnrelbar}{%
  \linethickness{\dimen2}%
  \sbox\z@{$\m@th\prec$}%
  \dimen@=1.1\ht\z@
  \begin{picture}(\dimen@,.4ex)
  \roundcap
  \put(0,.2ex){\line(1,0){\dimen@}}
  \put(\dimexpr 0.5\dimen@-.2ex\relax,0){\line(1,1){.4ex}}
  \end{picture}%
}

\newcommand{\precneq}{\mathrel{\vcenter{\hbox{\text{\prec@neq}}}}}
\newcommand{\prec@neq}{%
  \dimen2=\f@size\dimexpr.04pt\relax
  \oalign{%
    \noalign{\kern\dimexpr.2ex-.5\dimen2\relax}
    $\m@th\prec$\cr
    \noalign{\kern-.5\dimen2}
    \hidewidth\pnrelbar\hidewidth\cr
  }%
}
\makeatother

\newcommand{\Mf}{\ensuremath{\mathfrak{M}}}
\newcommand{\GG}{\ensuremath{{\mathbb{G}}}}
\newcommand{\bfm}{{\boldsymbol{m}}}
\newcommand{\bft}{{\boldsymbol{t}}}
\DeclareMathOperator{\conv}{conv}
\newcommand{\EEE}{\ensuremath{\mathscr{E}}}
\newcommand{\Ean}{\ensuremath{E^{\text{\upshape{an}}}}}
\newcommand{\muH}{\ensuremath{{\mu_{\text{\upshape Haar}}}}}

\begin{document}
\title{Local Heights of Toric Varieties over Non-archimedean Fields}
\author{Walter Gubler and Julius Hertel}

\begin{abstract}
We generalize results about local heights previously proved in the case of discrete absolute values to arbitrary non-archimedean absolute values. First, this is done for the induction formula  of Chambert-Loir and Thuillier. Then we prove the formula of Burgos--Philippon--Sombra for the toric local height of a proper normal toric variety in this more general setting. We apply the corresponding formula for Moriwaki's global heights over a finitely generated field to a fibration which is generically toric. We illustrate the last result in a natural example where non-discrete non-archimedean absolute values really matter.

\vspace{2mm}

\end{abstract}

\maketitle

\tableofcontents

\section*{Introduction}
The height of an algebraic point of a proper variety $X$ over a number field $F$ measures the arithmetic complexity of its coordinates.
It is a tool
to control the number and distribution of these points which is essential for proving finiteness results
in diophantine geometry like the theorems of Mordell--Weil and Faltings
(see, for example, \cite{BG}). More generally, there is a height of (sub-)varieties which might be seen as an arithmetic analogue 
of the degree used in algebraic geometry. 
In \cite{Fal}, Faltings made this precise writing the height of $X$ with respect to a hermitian line bunde $\overline\LL$ as an arithmetic degree using the 
arithmetic intersection theory of Gillet--Soulé \cite{GS}.

In  the adelic language introduced by Zhang \cite{ZhaS}, a hermitian line bundle can be seen as a line bundle $L$ over $X$ endowed 
with a smooth metric at every archimedean place and with a metric induced by a global model of the line bundle for every non-archimedean place of $F$. This flexible point of view allows to consider more generally semipositive (continuous) metrics which are obtained from   uniform limits of semipositive hermitian metrics or even DSP metrics which are differences of these   semipositive continuous metrics. 
A remarkable application of Zhang's heights is his proof of
the Bogomolov conjecture for abelian varieties over a number field in \cite{ZhaB}.

Following Weil and N\'eron, it is more convenient to define the height as a sum of local heights. 
Here, ``local'' means that we consider the contribution of a fixed place  and work over the corresponding completion.
Local heights of subvarieties can be studied for any field with a given absolute value which was systematically done
in \cite{GuM}, \cite{GuLocal}, \cite{GuPisa}.
By base change, we may assume that our base field is an algebraically closed field $K$ endowed with a non-trivial complete absolute value. 
The local height $\lambda_{(\ML_0,s_0),\dotsc,(\ML_n,s_n)}(X)$ of the $n$-dimensional proper variety $X$ over $K$ with respect to  DSP-metrized line bundle $\ML_0, \dots, \ML_n$ depends also on the choice of non-zero meromorphic sections  $s_j$ of $L_j$ for $j=1,\dots,n$ and is a well-defined real number under the assumption that  
\begin{align}\label{emptysection}
\left|\dvs(s_0)\right|\cap \dotsb \cap \left|\dvs(s_n)\right| =\emptyset.
\end{align}
In \cite{Cha}, Chambert-Loir introduced
a measure $\c1(\ML_1) \wedge \dots \wedge \c1(\ML_n)$ on the analytification $\Xan$. It plays an important role for equidistribution theorems.  For details about the theory of local heights and Chambert-Loir measures, we refer to Section \ref{local heights chapter}. 



The main  result in Section \ref{local heights chapter} is the following induction formula
 which
generalizes a result of Chambert-Loir and Thuillier \cite[Théorème 4.1]{CT}.


\begin{thma}[Induction formula]\label{thm1}
Under the hypotheses above,  the function $\log\|s_{0}\|$ is integrable with respect to $\c1(\ML_1) \wedge \dots \wedge \c1(\ML_n)$ and
we have
\[
\lambda_{(\ML_0,s_0),\dotsc,(\ML_n,s_n)}(X)
=\lambda_{(\ML_1,s_1),\dotsc,(\ML_n,s_{n})}(\cyc(s_{0}))-\int_{\Xan}\!\log\|s_{0}\|\,\c1(\ML_1) \wedge \dots \wedge \c1(\ML_n)  \]
where $\cyc(s_0)$ is the Weil divisor associated to $s_0$. 
\end{thma}
In fact, we will show in Theorem \ref{ifDSP} a more general result involving pseudo-divisors. 
Chambert-Loir and Thuillier proved the formula under the additional assumptions that 
$K$ is a completion of a number field and $s_0,\dotsc,s_n$ are global sections such that their Cartier divisors intersect properly.
The heart of the proof is an approximation theorem saying that
$\log\|s_{0}\|$ can be approximated by suitable functions $\log\|1\|_n$, where $\|\cdot\|_n$ are formal metrics on the trivial bundle $\OO_X$. To show this over any non-archimedean field $K$, we use techniques from analytic and formal geometry.


In Section \ref{local heights toric}, we deal with local heights of toric varieties. 
Toric varieties are a special class of varieties that have a nice description through combinatorial
data from convex geometry.
So they are well-suited for testing conjectures and for computations in algebraic geometry. 
Let $K$ be any field, then a complete fan $\Sigma$ of strongly convex rational polyhedral cones in a vector space 
$\NR\simeq \RR^n$ corresponds to a proper toric variety $\XS$ over $K$ with torus
$\TT\simeq \spec K[x_1^{\pm 1},\dotsc, x_n^{\pm 1}]$.
The torus $\TT$ acts on $\XS$ and hence every toric object should have a certain invariance property with respect to this action in order to describe it in terms of convex geometry. We recall the classical theory of toric varieties in §\,\ref{Toric Varieties}.

A support function on $\Sigma$, i.\,e. a concave function $\Psi\colon \NR \rightarrow \RR$ which is linear on each cone of $\Sigma$
and has integral slopes, corresponds to a base-point-free toric line bundle $L$ on $\XS$ together with a toric section $s$.
Moreover, one can associate to $\Psi$ a polytope $\Delta_\Psi=\{m\in \MRR\mid m\geq \Psi \}$ in the dual space $\MRR$ of $\NR$.
Then a famous result in classical toric geometry is the degree formula:
\[\deg_L(\XS)=n!\vol_M(\Delta_\Psi),\]
where $\vol_M$ is the Haar measure on $\MRR$ such that the underlying lattice $M\simeq \ZZ^n$ has covolume one.
As mentioned above, the arithmetic analogue of the degree of a variety with respect to a line bundle
is the height of a variety.
Burgos, Philippon and Sombra  proved in the monograph \cite{BPS} a similar formula for the toric local height. In the non-archimedean case, they assume that the valuation is discrete. It is the goal of Section 2 to remove this hypothesis. 
The problem is that the valuation ring is not noetherian and hence the usual methods from algebraic geometry do not apply.

Let $K$ be a field endowed with a non-trivial non-archimedean complete absolute value and value group $\Gamma$ as a subgroup of $\RR$. Let $X_\Sigma$ be the proper toric variety over $K$ associated to the complete fan $\Sigma$ and let $L$ be a toric line bundle on  $\XS$  together with any toric section $s$. 
Let $\Psi$ be the corresponding support function and let $\Delta_\Psi$ be the dual polytope. 
In §\,\ref{Toric Schemes over Valuation Rings of Rank One}, we recall the theory of toric schemes over the valuation ring $K^\circ$ given in \cite{Guide}. In particular, a strongly convex $\Gamma$-rational polyhedral complex $\Pi$ induces a normal toric scheme $\XXX_\Pi$ over $K^\circ$. Assuming that the recession cones in $\Pi$ give the fan $\Sigma$, the toric scheme $\XXX_\Pi$ is a $K^\circ$-model of $X_\Sigma$.
Generalizing the program from \cite{BPS}, we   describe toric Cartier divisors on $\XXX_\Pi$ in terms of piecewise affine functions on $\Pi$ (see §\,\ref{Toric Cartier Divisors on Toric Schemes}). To describe local heights, we may additionally assume that $K$ is algebraically closed. 
A continuous metric $\|\cdot\|$ on $L$ is toric if the function $p\mapsto\|s(p)\|$ is invariant under the action of the formal torus in $\Tan$. 
We will give the following classification of toric metrics:
\begin{thma}\label{thm2.5}
Let $\Psi$ be a support function on $\Sigma$. Then there is a bijective correspondence between the sets of 
\begin{enumerate}
\item semipositive toric metrics on $L$;
\item concave functions $\psi$ on $\NR$ such that the function $|\psi-\Psi|$ is bounded;
\item continuous concave functions $\vartheta$ on $\Delta_\Psi$.   
\end{enumerate}
\end{thma}
For the first bijection, one associates to the toric metric $\|\cdot\|$ the function $\psi$ on $\NR$ given by $\psi(u)=\log\|s\circ\trop^{-1}(u)\|$, 
where $\trop\colon \NR\rightarrow \Tan$ is the tropicalization map from tropical geometry (see \ref{trop}).
The second bijection is given by the Legendre-Fenchel dual from convex analysis (see \ref{a5}).

This theorem was proven by Burgos, Philippon and Sombra in \cite{BPS} if the absolute value on $K$ is discrete or archimedean. 
We will prove our generalization in Theorem \ref{sem-con}.
Essential for the proof are characterizations of semipositive formal metrics developed in \cite{GuKue} and some new results for strictly semistable formal schemes shown in Appendix \ref{B}.
Note that the concave function $\psi=\Psi$ defines a canonical metric on $L$ which will be used later.

In Theorem \ref{trop-mong}, we will show that the measure $\c1(\ML)^{\wedge n}$, induced by a semipositive toric metrized line bundle $\ML=(L,\|\cdot\|)$ on the $n$-dimensional proper toric variety $\XS$,
satisfies the following formula
\[\trop_*\bigl(\c1(\ML)^{\wedge n}\bigr)=n!\M_M(\psi),\] 
where $\psi$ is the concave function given by $\|\cdot\|$ and $\M_M(\psi)$ is the Monge-Ampère measure of $\psi$ (see \ref{a8}).

Now all ingredients from the program in \cite{BPS} are generalized  to  show the formula for the local toric height. 
Let $\XS$ be an $n$-dimensional projective toric variety over $K$ and $\ML$ a semipositive toric metrized line bundle. Denote by $\MLcan$ the same line bundle
equipped with the canonical metric.
The toric local height of $\XS$ with respect to $\ML$ is defined as
\[\lambdator_{\ML}(\XS)=\lambda_{(\ML,s_0),\dotsc,(\ML,s_n)}(\XS)-\lambda_{(\MLcan\!,s_0),\dotsc,(\MLcan\!,s_n)}(\XS),\]
where $s_0,\dotsc,s_n$ are any invertible meromorphic sections of $L$ satisfying the intersection condition \eqref{emptysection}.
\begin{thma} \label{thm3}
Using the above notation, we have
\[\lambdator_{\ML}(\XS)=(n+1)!\int_{\Delta_\Psi}\vartheta \dint\vol_M ,\]
where $\vartheta\colon \Delta_\Psi\rightarrow \RR$ is the concave function associated to $(\ML,s)$ given by Theorem \ref{thm2.5}.
\end{thma}
A slightly more general version for a proper toric variety $\XS$ will be shown in Theorem \ref{thetheorem}. 
The proof is analogous to \cite{BPS}. It is based on induction relative to $n$ and uses the induction formula (Theorem \ref{thm1}) in an essential way.

In Section \ref{global heights chapter}, we return to global heights. In \cite{GuM},  the notion of an $M$-field was introduced to capture all situations where global heights occur. This is a field
$K$ together with a measure space $(M, \mu)$ of absolute values on $K$. 
The standard examples are number fields and function fields, but the notion of $M$-field may be also used to study the characteristic function in Nevanlinna theory which is an analogue of a global height according to Osgood and Vojta. 

Let us consider a projective $n$-dimensional variety $X$ over an $M$-field $K$ satisfying the product formula  
and a line bundle $L$ on $X$. For $v \in M$, we write $X_v$ and $L_v$ for the base change to the completion $\CC_v$ of the algebraic closure of $K$. 
A DSP $M$-metric on $L$ is a family of DSP metrics $\|\cdot\|_{v}$ on $L_{v}$, $v\in M$.
Write $\ML=(L,(\|\cdot\|_{v})_v)$ and $\ML_{v}=(L_{v},\|\cdot\|_{v})$ for each $v\in M$. 
We consider now DSP $M$-metrized line bundles $\ML_0, \dots, \ML_n$ on $X$ assuming
 that the local heights 
\[M\longrightarrow \RR,\quad v\longmapsto\lambda_{(\ML_{0,v},s_0),\dotsc,(\ML_{n,v},s_n)}(X_v)\]
are $\mu$-integrable for any choice of non-zero meromorphic sections $s_0,\dotsc,s_n$ of $L$ which satisfy condition
\eqref{emptysection}. 
Then the (global) height of $X$ is defined as
\begin{align*}
\h_{\ML_0, \dots, \ML_n}(X)=\int_M \lambda_{(\ML_{0,v},s_0),\dotsc,(\ML_{n,v},s_n)}(X_v)\dint\mu(v).
\end{align*}
By the product formula, this definition is independent of the choice of sections.  Using  integration over all places, we give in Theorem \ref{ifMglobal}  a global version of the induction formula. In §\,\ref{global height over M}, the theory of global height is presented more generally for proper varieties over an $M$-field. In the classical case of number fields or function fields, the above integrability condition is always satisfied if the metrized line bundles are quasi-algebraic. The latter means for $\ML$ that up to finitely many places $v \in M$, the metrics $\metr_v$ are induced by a global model of $L$. In Theorem \ref{ifMglobal}, we give the induction formula for global heights.

In \cite{Mor}, Moriwaki defined the global height of a variety over a finitely generated field $K$ over $\QQ$
as an arithmetic intersection number  
and generalized the Bogomolov conjecture to such fields. 
As observed in \cite[Example 11.22]{GuPisa}, this finitely generated extension has a
$\Mf$-field structure
for a natural set of places $\Mf$ related to the normal variety $B$ with $K=\QQ(B)$. 
Burgos--Philippon--Sombra proved in \cite[Theorem 2.4]{BPS3} that the height of Moriwaki can be written as 
an integral of local heights over $\Mf$.
In Section \ref{global heights chapter}, we will generalize their result as follows.

Let $B$ be a $b$-dimensional normal projective variety over a global field $F$. 
We denote by $K$ the function field of $B$ which is a finitely generated extension of $F$. We define 
$$\Mf \coloneq \Bone \sqcup \bigsqcup_{v\in M_F}{\Bgen_v},$$
where $\Bone$ is the set of discrete absolute values corresponding to the orders in the prime divisors of $B$ and where $\Bgen_v$ is the set of absolute values induced by evaluating at generic points of the analytification of $B$ with respect to the place $v$ of $F$. Note that the elements in $\Bgen_v$ may lead to non-discrete valuations. 

Choosing  quasi-algebraic metrized line bundles $\MH_1,\dotsc, \MH_b$ on $B$, we can equip
$K$ with a natural structure $(\Mf,\mu)$ of an $\Mf$-field  satisfying the product formula. 
Here, we assume that the line bundles $\MH_j$ are all nef which means that all the metrics are semipositive and that the height of every algebraic point of $B$ is non-negative. The measure $\mu$ on $\Mf$ is given by the counting measure on $\Bone$ and by $\mu(v)\c1(H_1, \metr_v) \wedge \dots \wedge \c1(H_b,\metr_v)$ on $\Bgen_v$ where $\mu(v)$ is the weight of the product formula for $F$ in $v$. For more details, we refer to §\,\ref{section3.2}.

Let $\pi\colon \XX\rightarrow B$ be a dominant morphism of projective varieties over $F$ of relative dimension $n$ and denote by $X$ the generic fiber of $\pi$.
Let $\overline{\LL}_0,\dotsc,\overline{\LL}_n$ be semipositive quasi-algebraic line bundles
on $\XX$
and choose any invertible meromorphic sections $s_0,\dotsc, s_n$ of $\LL_0,\dotsc, \LL_n$ respectively,
which satisfy \eqref{emptysection}. These line bundles induce $\Mf$-metrized line bundles $\ML_0,\dotsc, \ML_n$ on $X$.
We prove in Theorem \ref{THM}:
\begin{thma}\label{thm2}
The function
$\Mf\rightarrow \RR$, $w\mapsto \lambda_{(\ML_{0,w},s_0),\dotsc,(\ML_{n,w},s_n)}(X)$,
is $\mu$-integrable and we have
\[\h_{\pi^*\MH_1,\dotsc,\pi^*\MH_b,\overline{\LL}_0,\dotsc,\overline{\LL}_n}(\XX)
=\int_{\Mf}\lambda_{(\ML_{0,w},s_0),\dotsc,(\ML_{n,w},s_n)}(X) \dint\mu(w).\]
\end{thma}
We will prove a more general version of this result in Theorem \ref{THM} where we also allow proper varieties.  
Burgos--Philippon--Sombra have shown this formula in the case when $F=\QQ$ and
the varieties $\XX$, $B$ and the occurring metrized line bundles are induced by models. Then the measure $\mu$ 
has support in the subset of $\Mf$ given by the archimedean and discrete absolute values. 
The main difficulty in their proof appears at the archimedean place, where well-known techniques
from analysis as Ehresmann's fibration theorem are used.
In our proof, the archimedean part and the part on $\Bone$ follow from their arguments, but the contribution  of   $\Bgen_v$ for non-archimedean $v$ is much more complicated as the support of $\mu$ can also contain non-discrete absolute values.

In §\,\ref{GlobalToric}, we will give the following application of the formula in Theorem \ref{thm3}. This was suggested to us by José Burgos Gil.
In the setting of Theorem \ref{thm2}, let $\pi\colon \XX\rightarrow B$ be a dominant morphism of projective varieties over a global field $F$
such that its generic fiber $X$ is an $n$-dimensional normal toric variety over the function field $K=F(B)$. This field is equipped with the $\Mf$-field structure induced by the metrized line bundles $\MH_1,\dotsc,\MH_b$.
Assume that $\overline{\LL_0}=\dotsb=\overline{\LL_n}=\overline{\LL}$ and that the
induced semipositive $\Mf$-metrized line bundle $\ML$ is toric.
Let $s$ be any toric section of $L$ and $\Psi$ the associated support function.
Then $\ML$ defines, for each $w\in \Mf$, 
a concave function $\vartheta_w\colon \Delta_\Psi\rightarrow \RR$.
Combining theorems \ref{thm3} and \ref{thm2}, 
we will obtain in Corollary \ref{theta-prop} the formula 
\begin{align*}
\h_{\pi^*\MH_1,\dotsc,\pi^*\MH_b,\overline{\LL},\dotsc,\overline{\LL}}(\XX)=(n+1)!\int_{\Mf}\int_{\Delta_\Psi}\vartheta_w(x)\dint\vol(x)\dint\mu(w).
\end{align*}
By means of this formula we can compute the height of a non-toric variety
coming from a fibration with toric generic fiber. 
It generalizes Corollary 3.1 in \cite{BPS3} where the global field is $\QQ$ and the metrized line bundles are induced by models over $\ZZ$ and hence 
only archimedean and discrete non-archimedean places occur. 

In §\,\ref{application example}, we will illustrate the formula in a special case with $B$ an elliptic curve over the global field $F$. Then the canonical metric on  an ample line bundle $H=H_1$ of $B$  leads to a natural example where non-discrete non-archimedean absolute values really matter.

\subsection*{Acknowledgements}

This is an extended version of the second author's thesis. 
We are very grateful to José Burgos Gil for suggesting us the application in §\,\ref{GlobalToric} and for his comments. We thank Jascha Smacka for proofreading and the referee for his helpful comments. 
This research was  supported by the DFG grant: SFB 1085 ``Higher invariants''.

\subsection*{Terminology}

For  the inclusion $A \subset B$ of sets, $A$ may be equal to $B$. The complement is denoted by $B \setminus A$. A disjoint union is denoted by $A \sqcup B$. 
A measure is a signed measure, i.\,e.\ it is not necessarily a positive measure. 

The set $\NN$ of natural numbers contains zero.
All occuring rings and algebras are commutative with unity. For a ring $R$, the group of units is denoted
by $R^\times$. 

A variety\index{variety} over a field $k$ is an irreducible and reduced scheme which is separated and of finite type over $k$.
The function field of a variety $X$ over $k$ is denoted by $k(X)$.
For a proper scheme $Y$ over a field, we denote by $Y^{(n)}$ the set of subvarieties of codimension $n$.
A prime cycle\index{cycle!prime} on $Y$ is just a closed subvariety of $Y$.

By a line bundle\index{line bundle} we mean a locally free sheaf of rank one. For an invertible meromorphic section $s$ 
of a line bundle,
we denote by $\dvs(s)$ the associated Cartier divisor and
by $\cyc(s)$ the associated Weil divisor. The support of $\dvs(s)$ is denoted by $\left|\dvs(s)\right|$.

A non-archimedean field is a field $K$ which is complete with respect to a non-trivial non-archimedean absolute value $|\cdot|$.
Its valuation ring is denoted by $K^\circ$  with valuation $\val\coloneq -\log |\cdot |$, value group $\Gamma \coloneq \val(K^\times)$ and  residue field  $\tilde{K}\coloneq \Kval/\Kmax$, where $\Kmax$ is the maximal ideal of $\Kval$. 

For the notations used from convex geometry, we refer to Appendix \ref{A}.



\section{Local heights} \label{local heights chapter}

In this section, we recall foundational notions and results for this work.  In the first subsection, we collect results about Berkovich spaces and admissible formal schemes in the sense of Raynaud. In the next subsection, we will introduce formal and algebraic models of proper algebraic varieties and their line bundles. The associated formal and algebraic metrics in the sense of Zhang lead to local heights and Chambert--Loir measures. This is generalized in §\,\ref{section: semipositivity} to semipositive continuous metrics on line bundles. The new results in this section are in §\,\ref{Induction formula for DSP local heights} where the induction formula for local heights of Chambert-Loir and Thuillier is generalized to arbitrary non-archimedean fields. 

\subsection{Analytic and formal geometry}

Let $K$ be a \emph{non-archimedean field}, i.\,e. a field which is complete with respect to a non-trivial non-archimedean absolute value $|\cdot|$.
Its valuation ring is denoted by $\Kval$, the associated maximal ideal by $\Kmax$ and the residue field by $\tilde{K}=\Kval/\Kmax$.

In this subsection we recall some facts about the (Berkovich) analytification of a scheme $X$ of finite type over $K$ and of an admissible formal $\Kval$-scheme $\XXXX$ in the sense of Raynaud.

\begin{art}
The \emph{Tate algebra} $\tate$ consists of the formal power series $f=\sum_\nu{a_\nu x^{\nu}}$ in $K[[x_1,\dotsc,x_n]]$ such that
$|a_\nu|\rightarrow 0$ as $|\nu|\rightarrow \infty$. This $K$-algebra is the completion of $K[x_1,\dotsc,x_n]$ with respect to the Gauß norm
$\|f\|=\max_\nu{|a_\nu|}$.

A \emph{$K$-affinoid algebra} is a $K$-algebra $\A$ which is isomorphic to $\tate/I$ for an ideal $I$. We may use the quotient norm from $\tate$ to
define a $K$-Banach algebra $(\A,\|\cdot\|)$. The presentation and hence the induced norm of an affinoid algebra is not unique but two 
norms on $\A$ are equivalent and thus they define the same concept of boundedness.
\end{art}

\begin{art}
The \emph{Berkovich spectrum} $\MMM(\A)$ of a $K$-affinoid algebra $\A$ is defined as the set of multiplicative seminorms $p$ on $\A$ satisfying $p(f)\leq \|f\|$
for all $f\in \A$. 
It only depends on the algebraic structure on $\A$.
As above we endow it with the coarsest topology such that the maps $p\mapsto p(f)$ are continuous for all $f\in \A$ and we obtain a nonempty compact space.
\end{art}

Roughly speaking, a Berkovich analytic space is given by an atlas of affinoid Berkovich spectra. For the precise definition, we refer to \cite{Ber2}. Note that we here only consider  analytic spaces which are called strict in \cite{Ber2}. We need mainly the following construction related to algebraic schemes:

\begin{art} 
First let $X=\spec(A)$ be affine. The \emph{analytification} $\Xan$  is  
the set of multiplicative seminorms on $A$ extending the absolute value $|\cdot|$ on $K$. 
We endow it with the coarsest topology such that the functions
$ \Xan\rightarrow \RR$, $p\mapsto p(f)$
are continuous for every $f\in A$. The sheaf of analytic functions $\OO_{\Xan}$ on $\Xan$ gives $\Xan$ the structure of  a Berkovich analytic space (see \cite[\S 3.4]{Ber}	and \cite[§\,1.2]{BPS}).	
\end{art}

\begin{art}\label{BerkAn}
For any scheme $X$ of finite type over $K$ we define the \emph{analytification} $\Xan$ as a Berkovich analytic space by gluing the affine analytic spaces obtained from an open affine cover of $X$.
For a morphism $\varphi:X\rightarrow Y$ of schemes of finite type over $K$ we have a canonical map $\varphi^{\an}:\Xan\rightarrow Y^{\an}$
defined by $\varphi^{\an}(p)\coloneq p\circ \varphi^{\sharp}$ on suitable affine open subsets.

The analytification functor preserves many properties of schemes and their morphisms.
An analytic space $\Xan$ is Hausdorff (resp. compact) if and only if $X$ is separated (resp. proper). For more details, we refer to \cite[§\,3.4]{Ber}.
\end{art}

The analytification of a formal scheme is more difficult because at first we need arbitrary analytic spaces.

\begin{art}\label{red}
We call a $\Kval$-algebra $A$ \emph{admissible} if it is isomorphic to $\Kval\left\langle x_1,\dotsc,x_n\right\rangle/I$ for an ideal $I$ and $A$ has no 
$\Kval$-torsion (or equivalently $A$ is $\Kval$-flat).
If $A$ is admissible, then $I$ is finitely generated.
A formal scheme $\XXXX$ over $\Kval$ is called  \emph{admissible} if there is a locally finite covering of open subsets isomorphic to formal affine schemes $\spf(A)$ for admissible
$\Kval$-algebras $A$.

Then the \emph{generic fiber} $\XXXX^{\an}$ of $\XXXX$ is the analytic space locally defined by the Berkovich spectrum
of the $K$-affinoid algebra $\A=A\otimes_{\Kval}K$.
Moreover we define the \emph{special fiber} $\widetilde{\XXXX}$ of $\XXXX$ as the $\tilde{K}$-scheme locally given by $\spec(A/\Kmax A)$, i.\,e.
$\widetilde{\XXXX}$ is a scheme of locally finite type over $\tilde{K}$ with the same topological space as $\XXXX$ and the structure sheaf
$\OO_{\widetilde{\XXXX}}\coloneq\OO_{\XXXX}\otimes_{\Kval}\tilde{K}$.

There is a \emph{reduction map} $\red:\XXXX^{\an}\rightarrow \widetilde{\XXXX}$
assigning each seminorm $p$ in a neighborhood $\MMM(A\otimes_{\Kval}K)$ to
the prime ideal $\{a\in A \,|\,p(a\otimes 1)< 1\}/\Kmax A$. This map is surjective and anti-continuous.
If $\widetilde{\XXXX}$ is reduced, then
$\red$ coincides with the reduction map in \cite[2.4]{Ber}.
In this case, for every irreducible component $V$ of $\widetilde{\XXXX}$, there is a unique point $\xi_V\in\XXXX^{\an}$
such that $\red(\xi_V)$ is the generic point of $V$ (see \cite[Proposition 2.4.4]{Ber}).
\end{art}



\begin{art}\label{redspecial}
Assume that $K$ is algebraically closed and let $\XXXX=\spf(A)$ 
be an admissible formal affine scheme over $\Kval$
with reduced generic fiber $\XXXX^{\an}$, but not necessarily with reduced special fiber.
Let $\A=A\otimes_{\Kval}K$ be the associated $K$-affinoid algebra and let $\A^\circ$
be the $\Kval$-subalgebra of power bounded elements in $\A$.
Then $\XXXX'\coloneq \spf(\A^\circ)$ is an admissible formal scheme over $\Kval$ with 
$\XXXX'^{\an}=\XXXX^{\an}$ and with reduced special fiber $\widetilde{\XXXX'}$.
The identity on the generic fiber extends to a canonical morphism $\XXXX'\rightarrow \XXXX$ whose restriction
to the special fibers is finite and surjective.
By gluing, these assertions also hold for non-necessarily affine formal schemes.
For details, we refer to \cite[Proposition 1.11 and 8.1]{GuLocal}.
\end{art}

\begin{art}\label{f-alg}Let $\XXX$ be a flat scheme of finite type over $\Kval$ with generic fiber $X$ and $\pi$ some non-zero element in $\Kmax$.
Locally we can replace the coordinate ring $A$ by the $\pi$-adic completion of $A$
and get an admissible formal scheme $\hat{\XXX}$ over $\Kval$ with special fiber equal to the special fiber $\widetilde{\XXX}$ of $\XXX$.
The generic fiber $\hat{\XXX}^{\text{an}}$, denoted by $X^\circ$, is an analytic subdomain of $\Xan$  locally given by
\[\left\{p\in (\spec A\otimes_{\Kval}K)^{\an}\mid p(a)\leq 1\ \forall \,a \in A \right\}.\]
If $\XXX$ is proper over $\Kval$, then $X^\circ=\Xan$ and the reduction map is defined on the whole of $\Xan$.
If $\widetilde{\XXX}$ is reduced, then each maximal point of $\widetilde{\XXX}$ has a unique inverse image in $X^\circ$.
We refer to \cite[4.9--4.13]{Guide} for details.

If $K$ is algebraically closed and $X$ is reduced, then the construction in \ref{redspecial}
gives us a formal admissible scheme $\XXXX$ over $\Kval$ with 
generic fiber $\XXXX^{\an}=X^\circ$
and with
reduced special fiber $\widetilde{\XXXX}$
such that 
the canonical morphism $\widetilde{\XXXX}\rightarrow \widetilde{\XXX}$ is finite and surjective.
\end{art}

\subsection{Metrics, local heights and measures}\label{localsec}

From now on, we assume that the non-archimedean field $K$ is algebraically closed. This is no serious restriction because
we can always perform base change to the completion of the algebraic closure of any non-archimedean field and
local heights and measures do not depend on the choice of the base field.

Let $X$ be a reduced proper scheme over $K$ and $L$ a line bundle on $X$. This defines a line bundle $\Lan$ on the compact space $\Xan$. 

In this subsection, we introduce algebraic (resp.\ formal) models of $X$ and $L$ and their associated algebraic (resp.\ formal) metrics on $L^{\an}$. After introducing metrized pseudo-divisors, we can study  local heights of subvarieties and Chambert-Loir measures. 

\begin{defn}\label{metric}
A \emph{metric}\index{metric} $\me$ on $L$ is the datum, for any section $s$ of $\Lan$ on a open subset $U\subseteq \Xan$, 
of a continuous function $\|s(\cdot)\|\colon U\rightarrow \RR_{\geq 0}$ such that
\begin{enumerate}
	\item it is compatible with the restriction to smaller open subsets;
	\item for all $p\in U$, $\|s(p)\|=0$ if and only if $s(p)=0$;
	\item for any $\lambda \in\OO_{\Xan}(U)$ and for all $p\in U$, $\|(\lambda s)(p)\|=|\lambda(p) |\cdot\|s(p)\|$.
\end{enumerate}
On the set of metrics on $L$ we define the distance function
\[\dist(\me,\me')\coloneq \sup_{p\in \Xan}\left|\log(\|s_p(p)\|/\|s_p(p)\|')\right|,\]
where $s_p$ is any local section of $\Lan$ not vanishing at $p$. Clearly, this definition is independent of the choices of the $s_p$.
The pair $\overline{L}\coloneq (L,\me)$ is called a \emph{metrized line bundle}\index{metrized line bundle}.
Operations on line bundles like tensor product, dual and pullback extend to metrized line bundles.
\end{defn}

\begin{defn}\label{formal}
A \emph{formal ($\Kval$-)model}\index{model!formal!of a variety} of $X$ is an admissible formal scheme $\XXXX$ over $\Kval$ with a fixed isomorphism $\XXXX^{\an}\simeq\Xan$.
Note that we identify $\XXXX^{\an}$ with $\Xan$ via this isomorphism.

A \emph{formal ($\Kval$-)model}\index{model!formal!of a line bundle} of $(X,L)$ is a triple $(\XXXX,\LLLL,e)$ consisting of a formal model $\XXXX$ of $X$, a line 
bundle $\LLLL$ on $\XXXX$ and an integer $e\geq 1$, together with an isomorphism $\LLLL^{\an}\simeq (L^{\otimes e})^{\an}$.
When $e=1$, we write $(\XXXX,\LLLL)$ instead of $(\XXXX,\LLLL,1)$.
\end{defn}

\begin{defn}\label{formalmetric}
To a formal $\Kval$-model $(\XXXX,\LLLL,e)$ of $(X,L)$ we associate a metric $\|\cdot\|$ on $L$ in the following way:
If $\mathfrak{U}$ is a formal trivialization of $\LLLL$ and if $s$ is a section of $\Lan$ on $\mathfrak{U}^{\an}$ 
such that $s^{\otimes e}$ 
corresponds to $\lambda \in \OO_{\Xan}(\mathfrak{U}^{\an})$ with respect
to this trivialization,
then \[\|s(p)\|=|\lambda(p)|^{1/e}\] for all $p\in\mathfrak{U}^{\an}$.
A metric on $L$ obtained in this way is called a \emph{$\QQ$-formal metric}\index{metric!$\QQ$-formal} and, if $e=1$, a \emph{formal metric}\index{metric!formal}.

Such a $\QQ$-formal metric is said to be \emph{semipositive}\index{metric!semipositive $\QQ$-formal} if the reduction $\widetilde{\LLLL}$ of $\LLLL$ on the special fiber $\widetilde{\XXXX}$ is nef, i.\,e.
we have $\deg_{\widetilde{\LLLL}}(C)\geq 0$ for every closed integral curve $C$ in $\widetilde{\XXXX}$.
\end{defn}


\begin{art}\label{formsemip}

The dual, the tensor product and the pullback of ($\QQ$-)formal metrics are again ($\QQ$-)formal metrics.
Furthermore, the tensor product and the pullback of semipositive $\QQ$-formal metrics are semipositive.
%
\end{art}

\begin{art}\label{propformal}
Every line bundle $L$ on $X$ has a formal $\Kval$-model $(\XXXX,\LLLL)$ and hence a formal metric $\|\cdot\|$. 
For proofs of this and the following statements we refer to \cite[§\,7]{GuLocal}.
Since $K$ is algebraically closed and $X$ is reduced, we may always assume that $\XXXX$ has reduced special fiber (see \ref{redspecial}). 
Then the formal metric determines the $\Kval$-model $\LLLL$ on $\XXXX$
up to isomorphisms, more precisely we have canonically
\begin{align}\label{formallb}
\LLLL(\mathfrak{U})\cong \left\{s\in \Lan(\UUUU^{\an})\mid \|s(p)\|\leq 1\ \forall p\in \UUUU^{\an}\right\}
\end{align}
for each formal open subset $\UUUU$ of $\XXXX$.
\end{art}


\begin{defn}\label{algmod2}
An \emph{algebraic $\Kval$-model}\index{model!algebraic!of a variety} $\XXX$ of $X$ is a flat and proper scheme over $\Kval$ together with an isomorphism of the generic fiber of $\XXX$ onto $X$.
An \emph{algebraic $\Kval$-model}\index{model!algebraic!of a line bundle} $(\XXX,\LLL,e)$ of $(X,L)$ consists of a line bundle $\LLL$ on an algebraic $\Kval$-model $\XXX$ of $X$ and a
fixed isomorphism $\LLL|_X\cong L^e$.

As in Definition \ref{formalmetric}, an algebraic model $(\XXX,\LLL,e)$ of $(X,L)$ induces a metric $\|\cdot\|$ on $L$, called \emph{algebraic metric}\index{metric!algebraic}.
Such a metric is said to be \emph{semipositive}\index{metric!semipositive algebraic} if, for every closed integral curve $C$ in the special fiber $\widetilde{\XXX}$, we have $\deg_{\LLL}(C)\geq 0$.
\end{defn}

The following  result of Gubler and K\"unnemann \cite[Corollary 5.12]{GuKue} shows that, on algebraic varieties, it is always possible to work with 
algebraic in place of $\QQ$-formal metrics.

\begin{prop}\label{algmod}
Let $L$ be a line bundle on a proper variety $X$ over $K$ and let $\me$ be a metric on $L$.
Then, $\me$ is $\QQ$-formal if and only if $\me$ is  algebraic.
\end{prop}

\begin{art}\label{pseudo}
A \emph{metrized pseudo-divisor}\index{metrized pseudo-divisor} $\D$ on $X$ is a triple $\D\coloneq (\overline{L}, Y, s)$ where $\overline{L}$ is a metrized line bundle, $Y$ is a closed subset of $X$
and $s$ is a nowhere vanishing section of $L$ on $X\setminus Y$. Then $(\OO(D),|D|,s_D)\coloneq (L,Y,s)$ is a pseudo-divisor in the sense of \cite[2.2]{Fu}.
In contrast to Cartier divisors, we can always define the pullback of a metrized pseudo-divisor $\D$ on $X$ by a proper morphism $\varphi\colon X'\rightarrow X$, namely
\[\varphi^*\D\coloneq(\varphi^*\overline{\OO}(D),\varphi^{-1}|D|,\varphi^*s_D).\]
\end{art}

\begin{ex}\label{exdvs}
Let $\ML$ be a metrized line bundle on $X$ and $s$ an \emph{invertible me\-ro\-morphic section}\index{meromorphic section!invertible} of $L$, i.\,e.
there is an open dense subset $U$ of $X$ such that $s$ restricts to a non-vanishing section of $L$ on $U$.
Then the pair $(\ML,s)$
determines a pseudo-divisor \[\widehat{\dvs}(s)\coloneq \left(\ML,|\dvs(s)|, s|_{X\setminus|\dvs(s)| }\right),\]
where $|\dvs(s)|$ is the support of the Cartier divisor $\dvs(s)$.

Every real-valued continuous function $\varphi$ on $\Xan$ defines a metric on the trivial line bundle $\OO_X$ given by $\|1\|=e^{-\varphi}$. We denote this metrized line bundle by $\OO(\varphi)$.
Then we get a metrized pseudo-divisor $\widehat{\OO}(\varphi)\coloneq (\OO(\varphi),\emptyset,1)$.
\end{ex}

\begin{art}\label{localdef}
Let $\D_0,\dotsc,\D_t$ be metrized pseudo-divisors with $\QQ$-formal metrics and let $Z$ be a $t$-dimensional cycle on $X$ with
\begin{align}\label{emptysect}
|D_0|\cap\dotsb\cap|D_t|\cap |Z|=\emptyset.
\end{align}
Note that condition \eqref{emptysect} is much weaker than the usual assumption that $\D_0,\dotsc,\D_t$ \emph{intersect properly}\index{proper intersection} on $Z$,
that is, for all $I\subseteq\{0,\dotsc,t\}$, each irreducible component of $Z\cap \bigcap_{i\in I}|D_i|$ has dimension $t-|I|$.

For $\QQ$-formal metrized pseudo-divisors there is a refined intersection product with cycles on $X$ developed in \cite[§8]{GuLocal} and \cite[§5]{GuPisa}.
By means of this product, one can define the \emph{local height}\index{height!local} $\lambda_{\D_0,\dotsc,\D_t}(Z)$ as the real intersection number of $\D_0,\dotsc,\D_t$ and $Z$ on a joint
formal $\Kval$-model. For details, we refer to \cite[§9]{GuLocal} and \cite[§9]{GuPisa}.
If $\Kval$ is a discrete valuation ring  with value group $\Gamma = \ZZ$ and all the $\Kval$-models are algebraic, then we can use the usual intersection product.
\end{art}
\begin{prop}\label{proplocal}
The local height $\lambda(Z)\coloneq \lambda_{\D_0,\dotsc,\D_t}(Z)$ is characterized by the following properties:
\begin{enumerate}
	\item It is multilinear and symmetric in $\D_0,\dotsc,\D_t$ and linear in $Z$.
	\item For a proper morphism $\varphi\colon X'\rightarrow X$ and a $t$-dimensional cycle $Z'$ on $X'$ satisfying
	$|D_0|\cap\dotsb\cap|D_t|\cap |\varphi(Z')|=\emptyset$, we have
	\[\lambda_{\varphi^*\D_0,\dotsc,\varphi^*\D_t}(Z')=\lambda_{\D_0,\ldots,\D_t}(\varphi_*Z').\]
	\item Let $\lambda'(Z)$ be the local height obtained by replacing the metric $\|\cdot\|$ of $\D_0$ by another $\QQ$-formal metric $\|\cdot\|'$.
	If the $\QQ$-formal metrics of $\D_1,\dotsc,\D_t$ are semipositive and if $Z$ is effective, then
	\begin{align}\label{continuity}
	|\lambda(Z)-\lambda'(Z)|\leq\dist(\|\cdot\|,\|\cdot\|')\cdot\deg_{\OO(D_1),\dotsc,\OO(D_t)}(Z).
	\end{align}
\end{enumerate}
\end{prop}
\begin{proof}
The properties (i) and (ii) follow from \cite[Theorem 10.6]{GuPisa} and 
the last property follows from 
the metric change formula in
\cite[Remark 9.5]{GuPisa}.
\end{proof}

\begin{art}\label{commonmod}
If $\XXX$ is an algebraic $\Kval$-model of $X$, then there is a $\Kval$-model $\YYY$ of $X$ with reduced special fiber
and a proper $\Kval$-morphism $\YYY\rightarrow \XXX$ which is the identity on $X$.
This follows from \cite[Theorem 2.1']{BLR}.

Moreover, let $L$, $L'$ be algebraic metrized line bundles on $X$ induced by algebraic $\Kval$-models
$(\XXX,\LLL,e)$ and $(\XXX',\LLL',e')$ respectively.
Taking the closure $\XXX''$ of $X$ in $\XXX\times_{\Kval}\XXX'$ and pulling back $\LLL$, $\LLL'$ to $\XXX''$,
we obtain models inducing the same metrics on $L$ and $L'$ but living on the same model $\XXX''$.

Hence, we can always assume that $\LLL$ and $\LLL'$ live on a common model with reduced special fiber.
The same holds for formal models, see \ref{redspecial}.

%
\end{art}

For global heights and archimedean local heights of subvarieties there is an induction formula which can be taken as definition for the heights
(see \cite[(3.2.2)]{BGS} and \cite[Proposition 3.5]{GuPisa}).
A. Chambert-Loir has introduced a measure on $\Xan$ such that an analogous induction formula holds for non-archimedean local heights (cf. \cite[2.3]{Cha}).
\begin{defn}\label{measure}
Let $\ML_i$, $i=1,\dotsc,d$ be $\QQ$-formal metrized line bundles on the reduced proper scheme $X$ over $K$ of dimension $d$.
By \ref{commonmod}, there is a formal $\Kval$-model $\XXXX$ of $X$ 
with reduced special fiber 
and, for each $i$, a formal $\Kval$-model $(\XXXX,\LLLL_i,e_i)$ of $(X,L_i)$
inducing the metric of $\ML_i$.
We denote by $\widetilde{\XXXX}^{(0)}$ the set of irreducible components of the special fiber $\widetilde{\XXXX}$.
Then we define a discrete (signed) measure on $\Xan$ by
\[\c1(\ML_1)\wedge \dotsb\wedge \c1(\ML_d)=\frac{1}{e_1\dotsb e_d}\sum_{V\in\widetilde{\XXXX}^{(0)}}{\deg_{\widetilde{\LLLL}_1,\dotsc,\widetilde{\LLLL}_d}(V)\cdot\delta_{\xi_V}},\]
where 
$\delta_{\xi_V}$ is the Dirac measure in the unique
point $\xi_V\in \Xan$ such that $\red(\xi_V)$ is the generic point of $V$ (see \ref{red}).

More generally, let $Y$ be a $t$-dimensional subvariety of $X$, then we define
\[\c1(\ML_1)\wedge \dotsb\wedge \c1(\ML_t)\wedge\delta_Y=i_*\left(\c1(\ML_1|_Y)\wedge \dotsb\wedge \c1(\ML_t|_Y)\right),\]
where $i\colon\! Y^{\an}\rightarrow \Xan$ is the induced immersion. 
We also write $\c1(\ML_1)\dotsb \c1(\ML_t)\delta_Y$.
This measure extends by linearity to $t$-dimensional cycles.
\end{defn}

\begin{art}
This measure is multilinear and symmetric in metrized line bundles.
Moreover, the total mass of $\c1(\ML_1)\wedge \dotsb\wedge \c1(\ML_t)\wedge\delta_Y$ equals the degree $\deg_{L_1,\dotsc,L_t}(Y)$,
and it is a positive measure
if the metrics of the $\ML_i$ are semipositive.
\end{art}

\begin{prop}[Induction formula]\label{ifformal}
Let $\D_0,\dotsc,\D_t$ be $\QQ$-formal metrized pseudo-divisors and let $Z$ be a $t$-dimensional prime cycle with $|D_0|\cap\dotsb\cap|D_t|\cap |Z|=\emptyset$.
If $|Z|\nsubseteq|D_t|$, then let $s_{D_t,Z}\coloneq s_{D_t}|_Z$, otherwise we choose any non-zero meromorphic section $s_{D_t,Z}$ of $\OO(D_t)|_Z$.
Let $Y$ be the Weil divisor of $s_{D_t,Z}$ considered as a cycle on $X$.
Then 
we have
\begin{align*}
\lambda_{\hat{D}_0,\dotsc,\hat{D}_t}(Z)
\, =\ &\lambda_{\hat{D}_0,\dotsc,\hat{D}_{t-1}}(Y)\\
&-\!\int_{\Xan}\log\|s_{D_t,Z}\|\cdot \c1(\overline{\OO}(D_0))
\wedge \dotsb \wedge \c1(\overline{\OO}(D_{t-1}))\wedge\delta_Z
.\end{align*}
\end{prop}
\begin{proof}
This follows from \cite[Remark 9.5]{GuPisa} and Definition \ref{measure}.
\end{proof}

\begin{rem} \label{independence of choice of s}
If $|Z| \subseteq |D_t|$, one might wonder why the left hand side depends on the metrized pseudo-divisor $\hat{D}_t$, which does not play a role on the right hand side, where however an arbitrarily chosen meromorphic section of $\OO(D_t)|_Z$ occurs. This is closely related to the refined  intersection products of pseudo-divisors on formal $\Kval$-models given in  \cite[\S 5]{GuPisa} and the dependence of the choice of the meromorphic section fizzles out by the condition $|D_0|\cap\dotsb\cap|D_{t-1}|\cap |Z|=\emptyset$. 
\end{rem}

\subsection{Semipositivity}  \label{section: semipositivity}

It would be nice if we could extend local heights to all continuous metrics. 
Although the $\QQ$-formal metrics are dense in the space of continuous metrics, this is not possible
because the continuity property (\ref{continuity}) in Proposition \ref{proplocal} only holds for semipositive $\QQ$-formal metrics.
Following Zhang, we extend the theory of local heights to limits of semipositive $\QQ$-formal metrics which is important for canonical metrics and  equidistribution theorems.

Let  $X$ be a proper variety over an algebraically closed non-archimedean field $K$.
\begin{defn}\label{semip}
Let $\ML=(L,\|\cdot\|)$ be a metrized line bundle on $X$. The metric $\|\cdot\|$ is called \emph{semipositive}\index{metric!semipositive} if 
there exists a sequence $(\|\cdot\|_n)_{n\in\NN}$ of semipositive $\QQ$-formal metrics on $L$ such that
\[\lim_{n\to\infty}\dist(\|\cdot\|_n,\|\cdot\|)=0.\]
In this case we say that $\ML=(L,\|\cdot\|)$ is a \emph{semipositive (metrized) line bundle}\index{metrized line bundle!semipositive}.
The metric is said to be \emph{DSP}\index{metric!DSP} (for ``difference of semipositive'') if there are semipositive metrized line bundles $\overline{M}$, $\overline{N}$ on $X$
such that
$\ML=\overline{M}\otimes\overline{N}^{-1}$. Then $\ML$ is called a \emph{DSP (metrized) line bundle}\index{metrized line bundle!DSP} as well.
\end{defn}

\begin{rem}
If $\|\cdot\|$ is a $\QQ$-formal metric, 
then \cite[Proposition 7.2]{GuKue} says that $\|\cdot\|$ is semipositive in the sense of Definition \ref{formalmetric}
if and only if $\|\cdot\|$ is semipositive as defined in Definition \ref{semip}. 
So there is no ambiguity in the use of the term semipositive metric.
This answers the question raised in \cite[Remark 1.4.2]{BPS}. 
\end{rem}


\begin{art}
The tensor product and the pullback (with respect to a proper morphism) of semipositive metrics are again semipositive.
The tensor product, the dual and the pullback of DSP metrics are also DSP.
\end{art}

\begin{art}\label{locallimit}
By means of Proposition \ref{proplocal}, we can easily extend the local heights to DSP metrics.
Concretely, let $Y$ be a $t$-dimensional prime cycle and $\D_i=(L_i,\|\cdot\|_i,|D_i|,s_i)$, $i=0,\dotsc,t$, a collection of semipositive metrized pseudo-divisors on $X$ with $|D_0|\cap\dotsb\cap|D_t|\cap Y=\emptyset$.
By Definition \ref{semip}, there is, for each $i$, an associated sequence of semipositive $\QQ$-formal metrics $\|\cdot\|_{i,n}$ on $L_i$
such that $\dist(\|\cdot\|_{i,n},\|\cdot\|_i)\to 0$ for $n\to \infty$.
Then we define the \emph{local height}\index{height!local} of $Y$ with respect to $\D_0,\dotsc,\D_t$ as
\begin{align}\label{heightlimit}
\lambda_{\D_0,\dotsc,\D_t}(Y)\coloneq \lim_{n \to \infty} {\lambda_{(L_0,\|\cdot\|_{0,n},|D_0|,s_0),
\dotsc,(L_t,\|\cdot\|_{t,n},|D_t|,s_t)}(Y)}.
\end{align}
This does not depend on the choice of the approximating semipositive formal metrics and extends to cylces. For details, see \cite[§\,1]{GuM} or \cite[Theorem 5.1.8]{GuHab}. 

Let $Z$ be a $t$-dimensional cycle of $X$ and $(\ML_i,s_i)$, $i=0,\dotsc,t$, DSP metrized line bundles on $X$ with invertible meromorphic sections such that
$\left|\dvs(s_0)\right|\cap\dotsb\cap\left|\dvs(s_t)\right|\cap |Z|=\emptyset$.
By Example \ref{exdvs}, we obtain DSP metrized pseudo-divisors ${\widehat{\dvs}(s_i)}$, $i=0,\dotsc,t$.
Then, we denote the local height by 
\begin{align}\label{div-pseu}
\lambda_{(\ML_0,s_0),\dotsc,(\ML_{t},s_t)}(Z)\coloneq \lambda_{{\widehat{\dvs}(s_0)},\dotsc,\widehat{\dvs}(s_t)}(Z)\, .
\end{align}
\end{art}

\begin{prop}\label{proplocal2}
Let $Z$ be a $t$-dimensional cycle of $X$ and $\D_0,\dotsc,\D_t$ DSP metrized pseudo-divisors on $X$
with $|D_0|\cap\dotsb\cap|D_t|\cap |Z|=\emptyset$. Then there is a unique local height $\lambda(Z)\coloneq \lambda_{\D_0,\dotsc,\D_t}(Z) \in \RR $ satisfying the following properties:
\begin{enumerate}
	\item If $\D_0,\dotsc,\D_t$ are $\QQ$-formal metrized, then $\lambda(Z)$ is the local height of \ref{localdef}.
	\item $\lambda(Z)$ is multilinear and symmetric in $\D_0,\dotsc,\D_t$ and linear in $Z$.
	\item For a proper morphism $\varphi\colon X'\rightarrow X$ and a $t$-dimensional cycle $Z'$ on $X'$ satisfying
	$|D_0|\cap\dotsb\cap|D_t|\cap |\varphi(Z')|=\emptyset$, we have
	\[\lambda_{\varphi^*\D_0,\dotsc,\varphi^*\D_t}(Z')=\lambda_{\D_0,\ldots,\D_t}(\varphi_*Z').\]
	In particular, $\lambda_{\D_0,\dotsc,\D_t}(Z)$ does not change when restricting the metrized pseudo-divisors to the prime cycle $Z$.
	\item Let $\lambda'(Z)$ be the local height obtained by replacing the metric $\|\cdot\|$ of $\D_0$ by another DSP metric $\|\cdot\|'$.
	If the metrics of $\D_1,\dotsc,\D_t$ are semipositive and if $Z$ is effective, then
	\[
	|\lambda(Z)-\lambda'(Z)|\leq\dist(\|\cdot\|,\|\cdot\|')\cdot\deg_{\OO(D_1),\dotsc,\OO(D_t)}(Z).
	\]
	\item Let $f$ be a rational function on $X$ and let $\D_0=\widehat{\dvs}(f)$ be endowed with the trivial metric on $\OO(D_0)=\OO_X$.
	If $Y=\sum_P{n_PP}$ is a cycle representing $D_1.\dotsb .D_t.Z\in \text{\upshape CH}_0\left(|D_1|\cap \dotsb \cap |D_t|\cap|Z|\right)$, 
	then \[\lambda(Z)=\sum_P{n_P\cdot\log|f(P)|}.\]
\end{enumerate}
\end{prop}
\begin{proof} This follows immediately from Proposition \ref{proplocal} and the construction in \ref{locallimit} and
is established in \cite[Theorem 10.6]{GuPisa}.
\end{proof}

Similarly, there is a generalization of Chambert Loir's measures to semipositive and DSP line bundles:
\begin{prop}\label{limitm}
Let $Y$ be a $t$-dimensional subvariety of $X$ and $\ML_i=(L_i,\|\cdot\|_i)$, $i=1,\dotsc,t$, semipositive line bundles.
For each $i$, let $(\|\cdot\|_{i,n})_{n\in\NN}$ be the corresponding sequence of
$\QQ$-formal semipositive metrics on $L_i$ converging to $\me_i$.
Then the measures
\[\c1(L_1,\|\cdot\|_{1,n})\wedge\dotsb \wedge \c1(L_t|,\|\cdot\|_{t,n})\wedge\delta_Y\]
converge weakly to a regular Borel measure on $\Xan$. This measure is independent of the choice of the sequences.
\end{prop}
\begin{proof}
This follows from \cite[Proposition 3.12]{GuTrop}.
\end{proof}

\begin{defn}\label{MeasSem}
Let $Y$ be a $t$-dimensional subvariety of $X$ and $\ML_i=(L_i,\|\cdot\|_i)$, $i=1,\dotsc,t$, semipositive line bundles.
We denote the limit measure in \ref{limitm} by $\c1(\ML_1)\wedge \dotsb\wedge \c1(\ML_t)\wedge\delta_Y$
or shortly by $\c1(\ML_1) \dotso \c1(\ML_t)\delta_Y$. By multilinearity
this notion extends to a $t$-dimensional cycle $Y$ of $X$ and DSP line bundles $\ML_1,\dotsc,\ML_t$.
\end{defn}

Chambert Loir's measure is uniquely determined by the following property which is taken as definition in \cite[3.8]{GuTrop}.
\begin{prop}\label{measuregubler}
Let $\ML_1,\dotsc,\ML_t$ be DSP line bundles on $X$ and let $Z$ be a $t$-dimensional cycle.
For $j=1,\dotsc,t$ we choose any metrized pseudo-divisor $\D_j$ with $\overline{\OO}(D_j)=\ML_j$, for example $\D_j=(\ML_j,X,0)$.

If $g$ is any continuous function on $\Xan$, then there is a sequence of $\QQ$-formal metrics $(\|\cdot\|_n)_{n\in\NN}$ on $\OO_X$
such that $\log\|1\|_n^{-1}$ tends uniformly to  $g$ for $n\to\infty$ and
\[\int_{\Xan}{g \cdot \c1(\ML_1) \wedge \dotsb \wedge \c1(\ML_t)\wedge\delta_Z}=\lim_{n\to \infty}\lambda_{(\OO_X,\|\cdot\|_n,\emptyset,1),\D_1,\dotsc,\D_t}(Z).\]
\end{prop}
\begin{proof}By \cite[Theorem 7.12]{GuLocal}, the $\QQ$-formal metrics are dense in the space of continuous metrics on $\OO_X$. This implies the existence of the sequence $(\|\cdot\|_n)_{n\in\NN}$.
The second part follows from \cite[Proposition 3.8]{GuTrop}.
\end{proof}

\begin{cor}\label{metricchange}
Let $Z$ be a cycle on $X$ of dimension $t$ and let $\D_0,\dotsc,\D_t$ be DSP metrized pseudo-divisors with $|D_0|\cap\dotsb\cap|D_t|\cap |Z|=\emptyset$.
Replacing the metric $\|\cdot\|$ on $\OO(D_0)$ by another DSP metric $\|\cdot\|'$, we obtain a metrized pseudo-divisor $\hat{E}$.
Then $g\coloneq \log(\|s_{D_0}\|/\|s_{D_0}\|')$ extends to a continuous function on $X$ and
\[\lambda_{\D_0,\dotsc,\D_t}(Z)-\lambda_{\hat{E},\D_1,\dotsc,\D_t}(Z)=\int_{\Xan}{g\cdot  \c1(\overline{\OO}(D_1))
\wedge \dotsb \wedge \c1(\overline{\OO}(D_{t}))\wedge\delta_Z}.\]
\end{cor}
\begin{proof}Clearly $g$ defines a continuous function on $X$ and the claim follows easily from Proposition \ref{measuregubler}.
\end{proof}

\begin{prop}\label{propmeasure}
Let $Z$ be a $t$-dimensional cycle of $X$ and $\ML_1,\dotsc,\ML_t$ DSP line bundles. Then the measure
$\c1(\ML_1)\wedge \dotsb\wedge \c1(\ML_t)\wedge\delta_Z$ has the following properties:
\begin{enumerate}
	\item It is multilinear and symmetric in $\ML_1,\dotsc,\ML_t$ and linear in $Z$.
	\item Let $\varphi\colon X'\rightarrow X$ be a morphism of proper schemes over $K$ and $Z'$ a $t$-dimensional cycle of $X'$, then
	\[\varphi_*\left(\c1(\varphi^*\ML_1)\wedge\dotsb \wedge \c1(\varphi^*\ML_t)\wedge \delta_{Z'}\right)
	=\c1(\ML_1)\wedge\dotsb \wedge \c1(\ML_t)\wedge \delta_{\varphi_*Z'}.\]
	\item If the metrics of $\ML_1,\dotsc,\ML_t$ are semipositive, then $\c1(\ML_1)\wedge \dotsb\wedge \c1(\ML_t)\wedge\delta_Z$
	is a positive measure with total mass $\deg_{L_1,\dotsc,L_t}(Z)$.
\end{enumerate}
\end{prop}

\begin{proof}
We refer to Corollary 3.9 and Proposition 3.12 in \cite{GuTrop}.
\end{proof}


\begin{rem} \label{archimedean local heights}
In the archimedean case, i.e. for $K=\CC$, there is a similar theory of local heights and Chambert-Loir measures as presented above. Formal metrics are replaced by smooth metrics, semipositivity means positive curvature. Then uniform limits of semipositive smooth metrics lead to semipositive continuous metrics on the complex analytification of the line bundle. The DSP metrics are defined as above and we get Chambert-Loir measures as before. All of the above properties remain valid. For details, we refer to \cite{GuCan}.
\end{rem}

\subsection{Induction formula for DSP local heights}  \label{Induction formula for DSP local heights}

It is quite difficult to generalize 
 the induction formula from Proposition \ref{ifformal} to DSP line bundles. In the case of a discrete valuation (or an archimedean place) and properly intersecting Cartier divisors on a projective variety, this was done by Chambert-Loir and Thuillier in \cite[Th\'eor\`eme 4.1]{CT}. The goal of this subsection is to generalize their result to 
any  proper variety $X$ over an algebraically closed non-archimedean field $K$. 
Note that, once the algebraically closed case is settled, invariance of the local heights by base change gives the induction formula for any non-archimedean field.

\begin{thm}[Approximation theorem]\label{appro}
Let $L$ be a  line bundle on $X$ endowed with a semipositive formal metric $\metr$ and let  $s$ be a global section of $L$ on $X$ which does not vanish identically. 
Then there is a sequence $(\|\cdot\|_n)_{n\in\NN}$ of formal metrics on the trivial bundle $\OO_X$ with the following properties:
\begin{itemize}
	\item[(i)] The sequence $\left(\log\|1\|^{-1}_n\right)_{n\in\NN}$ converges pointwise to $\log\|s\|^{-1}$ and it is monotonically increasing.
	\item[(ii)] For each $n\in \NN$, the formal metric $\|\cdot\|/\|\cdot\|_n$ on $L$ is semipositive.
\end{itemize}
\end{thm}

Over a complete discrete valuation ring and for a projective variety $X$, this was proven by Chambert-Loir and Thuillier  \cite[Th\'eor\`eme 3.1]{CT}. 
We will use a similar, but more analytic approach. 

\begin{proof} We fix some non-zero element $\pi$ in $\Kmax$ 
and define, for each $n\in\NN$, the strictly analytic domains 
\begin{align}\label{affdom}
A_n\coloneq \left\{x\in \Xan\ |\ \|s(x)\|\geq |\pi^n|\right\}\  \text{ and }\ B_n\coloneq\left\{x\in \Xan\ |\ \|s(x)\|\leq |\pi^n|\right\}.
\end{align}
By \ref{propformal}, the formal metric $\|\cdot\|$ on $L$ is given by a (finite) $\rm G$-covering
$\{U_i\}_{i\in I}$ of $\Xan$ by strictly affinoid domains and non-vanishing regular sections $t_i\in \Lan(U_i)$ with $\|t_i\|\equiv 1$. 
We refer to \cite[§\,9]{BGR} for the $\rm G$-topology on rigid analytic varieties and to \cite[\S 1.6]{Ber2} for the transition to Berkovich analytic spaces. 
Let $g_{ij}=t_j/t_i\in\OO(U_i\cap U_j)^\times$ be the transition functions.
Then the nowhere vanishing restrictions $s|_{U_i\cap A_n}$ may be identified with regular functions $f_i\in\OO(U_i\cap A_n)^\times$
satisfying $f_i=g_{ij}f_j$ on $U_i\cap U_j \cap A_n$.
There is a unique continuous metric $\metr_n$ on $\OO_X$ with
$$\|1\|_n := \max\{\|s\|, |\pi|^n\}.$$
Since the functions $f_i^{-1}\in\OO(U_i\cap A_n)$, $\pi^{-n}\in \OO(U_i\cap B_n)$ are local frames of $\OO_{\Xan}$  on a $\rm G$-covering of $\Xan$ by strictly affinoid domains and since these frames have norm $1$ with respect to $\metr_n$, it follows from 
\ref{propformal} that $\|\cdot\|_n$ is a formal metric on $\OO_X$. By construction, we have
\begin{equation} \label{formula1 for metric}
\|1\|_n
=\begin{cases}
  \|s\|  &\text{on }A_n,\\
  |\pi|^{n}  &\text{on }B_n.
\end{cases}
\end{equation}
Clearly, the sequence $\left(\log\|1\|_n^{-1}\right)_{n\in\NN}$ is monotonically increasing and converges pointwise to $\log\|s\|^{-1}$. This proves (i). 

To prove (ii), we show that, for each $n\in \NN$, the formal metric $\|\cdot\|_n'\coloneq \|\cdot\|/\|\cdot\|_n$ 
is semipositive on $L\otimes \OO_X^{-1}=L$. Note that 
\begin{equation} \label{formula2 for metric}
\|s\|_n'=\frac{\|s\|}{\|1\|_n}=
\begin{cases} 1 & \text{on }A_n,\\
  \|s\|\cdot|\pi^{-n}|  & \text{on }B_n.\end{cases}
\end{equation}

For the $\rm G$-covering $\{U_i\cap A_n, U_i\cap B_n\}_{i\in I}$ of $\Xan$ by strictly affinoid domains, there exists a formal
$\Kval$-model $\XXXX_n$ of $\Xan$ and a formal open covering $\{\UUUU_{i,n},\VVVV_{i,n}\}_{i\in I}$ of $\XXXX_n$ such that
$\UUUU_{i,n}^{\an}=U_i\cap A_n$ and $\VVVV_{i,n}^{\an}=U_i\cap B_n$ (see \cite[Theorem 5.5]{BL2}).
We may assume that $\XXXX_n$ has reduced special fiber (cf. \ref{commonmod}).
Then, by \ref{propformal}, the formal metric $\|\cdot\|_n'$ is associated to the formal $\Kval$-model $(\LLLL_n',\XXXX_n)$ of $(X,L)$ given by
\begin{align}\label{fff}
\LLLL_n'(\UUUU)=\left\{r\in \Lan(\UUUU^{\an})\,|\,\|r(x)\|_n'\leq 1\ \forall x\in \UUUU^{\an}\right\}
\end{align}
on a formal open subset $\UUUU$ of $\XXXX_n$.
Therefore, we can consider $s$ as a global section of $\LLLL_n'$ as  we have $\|s\|_n' = \|s\|_n/\|1\|_n \leq 1$.

Let $C\subseteq \widetilde{\XXXX}_n$ be a closed integral curve.
If $s$ doesn't vanish identically on $C$, then
\[\deg_{\widetilde{\LLLL}_n'}(C)=\deg(\c1(\widetilde{\LLLL}_n'). C)=\deg(\dvs(s|_C))\geq 0.\]
If $s$ vanishes identically on $C$,
let $\BBBB_n$ be the union of the formal open subsets $(\VVVV_{i,n})_{i \in I}$. 
Then it follows from (\ref{affdom}) and \eqref{formula2 for metric} that $\widetilde{\BBBB}_n=\red(B_n)$ contains
$C$.
By passing from the $\rm G$-covering $\{U_i\}_{i\in I}$ to the refinement $\{U_i\cap A_n,U_i\cap B_n\}_{i \in I}$ and using 
the frame $t_i$ with $\|t_i\|=1$ on $U_i\cap A_n$ and on $U_i\cap B_n$, we see again from \ref{propformal} that the formal metric $\metr$ is given by a $\Kval$-model
 $\LLLL_n$  on $\XXXX_n$. Moreover, 
$\LLLL_n$ 
satisfies a similar formula as in (\ref{fff}) with $\metr$ replacing $\metr_n'$. 
We get an isomorphism $\LLLL_n|_{\BBBB_n}\cong \LLLL_n'|_{\BBBB_n}$ given by $r\mapsto \pi^n\cdot r$. 
By assumption, the formal metric $\metr$ is semipositive and hence $\widetilde{\LLLL}_n$ is nef. 
Since $\tilde\BBBB_n$ is a neighborhood of $C$, we obtain
\[\deg_{\widetilde{\LLLL}'_n}(C)=\deg_{\widetilde{\LLLL}_n}(C)\geq0.\]
This implies the semipositivity of $\|\cdot\|/\|\cdot\|_n$ proving (ii).
\end{proof}

\begin{cor}\label{korappro}
We use the notations from the approximation theorem \ref{appro} and in addition, we consider 
semipositive line bundles $\ML_1,\dotsc,\ML_{t-1}$ on the $t$-dimensional proper variety $X$.
Let $\mu$ be a (signed) Radon measure on $\Xan$ such that, for every formal metric $\|\cdot\|'$ on $\OO_X$,
\begin{align} \label{limit equation}
\lim_{n \to \infty}{\int_{\Xan}{\log\|1\|'\cdot \c1\left(\OO_X,\|\cdot\|_n\right)\c1(\ML_1)\dotso \c1(\ML_{t-1}) }}
=\int_{\Xan}{\log\|1\|'\cdot\mu}\ .\end{align}
Then the sequence 
$\left(\c1(\OO_X,\|\cdot\|_n)\c1(\ML_1)\dotso \c1(\ML_{t-1})\right)_{n\in\NN}$ of measures on $\Xan$
converges weakly to $\mu$.
\end{cor}
\begin{proof}We define $\nu\coloneq \c1(L,\|\cdot\|) \c1(\ML_1)\dotso \c1(\ML_{t-1}) $
and, for each $n\in\NN$, we set $ \mu_n\coloneq  \c1(\OO_X,\|\cdot\|_n)\c1(\ML_1)\dotso \c1(\ML_{t-1})$.
Then, by the approximation theorem \ref{appro} and Proposition \ref{propmeasure}\,(iii), the measures
\[ \nu-\mu_n=\c1\Bigl(L,\tfrac{\|\cdot\|}{\|\cdot\|_n}\Bigr) \c1(\ML_1)\dotso \c1(\ML_{t-1})\]
are positive with finite total mass $\deg_{L,L_1,\dotsc,L_{t-1}}(X)$ independent of $n$.
By linearity, the equation \eqref{limit equation} also holds  for any $\QQ$-formal metric on $\OO_X$. 
By 
\cite[Theorem 7.12]{GuLocal}, the space $\{-\log\|1\|' \mid \text{$\metr'$ 
$\QQ$-formal metric on $\OO_X$}\}$ is dense in $C(\Xan)$. An easy application of the triangle inequality 
shows that 
$$\lim_{n \to \infty} \int_{\Xan}  f \cdot (\nu - \mu_n) = \int_{\Xan} f \cdot (\nu - \mu)$$
for any $f \in  C(\Xan)$. This proves the claim.
\end{proof}

\begin{thm}[Induction formula]\label{ifDSP}
Let $Z$ be a $t$-dimensional prime cycle on $X$
and let $\D_0=(\ML_0,|D_0|,s_0),\dotsc,\D_t=(\ML_t,|D_t|,s_t)$ be DSP pseudo-divisors with $|D_0|\cap\dotsb\cap|D_t|\cap |Z|=\emptyset$.
If $|Z|\nsubseteq|D_t|$, then let $s_{t,Z}\coloneq s_{t}|_Z$, otherwise we choose any non-zero meromorphic section $s_{t,Z}$ of $L_t|_Z$.
Let $\cyc(s_{t,Z})$ be the Weil divisor of $s_{t,Z}$ considered as a cycle on $X$.

Then the function $\log\|s_{t,Z}\|$ is integrable with respect to $\c1(\ML_0)\!\wedge \dotsb \wedge \c1(\ML_{t-1})\!\wedge\delta_Z$ and
we have
\begin{align}\label{inductionf}
\lambda_{\hat{D}_0,\dotsc,\hat{D}_t}(Z)
=\lambda_{\hat{D}_0,\dotsc,\hat{D}_{t-1}}(\cyc(s_{t,Z}&))\nonumber\\
-\!\int_{\Xan}\log&\|s_{t,Z}\|\cdot \c1(\ML_0)\!
\wedge \dotsb \wedge \c1(\ML_{t-1})\!\wedge\delta_Z\,.
\end{align}
\end{thm}
%
\begin{rem} \label{remark about induction formula}
If $\ML_0,\dotsc,\ML_t$ have $\QQ$-formal metrics, then this result is just Proposition \ref{ifformal}.
It is also evident if $L_t$ is the trivial bundle and hence, $\log\|s_{t,Z}\|$ is a continuous function on $Z$.
The difficulties of the general case arise from the relation between the limit process defining the measure,
and the singularities of the function $\log\|s_{t,Z}\|$.

Our proof is based on \cite[Th\'eor\`eme 4.1]{CT} where the formula is demonstrated under the additional assumptions that 
$X$ is projective over a complete discrete valuation ring and $s_0,\dotsc,s_t$ are global sections with associated Cartier divisors intersecting $Z$ properly. As explained at the beginning of our proof,  our more general setting of pseudo-divisors satisfying $|D_0|\cap\dotsb\cap|D_t|\cap |Z|=\emptyset$ can be  reduced to the projective situation with properly intersecting global sections. So the main difficulty is to deal with a non-discrete valuation where the main ingredient is now our generalization of the approximation theorem.

The induction formula \ref{ifDSP} also holds in the archimedean case, i.e.\ for $K=\CC$. Indeed, this was proven by Chambert-Loir and Thuillier in the projective situation described above and extends to our setting as explained above.
\end{rem}

\begin{proof}[Proof of the induction formula \ref{ifDSP}]
By Proposition \ref{proplocal2}\,(iii), we may assume that 
$X=Z$, especially $s_t=s_{t,Z}$.
Furthermore, we can suppose that $X$ is projective by Chow's lemma and functoriality of the height (Proposition \ref{proplocal2}). We denote the crucial difference of local heights by 
\begin{equation} \label{crucial difference}
\Delta(\hat{D}_0,\dotsc,\hat{D}_t):=\lambda_{\hat{D}_0,\dotsc,\hat{D}_t}(X)-
\lambda_{\hat{D}_0,\dotsc,\hat{D}_{t-1}}(\cyc(s_{t})).
\end{equation}
First note that $\Delta(\hat{D}_0,\dotsc,\hat{D}_t)$ depends continuously on the metrics $(\metr_0,\dots,\metr_t)$ by the metric change formula \ref{metricchange}. If all the metrics are formal, then the induction formula \ref{ifformal} holds and hence a continuity argument shows that $\Delta(\hat{D}_0,\dotsc,\hat{D}_t)$ only depends on $\ML_0, \dots, \ML_{t-1}$ and $\hat{D}_t$. In fact, it is easy to see that 
$$\Delta_{\ML_0,\dots,\ML_{t-1}}(\hat{D}_t):=\Delta(\hat{D}_0,\dotsc,\hat{D}_t)$$
makes sense for any DSP metrized line bundles $\ML_0,\dots,\ML_{t-1}$ and any  metrized DSP pseudo-divisor $\hat{D}_t$ by choosing (generic) pseudo-divisors $\hat{D}_0,\dotsc,\hat{D}_{t-1}$ for $\ML_0,\dots,\ML_{t-1}$ with $|D_0|\cap\dotsb\cap|D_t|=\emptyset$. By Proposition \ref{proplocal2}(i), $\Delta_{\ML_0,\dots,\ML_{t-1}}(\hat{D}_t)$ is multilinear in $(\ML_0,\dots,\ML_{t-1},\hat{D}_t)$. Using this and projectivity, we may assume that $L_0, \dots , L_t$ are very ample, $s_t$ is a global section and  the metrics are semipositive. 

Note that $-\log\|s_t\|$ is a measurable function on $\Xan$. In the following, we will integrate it with respect  to  positive Radon measures on $\Xan$ allowing the value $\infty$ for the integral, the value $-\infty$ is out of question as the function is bounded from below and $\Xan$ is compact. To prove the theorem, it is enough to show that 
\begin{equation} \label{crucial induction formula}
\Delta_{\ML_0,\dots,\ML_{t-1}}(\hat{D}_t)=-\int_{\Xan}{\log\|s_{t}\|\cdot \c1(\ML_0)\wedge \dotsb \wedge \c1(\ML_{t-1})}
\end{equation}
which implies also integrability of $-\log\|s_t\|$. 
The metric change formula \ref{metricchange} shows that the last metric $\metr_t$ may be assumed to be a semipositive formal metric. Using generic hyperplane sections of $L_0, \dots, L_{t-1}$, we may assume that the local heights in the theorem and in \eqref{crucial difference} are with respect to very ample pseudo-divisors $D_0,\dots, D_t$ which intersect properly. 

 We prove by induction on $k\in \{0,\dotsc,t\}$ that the theorem holds if $\ML_i$ is a formally metrized line bundle for $i\geq k$.
The case $k=0$ is just the induction formula for formal metrics (see Proposition \ref{ifformal}).
We assume that the statement holds for $k$ and demonstrate it for $k+1$. 

In the following we fix  a semipositive formal metric $\|\cdot\|'$ on $L_k$. We denote the corresponding metrized line bundle by $\MM_k$ 
and the metrized pseudo-divisor $(\MM_k,|D_k|,s_k)$ by $\E_k$. Then we can extend $\varphi_k\coloneq \log\|s_k\|'-\log\|s_k\|$ to a continuous function on $\Xan$ and $\OO_X(\varphi_k)\coloneq \ML_k\otimes\MM_k^{-1}$ is a DSP line bundle whose underlying line bundle is trivial.
To emphasize that the following line bundles are equipped with formal metrics, we will use 
$\MM_i:=\ML_i$ and $\E_i:=\D_i$ for $i=k+1,\dots,t$.
Since $\MM_k,\dotsc, \MM_{t-1}$ are formally metrized line bundles, our induction hypothesis implies
\begin{align}
&\int_{\Xan}{\log\|s_{t}\|\cdot  \c1(\ML_0)\dotso \c1(\ML_{k-1})\c1(\MM_k)\dotso \c1(\MM_{t-1})}\nonumber\\
=\hspace{0.3em}&\lambda_{\hat{D}_0,\dotsc,\hat{D}_{k-1},\hat{E}_k,\dotsc,\hat{E}_{t-1}}(\cyc(s_t))
-\lambda_{\hat{D}_0,\dotsc,\hat{D}_{k-1},\hat{E}_k,\dotsc,\hat{E}_t}(X).\label{1}
\end{align}
If we apply the change of metrics formula \ref{metricchange} twice in \eqref{1} to replace $\hat{E}_k$ by $\hat{D}_k$, then the same computation as in the proof of \cite[Th\'eor\`eme 4.1]{CT} shows that \eqref{crucial induction formula}  is equivalent to the claim 
\begin{align*}
&\int_{\Xan}{\varphi_k\cdot \c1(\ML_0)\dotso \c1(\ML_{k-1})\c1(\MM_{k+1})\dotso \c1(\MM_t)}\\
=&
\int_{\Xan}{\varphi_k\cdot \c1(\ML_0)\dotso \c1(\ML_{k-1})\c1(\MM_{k+1})\dotso \c1(\MM_{t-1})\delta_{\cyc(s_t)}}\\
&-\int_{\Xan}{\log\|s_{t}\|\cdot \c1(\OO(\varphi_k))\c1(\ML_0)\dotso \c1(\ML_{k-1})\c1(\MM_{k+1})\dotso \c1(\MM_{t-1})}\ .
\end{align*}
Using our reduction steps at the beginning, we can apply the \emph{approximation theorem \ref{appro}}:
Let $(\|\cdot\|_n)_{n\in\NN}$ be a sequence of formal metrics on $\OO_X$ 
such that the functions $g_n\coloneq \log\|1\|_n^{-1}$ tend pointwise to $\log\|s_{t}\|^{-1}$,
the sequence $(g_n)_{n\in\NN}$ is monotonically increasing
and $(\OO_X,\|\cdot\|_n)$ is a DSP line bundle.
Applying Lebesgue's monotone convergence theorem and using
Proposition \ref{measuregubler} and \ref{propmeasure}\,(i), we obtain
\begin{align*}
&\int_{\Xan}{\log\|s_{t}\|^{-1}\cdot \c1(\OO(\varphi_k)) \c1(\ML_0)\dotso \c1(\ML_{k-1})\c1(\MM_{k+1})\dotso \c1(\MM_{t-1})}\\
=&\lim_{n\to\infty}{\int_{\Xan}{g_n\cdot \c1(\OO(\varphi_k)) \c1(\ML_0)\dotso \c1(\ML_{k-1})\c1(\MM_{k+1})\dotso \c1(\MM_{t-1})}}\\
=&\lim_{n\to\infty}{\lambda_{\widehat{\OO}(g_n),\widehat{\OO}(\varphi_k),\hat{D}_0,\dotsc,\hat{D}_{k-1},
\hat{E}_{k+1},\dotsc,\hat{E}_{t-1}}}(X)\\
=
&\lim_{n\to\infty}{\lambda_{\widehat{\OO}(\varphi_k),\widehat{\OO}(g_n),\hat{D}_0,\dotsc,\hat{D}_{k-1},
\hat{E}_{k+1},\dotsc,\hat{E}_{t-1}}}(X)\\
=
&\lim_{n\to\infty}{\int_{\Xan}{\varphi_k \cdot \c1(\OO_X,\|\cdot\|_n) \c1(\ML_0)\dotso \c1(\ML_{k-1})\c1(\MM_{k+1})\dotso \c1(\MM_{t-1})}}\ .
\end{align*}
For  $\varphi_k=\log({\|\cdot\|'_k}/{\|\cdot\|_k})$, this shows that \eqref{crucial induction formula} is equivalent to the claim 
\begin{align}
&\lim_{n\to\infty}
{\int_{\Xan}{\varphi_k \cdot \c1(\OO_X,\|\cdot\|_n)\c1(\ML_0)\dotso \c1(\ML_{k-1})\c1(\MM_{k+1})\dotso \c1(\MM_{t-1})}}\nonumber\\
=&\int_{\Xan}\varphi_k\cdot  \left( \c1(\ML_0)\dotso \c1(\ML_{k-1})\c1(\MM_{k+1})\dotso \c1(\MM_t) \right. \label{letzteGl}\\
& \phantom{\int_{\Xan}\varphi_k\cdot  ( } \left. -  \c1(\ML_1)\dotso \c1(\ML_{k-1})\c1(\MM_{k+1})\dotso \c1(\MM_{t-1})\delta_{\cyc(s_t)} \right). \nonumber
\end{align}
The induction hypothesis implies that equation (\ref{letzteGl}) always holds if ${\|\cdot\|'_k}/{\|\cdot\|_k}$ is a formal metric.
But then Corollary \ref{korappro}  says that this equation is also true if $\varphi_k$ is only continuous.
This shows the induction formula (\ref{inductionf}) and therefore the integrability
of $\log\|s_t\|$ with respect to $\c1(\ML_0) \dotsb  \c1(\ML_{t-1})$.
\end{proof}

\begin{cor}\label{MeasZero}
Let $\ML_1, \dots, \ML_n$ be DSP-metrized line bundles on the $n$-dimensional proper variety $X$ over $K$ and let $Y$ be any proper closed subset of $X$ endowed with the induced reduced structure. Then 
$Y^{\rm an}$  has measure zero with respect to the measure $\c1(\ML_1)\wedge \dotso \wedge \c1(\ML_{n})$ on $X^{\rm an}$.
\end{cor}

\begin{proof}
By Chow's lemma and Proposition \ref{propmeasure}\,(ii), we may assume that $X$ is projective.
Then $Y$  is contained in the support of an effective pseudo-divisor $(L,|\dvs(s)|,s)$.
By the induction formula, the function $\log\|s\|^{-1}$ is integrable with respect to $\c1(\ML_1) \wedge \dotso \wedge \c1(\ML_{n-1})$, but
it takes the value $+\infty$ on $|\dvs(s)|^{\rm an}$. Thus  $|\dvs(s)|^{\rm an}$ and  the subset $Y^{\rm an}$ have measure zero with respect to the measure $\c1(\ML_1) \wedge \dotso \wedge \c1(\ML_{n})$ on $X^{\rm an}$.
\end{proof}

\begin{cor} \label{max and integrability}
Let $\ML_1, \dots, \ML_n$ be semipositive metrized line bundles on the $n$-dimensional proper variety $X$ over $K$. For $i=1, \dots, m$, let $s_i$ be a non-trivial meromorphic section of a DSP-metrized line bundle $M_i$ with DSP-metric $\metr_i$. Then $\phi := \max_i \log \|s_i\|_i$ is integrable with respect to the measure $\mu := \c1(\ML_1)\wedge \dotso \wedge \c1(\ML_{n})$ and the function $(\ML_1, \dots, \ML_n) \mapsto \int_{X^{\an}} \phi \, d\mu$ is continuous with respect to uniform convergence of the semipositive metrics on the fixed line bundles $L_1, \dots, L_n$. 
\end{cor}

\begin{proof}  
Using Chow's lemma, we may assume that $X$ is projective. 
We first handle the case $m=1$. Then it follows from Theorem \ref{ifformal} that $\log\|s\|$ is $\mu$-integrable and that we may write $\int_{X^{\an}} \phi \, d\mu$ as a difference of two local heights. Then continuity with respect to uniform convergence of the metrics follows from the similar property of local heights (see Proposition \ref{proplocal2}(iv)).  

Now we deal with the case $m \geq 2$. Let $H$ be an ample line bundle on $X$. Using that $M_i$ is the difference of $M_i \otimes H^{k}$ and $H^{k}$ and both are very ample for $k$ sufficiently large, we deduce easily that $s_i=t_i/t$ for $t_i \in \Gamma(X,M_i \otimes H^{k})$ and $t \in \Gamma(X,H^k)$. Moreover, we may assume that the same $k$ and that the same denominator $t$ work for all $i=1,\dots,m$. We endow $H$ with any DSP-metric and $M_i \otimes H^{k}$ with the tensor metric. Using the case $m=1$ handled above, it is enough to show the claim for  the function $ \max_i \log \|t_i\|_i$. Hence we may assume that all $s_i$ are global sections. 

Using the case $m=1$, we deduce that $\phi$ is a maximum of $\mu$-integrable functions and hence $\phi$ is also $\mu$-integrable. Since all the $s_i$ are now global sections, there is $C \in \RR$ such that $\log \|s_i\|_i \leq \phi \leq C$ for any $i=1, \dots , m$. A sandwich argument and the case $m=1$ yield that $\phi$ is continuous with respect to uniform convergence of the semipositive metrics  $L_1, \dots, L_n$.
\end{proof}



\section{Metrics and local heights of toric varieties} \label{local heights toric}

In this section, we study local heights of proper toric varieties with respect to line bundles endowed with toric metrics. We generalize the program of Burgos--Philippon--Sombra \cite[Chapter 4]{BPS}, where they assumed to a large extend that valuations are discrete, to arbitrary non-archimedean fields. In the first two sections, we recall the classical theory of toric varieties and the theory of toric schemes from \cite{Guide}. The first new results come in §\,\ref{Toric Cartier Divisors on Toric Schemes} where we describe Cartier divisors on toric schemes in terms of piecewise affine functions. In §\,\ref{Metrized line bundles on toric varieties}, we recall toric metrics, tropicalizations and the Kajiwara--Payne compactification.   In §\,\ref{Semipositive metrics and measures on toric varieties}, the main result is the characterization of semipositive toric metrics in terms of concave functions. Moreover, we give a formula for the Chambert--Loir measure of a semipositive toric metric in terms of the associated Monge--Amp\`ere measure. Finally, we prove in §\,\ref{Local heights of toric varieties} the formula for toric local heights explained in Theorem \ref{thm3} of the introduction.

\subsection{Toric varieties} \label{Toric Varieties}
We give a short overview of the theory of (normal) toric varieties over a field $K$ following closely the notation in  \cite[3.1-3.4]{BPS}.
For details, we refer to \cite{KKMS}, 
\cite{Fu2} and \cite{CLS}.


Let $M$ be a free abelian group of rank $n$ and $N\coloneq M^\vee\coloneq\Hom(M,\ZZ)$ its dual group. The natural pairing between $m\in M$ and $u\in N$ is denoted by 
$\la m,u \ra \coloneq u(m)$.
We have the split torus $\TT \coloneq\spec (K[M])$ over a field $K$ of rank $n$. Then $M$ can be considered as the character lattice of $\TT$ and $N$ as the lattice of one-parameter subgroups.
For $m\in M$ we will write $\chi^m$ for the corresponding character.
If $G$ is an abelian group, we set $N_G=  N\otimes_{\ZZ} G$. In particular, $\NR = N\otimes_{\ZZ} \RR$ is an $n$-dimensional real vector space with dual space $\MRR$.


\begin{defn} Let $K$ be a field and $\TT$ a split torus over $K$.
A ($\TT$-)\emph{toric variety} is an irreducible variety $X$ over $K$ containing $\TT$ as a (Zariski) open subset such that the translation action of $\TT$ on itself extends to an algebraic
action $\mu\colon \TT \times X \rightarrow X$.
\end{defn}

\begin{art}\label{afftor}
There is a nice description of normal toric varieties in combinatorial data.
At first we have a bijection between the sets of
\begin{enumerate}
	\item strongly convex rational polyhedral cones $\sigma$ in $\NR$, 
	\item isomorphism classes of affine normal $\TT$-toric varieties $X$ over $K$.
\end{enumerate}
This correspondence is given by $\sigma \mapsto U_\sigma=\spec(K[M_{\sigma}])$, where $K[M_{\sigma}]$ is the semigroup algebra of
\[M_{\sigma}=\sigmad\cap M =\left\{m\in M \mid  \la m,u\ra\geq 0 \ \forall u \in \sigma\right\}.\]
The action of $\TT$ on $U_\sigma$ is induced by
\[K[M_\sigma]\rightarrow K[M]\otimes K[M_\sigma],\quad \ \chi^m\mapsto \chi^m\otimes \chi^m .\]

More generally, we consider a fan $\Sigma$ in $\NR$ (Definition \ref{a4}). If $\sigma,\sigma'\in\Sigma$, then $U_{\sigma}$ and $U_{\sigma'}$ glue together along the open 
subset $U_{\sigma\cap\sigma'}$. this leads to a normal $\TT$-toric variety
\[X_{\Sigma}=\bigcup_{\sigma\in\Sigma}{U_{\sigma}}.\]
This construction induces a bijection between the set of fans $\Sigma$ in $\NR$ and the set of isomorphism classes of normal toric varieties $X_{\Sigma}$ with torus $\TT$.
\end{art}

\begin{art}\label{proptor}
Many properties of normal toric varieties are encoded in their fans, for example:
\begin{enumerate}
\item A normal toric variety $X_\Sigma$ is proper if and only if the fan is complete, i.\,e. $|\Sigma|\coloneq\bigcup_{\sigma\in\Sigma}{\sigma}=\NR$.
\item A normal toric variety $X_\Sigma$ is smooth if and only if the minimal generators of each cone $\sigma\in\Sigma$ are part of a $\ZZ$-basis of $N$.
\end{enumerate}
\end{art}

\begin{art}\label{orbit-cone}
Let $X_\Sigma$ be the normal toric variety of the fan $\Sigma$ in $\NR$. Then there is a bijective correspondence
between the cones in $\Sigma$ and the $\TT$-orbits in $X_\Sigma$. 
The closures of the orbits in $X_\Sigma$ have the structure of normal toric varieties what we describe in the following:
For $\sigma \in \Sigma$ we set
\[N(\sigma)=N/\la N\cap \sigma\ra,\quad M(\sigma)
=N(\sigma)^\vee
=M\cap \sigma ^\bot,\quad  O(\sigma)=\spec(K[M(\sigma)]),\]
where $\sigma^\bot$ denotes the orthogonal space to $\sigma$. 
Then $O(\sigma)$ is a torus over $K$ of dimension $n-\dim(\sigma)$ which can be identified with a $\TT$-orbit in $X_\Sigma$ via
the surjection
\[K[M_\sigma]\longrightarrow K[M(\sigma)],\quad \chi^m\longmapsto \begin{cases} \chi^m &\mbox{if } m\in \sigma^\bot, \\
0 & \mbox{otherwise} . \end{cases}\]
We denote by $V(\sigma)$ the closure of $O(\sigma)$ in $X_\Sigma$.
Then $V(\sigma)$ can be identified with the normal $O(\sigma)$-toric variety $X_{\Sigma(\sigma)}$ which is given by the fan
\begin{align}\label{sigmafan}
\Sigma(\sigma)=\{\tau+\la N\cap \sigma\ra_{\RR}\mid   \tau \in \Sigma,\,\tau\supseteq\sigma \}
\end{align}
in $N(\sigma)_{\RR}=\NR/\la N\cap \sigma\ra_{\RR}$. 
\end{art}

\begin{defn}
Let $X_i$, $i=1,2$, be toric varieties with dense open torus $\TT_i$. We say that a morphism $\varphi\colon X_1\rightarrow X_2$ is \emph{toric} if $\varphi$
maps $\TT_1$ into $\TT_2$ and $\varphi|_{\TT_1}\colon \TT_1\rightarrow\TT_2$ is a group morphism.
\end{defn}

\begin{art}
Any toric morphism $\varphi\colon X_1\rightarrow X_2$ is \emph{equivariant}, i.\,e. we have a commutative diagram
\[
\begin{xy}
  \xymatrix{
      \TT_1\times X_1 \ar[r]^{\mu_1} \ar[d]_{\varphi|_{\TT_1}\times\varphi}    &   X_1 \ar[d]^\varphi  \\
      \TT_1\times X_1\ar[r]^{\mu_2}            &   X_2  \ ,
  }
\end{xy}
\]
where $\mu_1,\mu_2$ denote the torus actions.
\end{art}

Toric morphisms can be desribed in combinatorial terms.
\begin{art}\label{tormor-com}
For $i=1,2$, let $N_i$ be a lattice with associated torus $\TT_i=\spec K[N_i^\vee]$ and let $\Sigma_i$ be a fan in $N_{i,\RR}$.
Let $H\colon N_1\rightarrow N_2$ be a linear map which is \emph{compatible} with $\Sigma_1$ and $\Sigma_2$. That is, for each cone $\sigma_1\in \Sigma_1$,
there exists a cone $\sigma_2\in\Sigma_2$ with $H(\sigma_1)\subseteq \sigma_2$.
Then $H$ induces a group morphism $\TT_1\rightarrow \TT_2$ of tori and, by the compatibility of $H$, this group morphism extends to a toric morphism
$\varphi_H\colon X_{\Sigma_1}\rightarrow X_{\Sigma_2}$.

We fix $N_i$, $\TT_i$ and $\Sigma_i$, $i=1,2$, as above.
Then the assignment $H\mapsto \varphi_H$ induces a bijection between the sets of 
\begin{enumerate}
	\item linear maps $H\colon N_1\rightarrow N_2$ which are compatible with $\Sigma_1$ and $\Sigma_2$;
	\item toric morphisms $\varphi\colon X_{\Sigma_1}\rightarrow X_{\Sigma_2}$.
\end{enumerate}
A toric morphism $\varphi_H\colon X_{\Sigma_1}\rightarrow X_{\Sigma_2}$ is proper if and only if $H^{-1}(|\Sigma_2|)=|\Sigma_1|$.
\end{art}

\begin{defn}
A \emph{$\TT$-Cartier divisor} on a $\TT$-toric variety $X$ is a Cartier divisor $D$ on $X$ which is invariant under the action of $\TT$ on $X$, i.\,e. we have $\mu^*D=p_2^*D$ denoting by $\mu\colon\TT\times X\rightarrow X$ the toric action and by $p_2\colon \TT\times X\rightarrow X$ the second projection.
\end{defn}

Torus-invariant Cartier divisors on a normal $\TT$-toric variety $X=X_\Sigma$ can be described in terms of support functions:

\begin{defn}\label{SF-Def}
A function $\Psi\colon |\Sigma|\longrightarrow \RR$ is called a \emph{virtual support function} on $\Sigma$, if there exists a set
$\{m_{\sigma}\}_{\sigma\in\Sigma}$ of elements in $M$ such that, for each cone $\sigma \in\Sigma$, we have
$\Psi (u)=\la m_{\sigma},u\ra$ for all $u\in \sigma$.
It is said to be \emph{strictly concave} if, for  maximal cones $\sigma, \tau\in \Sigma$, we have $m_\sigma =  m_\tau$ if and only if $\sigma = \tau$.
A \emph{support function} is a concave virtual support function on a fan.
\end{defn}

\begin{art}\label{SF-CD}
Let $\Psi$ be a virtual support function given by the data $\{m_{\sigma}\}_{\sigma\in\Sigma}$.
Then $\Psi$ determines a $\TT$-Cartier divisor
\[D_\Psi\coloneq \left\{\left(U_{\sigma},\chi^{-m_{\sigma}}\right)\right\}_{\sigma\in\Sigma}\]
on $X_{\Sigma}$.
The map $\Psi \mapsto D_{\Psi}$ is an isomorphism between the group of virtual support functions on $\Sigma$ and the group
of $\TT$-Cartier divisors on $\XS$. The divisors $D_{\Psi_1}$ and $D_{\Psi_2}$ are rationally equivalent if and only if $\Psi_1-\Psi_2$ is linear.
\end{art}

\begin{defn}
Let $X$ be a toric variety. A \emph{toric line bundle} on $X$ is a pair $(L,z)$ consisting of a line bundle $L$ on $X$ and a non-zero element $z$ in the fiber
$L_{x_0}$ of the unit point $x_0$ of $U_0=\TT$.
A \emph{toric section} is a rational section $s$ of a toric line bundle which is regular and non-vanishing on the torus $U_0$ and such that $s(x_0)=z$.
\end{defn}


\begin{art}\label{CarSup}
Let $D$ be a $\TT$-Cartier divisor on a normal toric variety $X_\Sigma$. Then there is an associated line bundle $\OO(D)$ and a rational section $s_D$ such that $\dvs(s_D)=D$.
Since the support of $D$ lies in the complement of $\TT$, the section $s_D$ is regular and non-vanishing on $\TT$. Thus, $D$ corresponds to
a toric line bundle $(\OO(D),s_D(x_0))$ with toric section $s_D$.
This assignment determines an isomorphism between
the group of $\TT$-Cartier divisors on $X_\Sigma$ and the group of isomorphism classes of toric line bundles with toric sections.

Let $\Psi$ be a virtual support function on $\Sigma$. 
This function corresponds bijectively to the
isomorphism class of the toric line bundle with toric section $((\OO(D_\Psi),s_{D_\Psi}(x_0)),s_{D_\Psi})$, which we simply denote by
$(L_\Psi,s_\Psi)$.

Note that a line bundle with a toric section admits a unique structure of a $\TT$-equivariant line bundle such that the toric section becomes $\TT$-invariant. Conversely, any $\TT$-equivariant toric line bundle has a unique toric section which is $\TT$-equivariant (see \cite[Remark 3.3.6]{BPS}).
\end{art}

\begin{art}\label{torseq}
Let $\XS$ be a $\TT$-toric variety. We denote by $\pic(\XS)$ the Picard group of $\XS$ and by $\DivT(\XS)$ the group of $\TT$-Cartier divisors. Then we have a short exact sequence of abelian groups
\[M \longrightarrow \DivT(\XS)\longrightarrow \pic(\XS)\longrightarrow 0,\]
where the first morphism is given by $m \mapsto \dvs(\chi^m)$.
In particular, every toric line bundle admits a toric section and, if $s$ and $s'$ are two toric sections, then there is an $m\in M$ such that $s'=\chi^ms$.
\end{art}

\begin{art}\label{Weildiv}
Let $D_\Psi$ be a $\TT$-Cartier divisor on a normal toric variety $X_\Sigma$.
Then the associated Weil divisor $\cyc(s_\Psi)$ is invariant under the torus action.
Indeed, let $\Sigma^{(1)}$ be the set of one-dimensional cones in $\Sigma$.
Each ray $\tau\in \Sigma^{(1)}$ gives a minimal generator $v_\tau \in \tau \cap N$ and a corresponding $\TT$-invariant prime divisor $V(\tau)$ on $X_\Sigma$ (see \ref{orbit-cone}).
Then we have 
\begin{align}\label{Weild}
\cyc(s_\Psi)=\sum_{\tau\in\Sigma^{(1)} }{-\Psi(v_\tau)V(\tau)}.
\end{align}
\end{art}

\begin{art}\label{inter-clos}
We describe the intersection of a $\TT$-Cartier divisor with the closure of an orbit.
Let $\Sigma$ be a fan in $\NR$ and $\Psi$ a virtual support function on $\Sigma$ given by the defining vectors $\{m_\tau\}_{\tau\in\Sigma}$.
Let $\sigma$ be a cone of $\Sigma$ and $V(\sigma)$ the corresponding orbit closure.
Each cone $\tau\succeq\sigma$ corresponds to a cone $\overline{\tau}$ of the fan $\Sigma(\sigma)$ defined in (\ref{sigmafan}).
Since $m_\tau-m_\sigma|_\sigma=0$, we have $m_\tau-m_\sigma\in M(\sigma)=M\cap\sigma^\bot$.
Thus, the defining vectors $\{m_\tau-m_\sigma\}_{\overline{\tau}\in\Sigma(\sigma)}$ gives us
a well-defined virtual support function $(\Psi-m_\sigma)(\sigma)$ on $\Sigma(\sigma)$.

When $\Psi|_\sigma\neq 0$, then $D_\Psi$ and $V(\sigma)$ do not intersect properly. 
But $D_\Psi$ is rationally equivalent to $D_{\Psi-m_\sigma}$ and the latter divisor properly intersects $V(\sigma)$.
Moreover, we have $D_{\Psi-m_\sigma}|_{V(\sigma)}=D_{\left(\Psi-m_\sigma\right)(\sigma)}$.
For details, we refer to \cite[Proposition 3.3.14]{BPS}.
\end{art}

We end this subsection with some positivity statements about $\TT$-Cartier divisors. For this,
we consider a complete fan $\Sigma$ in $\NR$ and a virtual support function $\Psi$ on $\Sigma$ given by the defining vectors $\{m_\sigma\}_{\sigma\in \Sigma}$.
\begin{art} \label{pos-cart}
Many properties of the associated toric line bundle $\OO(D_\Psi)$ are encoded in its (virtual) support function.
\begin{enumerate}
\item $\OO(D_\Psi)$ is generated by global sections if and only if $\Psi$ is concave;
\item $\OO(D_\Psi)$ is ample if and only if $\Psi$ is strictly concave.
\end{enumerate}

If $\Psi$ is concave, then the stability set $\Delta_\Psi$ from \ref{a5} is a lattice polytope 
and $\{\chi^m\}_{m\in M\cap\Delta_\Psi}$ is a basis of the $K$-vector space
$\Gamma(X_\Sigma,\OO(D_\Psi))$.
Moreover, we have in this case
\begin{align}\label{deg-vol}
\deg_{\OO(D_\Psi)}(X_\Sigma)=n!\vol_M(\Delta_\Psi).
\end{align}
\end{art}

\begin{art}\label{Weild2}
Assume that $\Psi$ is strictly concave or equivalently that $D_\Psi$ is ample.
We use the notations and statements from \ref{pol-fan}.
Then the stability set $\Delta\coloneq\Delta_\Psi$ is a full dimensional lattice polytope 
and $\Sigma$ coincides with the normal fan $\Sigma_{\Delta}$ of $\Delta$.
Thus, a facet $F$ of $\Delta$ correspond to a ray $\sigma_F$ of $\Sigma$
and we can reformulate (\ref{Weild}),
\[\cyc(s_\Psi)=\sum_F{-\la F, v_F \ra V(\sigma_F)},\]
where the sum is over the facets $F$ of $\Delta$ and where $v_F$ is the primitive inner normal vector to $F$ (see \ref{a9}).
\end{art}

\begin{art}\label{projtor}
Assume that $\Psi$ is concave, i.e.\ $D_\Psi$ is generated by global sections. Then $\Delta=\Delta_\Psi$ is a (not necessarily full dimensional) lattice polytope.
We set
\[M(\Delta)=M\cap\Lb_\Delta,\quad N(\Delta)=M(\Delta)^\vee=N\left/\bigl(N\right.\cap\Lb_\Delta^\bot\bigr),\]
where $\Lb_\Delta$ denotes the linear subspace of $\MRR$ associated to the affine hull $\aff(\Delta)$ of $\Delta$.
We choose any $m\in \aff(\Delta)\cap M$.
Then, the lattice polytope $\Delta-m$ is full dimensional in $\Lb_\Delta=M(\Delta)_\RR$.
Let $\Sigma_\Delta$ be the normal fan of $\Delta-m$ in $N(\Delta)_\RR$ (see \ref{pol-fan}).
The projection $H\colon N\rightarrow N(\Delta)$ is compatible with $\Sigma$ and $\Sigma_\Delta$ and
so, by \ref{tormor-com}, it induces  a proper toric morphism $\varphi\colon X_\Sigma\rightarrow X_{\Sigma_\Delta}$.
We set $\Delta'=\Delta-m$ and consider the function \[ \Psi_{\Delta'}\colon N(\Delta)_\RR\longrightarrow \RR,\quad u\longmapsto \min_{m'\in\Delta'}{\la m',u\ra}.\]
This is a strictly concave support function on $\Sigma_\Delta$.
By \ref{pos-cart}, the divisor $D_{\Psi_{\Delta'}}$ is ample and,
\begin{align}\label{cartpull}
D_\Psi=\varphi^*D_{\Psi_{\Delta'}}+\dvs(\chi^{-m}).
\end{align}

\end{art}

\subsection{Toric schemes over valuation rings of rank one} \label{Toric Schemes over Valuation Rings of Rank One}

In this subsection we summarize some facts from the theory of toric schemes over valuation rings of rank one developed in \cite{Guide} and \cite{GuSo}.

Let $K$ be a field equipped with a non-archimedean absolute value $|\cdot |$. Then we have a valuation ring $\Kval$ with valuation $\val\coloneq -\log |\cdot |$ of rank one and a value group
 $\Gamma \coloneq \val(K^{\times})$.
As usual, we fix a free abelian group $M$ of rank $n$ with dual $N$.
Then we denote by $\TT_S$ the split torus $\TT_S=\spec{(\Kval[M])}$ over $S=\spec(\Kval)$ with generic fiber $\TT=\spec(K[M])$
and special fiber $\TT_{\tilde{K}}=\spec(\tilde{K}[M])$ over the residue field $\tilde K$.

\begin{defn}
A ($\TT_S$-)\emph{toric scheme} is an integral separated $S$-scheme $\XXX$ such that
the generic fiber $\XXX_{\eta}$ contains $\TT$ as an open subset and the translation action of $\TT$ on itself extends to an algebraic action $\TT_S \times_{S}\XXX\rightarrow \XXX$ over $S$.
\end{defn}

Note that by \cite[Lemma 4.2]{Guide} a toric scheme $\XXX$ is flat over $S$ and the generic fiber $\XXX_\eta$ is a $\TT$-toric variety over $K$.


\begin{defn}
Let $X$ be a $\TT$-toric variety and $\XXX$ a $\TT_S$-toric scheme. Then $\XXX$ is called a ($\TT_S$-)\emph{toric model} of $X$ if  it comes with a  fixed isomorphism $\XXX_\eta \simeq X$ which is the identity on $\TT$.
If $\XXX$ and $\XXX'$ are toric models of $X$ and $\alpha\colon\XXX\rightarrow \XXX'$ is an $S$-morphism, we say that $\alpha$ is a \emph{morphism of toric models} if its restriction
to $\TT$ is the identity.
\end{defn}



\begin{art} \label{admissible cones and toric schemes}
A \emph{$\Gamma$-admissible cone} $\sigma$ in $\NRR$ is a strongly convex cone which is of the form
\[\sigma =\bigcap_{i=1}^k{\left\{(u,r)\in \NRR \mid \la m_i,u\ra+l_i \cdot r \geq 0 \right\}} \ \text{ with }  m_i \in M,\,l_i\in \Gamma,i=1,\dotsc,k .\]
For such a cone $\sigma$, we define
\[K[M]^{\sigma}\coloneq \Bigl\{\,\sum_{m\in M}{\alpha_m \chi^m \in K[M]\mid \la m,u\ra+ \val(\alpha_m)\cdot r \geq 0 \ \forall (u,r)\in \sigma}\Bigr\}.\]
This is an $M$-graded $\Kval$-subalgebra of $K[M]$ which is an integrally closed domain. 
Hence, we get
an affine  normal $\TT_S$-toric scheme $\UUU_{\sigma}\coloneq \spec(K[M]^{\sigma})$ over $S$. 
It is finitely generated as a $\Kval$-algebra if and only if
the following condition (F) is fulfilled:
\begin{itemize}
\item[(F)] The value group $\Gamma$ is discrete or the
vertices of $\sigma\cap(\NR\times \{1\})$ are contained in $N_{\Gamma}\times\{1\}$.
\end{itemize}
Hence  $\UUU_{\sigma}$ is of finite type 
if and only if (F) holds. If $\Gamma$ is discrete or divisible, then (F) is always satisfied.
\end{art}


\begin{art}\label{admissible}
A fan in $\NRR$ is called \emph{$\Gamma$-admissible} if it consists of $\Gamma$-admissible cones.
Given such a fan $\SSigma$, the affine $\TT_S$-toric schemes $\UUU_{\sigma}, \sigma\in\SSigma$, glue together along the open subschemes
corresponding to the common faces as in the case of toric varieties. This leads to a normal $\TT_S$-toric scheme 
\begin{equation}\label{opcov}
\XXX_{\SSigma}=\bigcup_{\sigma\in\SSigma}{\UUU_{\sigma}}
\end{equation}
over $S$. It is universally closed if and only if $\SSigma$ is complete, i.\,e. $|\SSigma|=\NRR$ (see \cite[Proposition 11.8]{Guide}).
\end{art}
 
\begin{art} \label{sumihiro}
We have the following generalization of the classification of toric varieties over a field: 
By \cite[Theorem 3]{GuSo}, $\SSigma\mapsto \XXX_{\SSigma}$ defines a bijection between the sets of 
\begin{enumerate}
\item $\Gamma$-admissible fans in $\NRR$ whose cones satisfy condition (F),
\item isomorphism classes of normal $\TT_S$-toric schemes of finite type over $S$.
\end{enumerate}
In this case, $\XXX_{\SSigma}$ is proper over $S$ if and only if $\SSigma$ is complete.
\end{art}

\begin{art}
It is also possible to describe toric schemes in terms of polyhedra in $\NR$.
Let $\sigma$ be a cone in $\NRR$. For $r\in\RR_{\geq 0}$, we set 
\[\sigma_r\coloneq \{u\in \NR\mid  (u,r)\in \sigma\}.\]
Then $\sigma\mapsto \sigma_1$ defines a bijection between the set of $\Gamma$-admissible cones in $\NRR$, which are not contained in $\NR \times \{0\}$,
and the set of strongly convex $\Gamma$-rational polyhedra in $\NR$.
The inverse map is given by $\Lambda \mapsto \cone(\Lambda)$, where $\cone(\Lambda)$ is the closure of $\RR_{>0}(\Lambda\times \{1\})$ in $\NRR$.
\end{art}

\begin{art} \label{main example of toric scheme}
Let $\SSigma$ be a $\Gamma$-admissible fan.
Then we have two kinds of cones $\sigma$ in $\SSigma$:
\begin{enumerate}
\item If $\sigma$ is contained in $\NR\times\{0\}$, then $K[M]^{\sigma}=K[M_{\sigma_0}]$. Hence, $\UUU_{\sigma}$ is equal to the toric variety $U_{\sigma_0}$ associated to the cone $\sigma_0$ (see \ref{afftor})
and it is contained in the generic fiber of $\XXX_{\SSigma}$.
\item
If $\sigma$ is not contained in $\NR\times\{0\}$, then $\Lambda\coloneq \sigma_1$ is a strongly convex $\Gamma$-rational polyhedron in $\NR$.
It easily follows that $K[M]^{\sigma}$ is equal to
\[K[M]^{\Lambda}\coloneq \Bigl\{\,\sum_{m\in M}{\alpha_m \chi^m \in K[M]\mid \la m,u\ra+ \val(\alpha_m) \geq 0 \ \forall u \in \Lambda}\Bigr\}.\]
Thus, $\UUU_{\sigma}$ equals the $\TT_S$-toric scheme $\UUU_{\Lambda}\coloneq\spec(K[M]^{\Lambda})$.
The generic fiber of $\UUU_{\Lambda}=\UUU_{\sigma}$ is identified with the $\TT$-toric variety $U_{\sigma_0}=U_{\rec(\Lambda)}$.
\end{enumerate}
We set $\Sigma\coloneq\{\sigma_0\,|\, \sigma \in \SSigma\}$ and $\Pi\coloneq\{\sigma_1\,|\, \sigma \in \SSigma\}$.
Then $\Sigma$ is a fan in $\NR$ and $\Pi$ is a $\Gamma$-rational strongly convex polyhedral complex in $\NR$.
Now we can rewrite the open cover \eqref{opcov} as
\begin{equation} \label{horizontal and vertical U}
\XXX_{\SSigma}=\bigcup_{\sigma \in \Sigma}{U_\sigma}\cup\bigcup_{\Lambda\in\Pi}{\UUU_\Lambda}
\end{equation}
using the same gluing data.
The generic fiber of this toric scheme is the $\TT$-toric variety $X_{\Sigma}$ associated to $\Sigma$,
i.\,e. $\XXX_{\SSigma}$ is a toric model of $X_{\Sigma}$.
\end{art}

\begin{art}\label{pol-scheme}
If the value group $\Gamma$ is discrete, then the special fiber $\XXX_s$ is reduced for $\XXX \coloneq \XXX_{\SSigma}$ if and only if the vertices of all $\Lambda \in\Pi$ are contained in $N_\Gamma$.
If the valuation is not discrete, then ${\XXX}_s$ is always reduced (see \cite[Proposition 7.11 and 7.12]{Guide}).
\end{art}

\begin{art}\label{fanpol}
Conversely, if we start with an arbitrary $\Gamma$-rational strongly convex polyhedral complex $\Pi$, we can't expect that
\[\cone(\Pi)\coloneq \{\cone({\Lambda})\mid \Lambda\in\Pi\}\cup \{\rec(\Lambda)\times\{0\}\mid  \Lambda\in\Pi\}\]
is a fan in $\NRR$. Burgos and Sombra have shown that the correspondence $\Pi \mapsto \cone(\Pi)$ gives a bijection between \textit{complete} $\Gamma$-rational strongly convex polyhedral complexes
in $\NR$ and \textit{complete} $\Gamma$-admissible fans in $\NRR$ (see \cite[Corollary 3.11]{BS}). We set $\XXX_\Pi\coloneq\XXX_{\cone(\Pi)}$ and we identify the generic fiber $\XXX_{\Pi,\eta}$ with the toric variety $X_{\rec(\Pi)}$.
\end{art}

We end this subsection with a description of the orbits of a toric scheme.
As in \ref{main example of toric scheme}, we consider a  $\Gamma$-admissible fan $\SSigma$ with associated fan $\Sigma\coloneq\{\sigma_0\,|\, \sigma \in \SSigma\}$ in $\NR$ and with associated polyhedral complex $\Pi\coloneq\{\sigma_1\,|\, \sigma \in \SSigma\}$ in $\NR$.

\begin{nota}\label{M(L)}
For $\Lambda\in \Pi$, let $\Lb_\Lambda$ be the $\RR$-linear subspace of $\NR$ associated to the affine space $\aff(\Lambda)$.
We set
\begin{align*}
N(\Lambda)=N/{( N\cap \Lb_\Lambda)},\quad M(\Lambda)=N(\Lambda)^\vee=M\cap \Lb_\Lambda^\bot,
\end{align*}
generalizing the notation in \ref{orbit-cone}.
Furthermore, we define
\begin{align*}
\widetilde{M}(\Lambda)=\{m\in M(\Lambda)\mid\la m,u \ra\in\Gamma\ \forall\, u\in \Lambda\},\quad 
\widetilde{N}(\Lambda)=\widetilde{M}(\Lambda)^\vee \,.
\end{align*}
Because of the $\Gamma$-rationality of $\Lambda$, the lattice $\widetilde{M}(\Lambda)$
is of finite index in $M(\Lambda)$.
We define the \emph{multiplicity} of a polyhedron $\Lambda\in\Pi$ by
\begin{align}\label{mult}
\mult(\Lambda)=\bigl[M(\Lambda):\widetilde{M}(\Lambda)\bigr].
\end{align}
Let $\Lambda'\in\Pi$ and $\Lambda$ a face of $\Lambda'$. The \emph{local cone} (or \emph{angle}) of $\Lambda'$ at $\Lambda$ is defined as
\[ \angle(\Lambda,\Lambda')\coloneq\left\{t(u-v) \mid u\in\Lambda',v\in\Lambda,t\geq 0\right\}.\]
This is a polyhedral cone.
\end{nota}

There is a bijection between torus orbits of $\XXX_{\SSigma}$ and the two kinds of cones in $\SSigma$ corresponding to cones in $\Sigma$ as well as polyhedra in $\Pi$. 

First, the cones in $\Sigma$ correspond to the $\TT$-orbits on the generic fiber $X_\Sigma$ via 
$\sigma\mapsto O(\sigma)$ as in \ref{orbit-cone}.
We denote by $\VVV(\sigma)$ the Zariski closure of $O(\sigma)$ in $\XXX_{\SSigma}$.
Then $\VVV(\sigma)$ is a scheme of relative dimension $n-\dim(\sigma)$ over $S$.
Moreover, we have $\tau\preceq \sigma$ if and only if $O(\sigma)\subseteq \VVV(\tau)$.

\begin{prop} \label{orbit1}
Using the notation above, there is a canonical isomorphism from $\VVV(\sigma)$ to the $\spec(\Kval[M(\sigma)])$-toric scheme  over $\Kval$ associated to the $\Gamma$-admissible fan $\SSigma(\sigma):=\{(\pi_\sigma \times {\rm id}_{\RR_{\geq 0}})(\nu) \mid \nu \in \Sigma, \, \nu \supset \sigma \}$ in $N(\sigma)_\RR \times \RR_{\geq 0}$.
The associated   polyhedral complex $\Pi(\sigma)$ in $N(\sigma)_{\RR}=\NR/\la N\cap \sigma\ra_{\RR}$ is given by 
\[\Pi(\sigma)=\{\Lambda+\la N\cap \sigma\ra_{\RR}\mid\Lambda\in \Pi,\,\rec(\Lambda)\supseteq \sigma\}.\]
\end{prop}
\begin{proof}This follows from \cite[Proposition 7.14]{Guide}.
\end{proof}

The polyhedra of $\Pi$ correspond to the $\TT_{\tilde{K}}$-orbits on the special fiber of $\XXX_{\SSigma}$. This bijective correspondence
is given by 
\[O\colon\Lambda \longmapsto \red(\trop^{-1}(\ri{\Lambda})),\]
where $\red$ is the reduction map from \ref{f-alg}, $\trop$ is the tropicalization map from \ref{trop} and $\ri(\Lambda)$ is
the relative interior of $\Lambda$ from \ref{a1}. 
For details, we refer to \cite[Proposition 6.22 and 7.9]{Guide}.
For $\Lambda\in\Pi$, we denote by 
$V(\Lambda)$ the Zariski closure of $O(\Lambda)$ in $\XXX_{\SSigma}$.
Then $V(\Lambda)$ is contained in the special fiber of $\XXX_{\SSigma}$ 
and has dimension $n-\dim(\Lambda)$.
For $\Lambda, \Lambda' \in \Pi$ and $\sigma \in \Sigma$, we have
\begin{align}\label{adjacency}
\Lambda\preceq \Lambda'\ \Longleftrightarrow\ O(\Lambda')\subseteq V(\Lambda)\quad\text{and}\quad\sigma\preceq \rec(\Lambda)\ \Longleftrightarrow \ O(\Lambda)\subseteq \VVV(\sigma)\ .
\end{align}

\begin{prop} \label{orbit2}
The variety $V(\Lambda)$ is equivariantly (but non-canonically) isomorphic to the 
$\spec(\tilde{K}[\widetilde{M}(\Lambda)])$-toric variety over $\tilde{K}$
associated to the fan
\begin{align}\label{pi(lambda)}
\Pi(\Lambda)=\{\angle(\Lambda,\Lambda')+\Lb_\Lambda\mid \Lambda'\in\Pi,\, \Lambda'\supseteq \Lambda\}
\end{align}
in $\widetilde{N}(\Lambda)_\RR=N(\Lambda)_\RR=\NR/\la N\cap \Lb_\Lambda \ra_{\RR}$.
\end{prop}
\begin{proof}This is \cite[Proposition 7.15]{Guide}.
\end{proof}
\begin{art}\label{ver-irr}
In particular, there is a bijection between vertices of $\Pi$ and the irreducible components of the special fiber of $\XXX_{\SSigma}$.
For each $v\in\Pi^0$, the component $V(v)$ is a toric variety over $\widetilde{K}$
with torus associated to the character lattice
$\{m\in M\mid \la m,v \ra\in \Gamma\}$
and
given by the fan
$\Pi(v)=\left\{\RR_{\geq0}(\Lambda'-v)\mid \Lambda'\in\Pi,\Lambda' \ni v\right\}$
in $\NR$.
\end{art}

\subsection{Toric Cartier divisors on toric schemes} \label{Toric Cartier Divisors on Toric Schemes}

We extend the theory of $\TT$-Cartier divisors to the case of toric schemes over a valuation ring
of rank one.
This generalizes \cite[§\,3.6]{BPS} where the case of discrete valuation is handled and which we use
as a guideline.
We keep the notations of the previous subsection.

\begin{defn}
A \emph{$\TT_S$-Cartier divisor} on a $\TT_S$-toric scheme $\XXX$ is a Cartier divisor $D$ on $\XXX$ which is invariant under the action of $\TT_S$ on $\XXX$, i.\,e. we have $\mu^*D=p_2^*D$ denoting by $\mu\colon\TT_S\times \XXX\rightarrow \XXX$ the toric action and by $p_2\colon \TT_S\times \XXX\rightarrow \XXX$ the second projection.
\end{defn}

In the following, we consider the $\TT_S$-toric scheme $\XXX$ given by a $\Gamma$-admissible fan $\SSigma$ in $\NR \times \RR_{\geq 0}$ with associated fan $\Sigma\coloneq\{\sigma_0\,|\, \sigma \in \SSigma\}$ in $\NR$ and with associated polyhedral complex $\Pi\coloneq\{\sigma_1\,|\, \sigma \in \SSigma\}$ in $\NR$ as in \ref{main example of toric scheme}.

\begin{art}\label{tcart}
A  function $\psi:|\SSigma| \to \RR$ is called  a  {\it $\Gamma$-admissible virtual support function on  $\SSigma$} if for every $\sigma \in \SSigma$, there is $m_\sigma \in M$ and $l_\sigma \in \Gamma$ such that 
\begin{equation} \label{admissible virtual support function}
\psi(u,r)= \langle m_\sigma, u \rangle +  l_\sigma r 
\end{equation}
for all $(u,r) \in \sigma$. By restriction to $\SSigma_0=\Sigma$ (resp.\ $\SSigma_1=\Pi$), we get a virtual support function $\psi_0$ on $\Sigma$ (resp.\ a $\Gamma$-rational piecewise affine  function $\psi_1$ on $\Pi$). 

Conversely, suppose that we have $(m_\sigma,l_\sigma) \in M \times \RR$ for every $\sigma \in \SSigma$ satisfying 
the condition
\begin{align}\label{glfunc}
\la m_\sigma,u\ra+ l_\sigma r =\la m_{\sigma'},u\ra+ l_{\sigma'}r \quad \text{for all }(u,r) \in \sigma \cap \sigma' \text{ and all }\sigma,\sigma'\in \SSigma.
\end{align}
Then it is clear that \eqref{admissible virtual support function} defines a  $\Gamma$-admissible virtual support function on  $\SSigma$. 

On each open subset $\UUU_\sigma$, the vector $(m_\sigma,l_\sigma)$ determines a rational function $\alpha_\sigma^{-1}\chi^{-m_\sigma}$, where $\alpha_\sigma\in K^\times$ is any
element with $\val(\alpha_\sigma)=l_\sigma$.
For $\sigma,\sigma'\in \SSigma$, condition (\ref{glfunc}) implies that this function 
is regular and non-vanishing on $\UUU_\sigma\cap \UUU_{\sigma'}=\UUU_{\sigma\cap\sigma'}$.
By construction $\{\UUU_\sigma\}_{\sigma\in \SSigma}$ is an open covering of $\XXX$.
Thus, $\psi$ defines a Cartier divisor
\begin{align}\label{cart}
D_\psi=\left\{\left(\UUU_\sigma,\alpha_{\sigma}^{-1}\chi^{-m_\sigma}\right)\right\}_{\sigma\in\SSigma},
\end{align}
where $\alpha_{\sigma}\in K^\times$ is any element with $\val(\alpha_{\sigma})=l_\sigma$.
The divisor $D_\psi$ only depends on $\psi$ and not on the particular choice of the defining vectors and elements $\alpha_\sigma$. It is easy to see that $D_\psi$ is $\TT_S$-invariant.
\end{art}

\begin{thm}\label{class-cart1}
Let $\SSigma$ be a $\Gamma$-admissible fan  in $\NR \times \RR_{\geq 0}$ with corresponding
 $\TT_S$-toric scheme  $\XXX=\XXX_{\SSigma}$.
\begin{enumerate}
\item The assignment $\psi\mapsto D_\psi$ is an isomorphism between the group of
$\Gamma$-admissible virtual support functions on $\SSigma$ and the group of $\TT_S$-Cartier divisors on $\XXX$.
\item The divisors $D_{\psi}$ and $D_{\psi'}$ are rationally equivalent if and only if $\psi'-\psi$
is affine.
\end{enumerate}
\end{thm}

For the proof, we need the following helpful lemma.

\begin{lemma}\label{affcart}
Let $\sigma$ a $\Gamma$-admissible cone in $\NR \times \RR_{\geq 0}$. Then for each $\TT_S$-Cartier divisor $D$ on $\UUU_\sigma$ we have
\[D=\dvs(\alpha\chi^m)\]
for some $m\in M$ and $\alpha\in K^\times$.
\end{lemma}

\begin{proof}
Using the decomposition \eqref{horizontal and vertical U}, we may assume that $\UUU_\sigma = \UUU_\Lambda$ for $\Lambda \in \Pi$ with $\cone(\Lambda)=\sigma$ as the 
case $\sigma \in \Sigma$ is well-known. 
Let us consider the $\Kval$-algebra $A\coloneq \OO_{\UUU_\Lambda}(\UUU_\Lambda)=K[M]^\Lambda$
and the fractional ideal $I\coloneq \Gamma(\UUU_\Lambda,\OO_{\UUU_\Lambda}(-D))$ of $A$.
Since $D$ is $\TT_S$-invariant, the $\Kval$-module $I$ is graded by $M$, i.\,e. we can write $I=\bigoplus_{m\in M}{I_m}$,
where $I_m$ is a $\Kval$-submodule contained in $K\chi^m$.
Because $\Kval$ is  a valuation ring of rank one,
either $I_m=(0)$ or $I_m=\Kmax\alpha_m\chi^m$ or $I_m=\Kval \alpha_m\chi^m$ or $I_m=K\chi^m$ for some $m\in M$, $\alpha_m \in K^\times$.
Since $I$ is finitely generated as an $A$-module, we deduce
\begin{align}\label{Mgrad}
I=\bigoplus_{\alpha_m\chi^m\in I}{\Kval \alpha_m\chi^m}\ .
\end{align}

Now we fix a point $p\in O(\Lambda)$. Then $D$ is principal on an open neighborhood $U$ of $p$ in $\UUU_\Lambda$.
We may assume that $U=\spec(A_h)$ for some $h\in A$ with $h(p)\neq 0$.
Hence, $D|_U=\dvs(f)|_U$ for some $f\in K(M)^\times=\quot(A)^\times$.
This implies
\[I_h=\OO_{\UUU_\Lambda}(-D)(U)=f\cdot\OO_{\UUU_\Lambda}(U)=f\cdot A_h\ .\]
In particular, $f\in I_h$, and by (\ref{Mgrad}), we can write
\[
f=\sum_i \frac{c_i}{h^k}\alpha_{m_i}\chi^{m_i} \quad \text{with } c_i\in \Kval\setminus \{0\}, k\in \NN.
\]
Since $\alpha_{m_i}\chi^{m_i}/f\in \OO_{\UUU_\Lambda}(U)$ and $p\in U$, we deduce $\left(\alpha_{m_i}\chi^{m_i}/f\right)(p)\neq 0$ for some $i$.
There exists an open neighborhood $W\subseteq U$ of $p$ on which $\alpha_{m_i}\chi^{m_i}/f$ is non-vanishing and thus,
\begin{align}\label{retoW}
\dvs(\alpha_{m_i}\chi^{m_i})|_W=\dvs(f)|_W=D|_W.
\end{align}

By \cite[Corollary 2.12\,(c)]{GuSo}, we have an injective homomorphism $D\mapsto \cyc(D)$ from the group of Cartier divisors on $\UUU_\Lambda$ to the group
of Weil divisors on $\UUU_\Lambda$, which restricts to a homomorphism of the corresponding groups of $\TT_S$-invariant divisors.
The $\TT_S$-invariant prime (Weil) divisors are exactly the $\TT_S$-orbit closures of codimension one.
By (\ref{adjacency}), 
\[p\in O(\Lambda)\subseteq \bigcap_{\substack{v\in \Pi^0 ,\\ v\preceq \Lambda}}{V(v)}\cap\!\bigcap_{\substack{\tau\in \rec(\Pi)^1, \\ \tau\preceq \rec(\Lambda)}}{\VVV(\tau)},\]
and therefore, $W$ meets each $\TT_S$-invariant prime divisor of $\UUU_\Lambda$.
Thus, equation (\ref{retoW}) implies $\cyc(D)=\cyc\left(\dvs(\alpha_{m_i}\chi^{m_i})\right)$ and hence $D=\dvs(\alpha_{m_i}\chi^{m_i})$.
\end{proof}

\begin{proof}[Proof of Theorem \ref{class-cart1}]
(i) Let $\sigma$ be a $\Gamma$-admissible virtual support function  on $\SSigma$ given by defining vectors $\left\{\left(m_\sigma,\val(\alpha_\sigma)\right)\right\}_{\sigma\in \SSigma}$. 
Then, by the construction in \ref{tcart}, $D_\psi$ is a well-defined $\TT_S$-Cartier divisor on $\XXX$.
It is easy to see that this assignment defines a group homomorphism.

To prove injectivity, we assume that $\psi$ maps to the zero divisor $(\XXX,1)$. Then, for each $\sigma\in \SSigma$, the function
$\alpha_\sigma^{-1}\chi^{-m_\sigma}$ is invertible on $\UUU_\sigma$ or equivalently,
\[\psi(u,r)=\la m_\sigma,u \ra+\val(\alpha_\sigma)r=0\quad\text{for all }(u,r)\in\sigma.\]
Therefore, $\psi$ is identically zero and we proved the injectivity.

For surjectivity, let $D$ be an arbitrary $\TT_S$-Cartier divisor on $\XXX$. By Lemma \ref{affcart}, there exist, for each $\sigma\in\SSigma$, elements $\alpha_\sigma\in K^\times$ and
$m_\sigma\in M$, such that $D|_{\UUU_\sigma}=\dvs(\alpha_\sigma\,\chi^{m_\sigma})|_{\UUU_\sigma}$.
Since $D$ is a Cartier divisor, we have, for $\sigma,\sigma'\in\SSigma$,
\[\dvs(\alpha_\sigma\,\chi^{m_\sigma})|_{\UUU_{\sigma\cap\sigma'}}=\dvs(\alpha_{\sigma'}\,\chi^{m_{\sigma'}})|_{\UUU_{\sigma'\cap\sigma}}\ ,\]
which implies that
\begin{align}\label{funcgl}
\val(\alpha_\sigma)r+\la m_\sigma, u\ra=\val(\alpha_{\sigma'})r+\la m_{\sigma'}, u\ra \quad \text{for all }(u,r)\in\sigma\cap\sigma'.
\end{align}
For each $\sigma\in\SSigma$, we set $\psi(u,r)\coloneq \la -m_\sigma,u\ra-\val(\alpha_\sigma)r$ for all $(u,r)\in \sigma$.
By (\ref{funcgl}), this determines a well-defined $\Gamma$-admissible virtual support function $\psi\colon \NR\rightarrow \RR$ and, by (\ref{cart}), $\psi$ maps to $D$.

(ii) We claim that a $\TT_S$-Cartier divisor on $\XXX$ is principal if and only if it has the form $\dvs(\alpha\chi^m)$ for $\alpha\in K^\times, m\in M$.
Indeed, let $D$ be any principal $\TT_S$-Cartier divisor on $\XXX$, i.\,e.\ $D=\dvs(f)$ for some $f\in K(\XXX)^\times$.
The support of $D$ is disjoint from the torus $\TT$. Therefore, when regarded as an element of $K(\TT)^\times$, the divisor of $f|_\TT$ is zero.
This implies $f\in K[M]^\times$ and thus, $f=\alpha\chi^m$ for some $\alpha\in K^\times$ and $m\in M$.

Using this equivalence, statement (ii) follows easily from (i).
\end{proof}

\begin{art} \label{toric models}
Let $\XXX$ be a toric scheme over $S$. A \emph{toric line bundle} on $\XXX$ is a pair $(\LLL,z)$ consisting
of a line bundle $\LLL$ on $\XXX$ and a non-zero element $z$ in the fiber $\LLL_{x_0}$ of the 
unit point $x_0\in\XXX_\eta$.
A \emph{toric section} is a rational section $s$ of a toric line bundle which is regular and non-vanishing
on the torus $\TT\subseteq \XXX_\eta$ and such that $s(x_0)=z$.

As in \ref{CarSup}, each $\TT$-Cartier divisor $D$ on $\XXX$ gives us a toric line bundle
$(\OO(D),s_D(x_0))$ with toric section $s_D$. Usually, we are concerned with toric schemes $\XXX:=\XXX_{\SSigma}$ associated to a $\Gamma$-admissible fan 
$\SSigma$ in $N_\RR \times \RR$. Then  
each $\Gamma$-admissible virtual support function $\psi$ on $\SSigma$ defines a toric line bundle with toric section
$((\OO(D_\psi),s_{D_\psi}(x_0)),s_{D_\psi})$, which we simply denote by
$(\LLL_\psi,s_\psi)$.

Let $(X_\Sigma,D_\Psi)$ be a  toric variety with a $\TT$-Cartier divisor.
A \emph{toric model} of $(X_\Sigma,D_\Psi)$ is a triple $(\XXX,D,e)$ consisting of a $\TT_S$-toric
model $\XXX$ of $\XS$ associated to a $\Gamma$-admissible fan $\SSigma$ in $N_\RR \times \RR_{\geq 0}$, a $\TT_S$-Cartier divisor $D$ on $\XXX$ and an integer $e>0$ such that
$D|_{\XS}=eD_\Psi$.  
\end{art}

\begin{rem} \label{class-cart}
In our applications, we consider a complete $\Gamma$-rational polyhedral complex $\Pi$ or equivalently a complete $\Gamma$-admissible cone $\SSigma = \cone(\Pi)$. As in 
\ref{fanpol}, we get a universally closed $\TT_S$-toric scheme $\XXX_\Pi \coloneq \XXX_{\cone(\Pi)}$. In this case, the map $\psi \to \psi_1$ from \ref{tcart} gives a bijective 
correspondence between $\Gamma$-admissible virtual support functions on $\SSigma$ and $\Gamma$-lattice functions on $\Pi$. \

For a $\Gamma$-lattice function $\phi$ with corresponding $\Gamma$-admissible virtual support function $\psi$, \ref{tcart} gives an associated Cartier divisor $D_\phi := D_\phi$. 
We conclude from Theorem \ref{class-cart1} that $\phi \mapsto D_\psi$ is an isomorphism from the group of $\Gamma$-lattice functions on $\Pi$ onto the group of $\TT$-Cartier divisors on $\XXX_\Pi$. Moreover, the Cartier divisors $D_{\phi}$ and $D_{\phi'}$ are linearly equivalent if and only if $\phi'-\phi$ is affine.
\end{rem}

\begin{thm}\label{class-tor}
Suppose that the value group $\Gamma$ is discrete or that $\Gamma$ is divisible. 
Let $\Sigma$ be a complete fan in $\NR$ and $\Psi$ a virtual support function on $\Sigma$.
Then the assignment $(\Pi,\phi)\mapsto (\XXX_\Pi,D_\phi)$ gives a bijection between the sets of
\begin{enumerate}
\item pairs $(\Pi,\phi)$, where $\Pi$ is a complete $\Gamma$-rational polyhedral complex in $\NR$
 and $\rec(\Pi)=\Sigma$, and $\phi$ is a $\Gamma$-lattice function on $\Pi$ with $\rec(\phi)=\Psi$;
\item isomorphism classes of proper normal toric models $(\XXX,D,1)$ of $(X_\Sigma,D_\Psi)$.
\end{enumerate}
\end{thm}

\begin{proof}
Let $(\Pi,\phi)$ be a pair as in (i) and let $\{(m_\Lambda,\val(\alpha_\Lambda))\}_{\Lambda\in\Pi}$ be defining vectors of $\phi$.
Then
\[D_\phi|_{X_\Sigma}=\bigl.\bigl\{\bigl(\UUU_\Lambda,\alpha_\Lambda^{-1}\chi^{-m_\Lambda}\bigr)\bigr\}\bigr|_{X_{\rec(\Pi)}}=\bigl\{\bigl(U_{\rec(\Lambda)},\chi^{-m_\Lambda}\bigr)\bigr\}
=D_{\rec(\phi)}=D_\Psi\ .\]
Hence, $(\XXX_\Pi,D_\phi,1)$ is a toric model of $(X_\Sigma, D_\Psi)$. 
The statement follows from \ref{sumihiro} and Theorem \ref{class-cart1}.
\end{proof}

Now we describe the restriction of $\TT_S$-Cartier divisors to closures of orbits.
But we are only interested in the case of orbits lying in the special fiber. The other case
can be handled analogously to \cite[Proposition 3.6.12]{BPS}.

Let $\SSigma$ be a $\Gamma$-admissible fan in $\NR \times \RR_{\geq 0}$ with associated fan $\Sigma \coloneq \SSigma_0$ in $\NR$ and associated polyhedral complex $\Pi\coloneq\SSigma_1$  in $\NR$.  
Let $\psi$ be a $\Gamma$-admissible virtual support function on $\SSigma$ and let $D_\psi$ be the associated $\TT_S$-Cartier divisor. We also consider the associated $\Gamma$-lattice function $\phi:=\psi_1$ on $\Pi$.

Let $\Lambda\in\Pi$ be a polyhedron. Then we have $\phi(u)=\la m_\Delta, u \ra + l_\Delta$ on $\Delta$ for some $m_\Delta \in M$ and $l_\Delta \in \Gamma$. 
We assume $\phi|_\Lambda=0$.
Using Notation \ref{M(L)} and (\ref{pi(lambda)}), we define a virtual support function $\phi(\Lambda)$ on 
the rational fan $\Pi(\Lambda)$ in
$N(\Lambda)_\RR$ given by the 
following defining vectors $\{m_\pi\}_{\pi \in\Pi(\Lambda)}$.
For each cone $\pi\in\Pi(\Lambda)$, let $\Lambda_\pi\in\Pi$ be the unique polyhedron with
$\Lambda\preceq \Lambda_\pi$ and $\angle(\Lambda,\Lambda_\pi)+\Lb_\Lambda=\pi$.
The condition $\phi|_\Lambda=0$ implies that $m_{\Lambda_\pi}\in\Lb_\Lambda^\bot$ 
and $\la m_{\Lambda_\pi}, u\ra=-l_{\Lambda_\pi}\in \Gamma$ for all $u\in \Lambda$. 
Hence, $m_{\Lambda_\pi}$ lies in $\widetilde{M}(\Lambda)$.
We set $m_\pi\coloneq m_{\Lambda_\pi}$.

\begin{prop} \label{D-lambda}
We use the above notation.
If $\phi|_\Lambda=0$, then $D_\psi$ properly intersects the orbit closure $V(\Lambda)$. Moreover,
the restriction of $D_\psi$ to $V(\Lambda)$ is the divisor $D_{\phi(\Lambda)}$.
\end{prop}
\begin{proof}
The $\TT_S$-Cartier divisor $D_\psi$ is given on $\UUU_\Lambda$ by $\alpha_\Lambda^{-1}\chi^{-m_\Lambda}$, where $\alpha_\Lambda\in K^\times$ is any element of $K^\times$ with $\val(\alpha_\Lambda)=l_\Lambda$.
If $\phi|_\Lambda=0$, then $\val(\alpha_\Lambda)+\la m_\Lambda,u\ra=0$ for all $u\in\Lambda$.
Thus, the local equation $\alpha_\Lambda^{-1}\chi^{-m_\Lambda}$ of $D_\psi$ in $\UUU_\Lambda$
is a unit in $\OO(\UUU_\Lambda)=K[M]^\Lambda$.
Hence, the orbit $O(\Lambda)\subseteq \UUU_\Lambda$ does not meet the support of $D_\psi$
and hence $V(\Lambda)$ and $D_\psi$ intersect properly.
Furthermore, 
$$D_\psi|_{V(\Lambda)}=
\left\{\left(\UUU_{\Lambda_\pi}\cap V(\Lambda),\alpha_{\Lambda_\pi}^{-1}\chi^{-m_{\Lambda_\pi}}|_{\UUU_{\Lambda_\pi}\cap V(\Lambda)}\right)\right\}_{\pi\in\Pi(\Lambda)}.$$
Using the  non-canonical isomorphism $\tilde{K}[U_\pi]\simeq \tilde{K}[\UUU_{\Lambda_\pi}\cap V(\Lambda)]$, we get 
$$D_\psi|_{V(\Lambda)}=\left\{\left(U_\pi,\chi^{-m_\pi}\right)\right\}_{\pi\in\Pi(\Lambda)}
=D_{\phi(\Lambda)}$$
proving the claim.
\end{proof}  

\begin{prop}\label{mult-deg-vol}
Let $\Pi$ be a complete $\Gamma$-rational polyhedral complex in $\NR$ with associated $\TT_S$-toric scheme $\XXX_\Pi$ and let $\phi$ be a concave
$\Gamma$-lattice function on $\Pi$ with associated $\TT_S$-Cartier divisor $D_\phi$ on $\XXX_\Pi$. Let $\Lambda\in\Pi$ be a $k$-dimensional polyhedron and $v\in\ri(\Lambda)$.
Then we have
\begin{align}\label{degvol}
\mult(\Lambda)\deg_{D_\phi}(V(\Lambda))=(n-k)!\vol_{M(\Lambda)}(\partial \phi(v)),
\end{align}  
where $\mult(\Lambda)$ is the multiplicity of $\Lambda$ (see (\ref{mult})) and
$\partial \phi (v)$ is the sup-differential of $\phi$ at $v$ (see \ref{sup-dif}) which is in fact a polytope contained in a translate of $M(\Lambda)_\RR$. 
\end{prop}
\begin{proof} It follows from Proposition \ref{orbit2} that $V(\Lambda)$ is a toric variety over $\tilde{K}$ and that the associated fan $\Pi(\Lambda)$ is complete. We conclude from \ref{proptor} that 
$V(\Lambda)$ is a proper variety and hence the degree $\deg_{D_\phi}(V(\Lambda))$ is well-defined. 
We have $\phi(u)= \la m_\Lambda, u \ra + l_\Lambda$ on $\Lambda$ for some $m_\Lambda \in M$ and $l_\Lambda \in \Gamma$.
Then $D_\phi$ is rationally equivalent to $D_{\phi-m_\Lambda-l_\Lambda}$
and $\partial(\phi-m_\Lambda-l_\Lambda)(v)=\partial\phi(v)-m_\Lambda$.
Thus, replacing $\phi$ by $\phi-m_\Lambda-l_\Lambda$ does not change both sides of equation
(\ref{degvol}) and we may assume that $\phi|_\Lambda=0$.

By Proposition \ref{D-lambda} and (\ref{deg-vol}),
\[\deg_{D_\phi}(V(\Lambda))=\deg_{D_{\phi(\Lambda)}}\bigl(X_{\Pi(\Lambda)}\bigr)
=(n-k)!\vol_{\widetilde{M}(\Lambda)}(\Delta_{\phi(\Lambda)}).\]
Using Remark \ref{piecewise affine and sup-differential}, we have $\Delta_{\phi(\Lambda)}=(\partial\phi(\Lambda))(\overline{0})=\partial \phi(v)$. We have
\[\vol_{\widetilde{M}(\Lambda)}(\Delta_{\phi(\Lambda)})
=\vol_{\widetilde{M}(\Lambda)}(\partial(\phi(\Lambda))(v))
=\frac{1}{\bigl[M(\Lambda):\widetilde{M}(\Lambda)\bigr]} \vol_{M(\Lambda)}(\partial(\phi(\Lambda))(v)),\]
and hence we get the claim.   
\end{proof}

\subsection{Metrized line bundles on toric varieties} \label{Metrized line bundles on toric varieties}

In this subsection, we recall results about toric metrics  on a toric line bundle from \cite[Chapter 4]{BPS}. 
Note that in \cite[§\,4.1--4.3]{BPS} the non-archimedean fields are not assumed to be discrete, in contrast to the rest of their Chapter $4$.

We fix the following notation.
Let $K$ be  field which is complete with respect to a non-trivial non-archimedean absolute value $|\cdot |$. Then we have a valuation $\val\coloneq -\log |\cdot |$ and a value group
 $\Gamma \coloneq \val(K^{\times})$ of rank one. 
We fix a free abelian group $M$ of rank $n$ with dual $N$ and denote by $\TT=\spec(K[M])$ the $n$-dimensional split torus over $K$.

Let $\Sigma$ be a complete fan in $\NR$ and $X_\Sigma$ the corresponding proper toric variety. Furthermore, let $\Psi$ be a virtual support function on $\Sigma$
and $(L,s)$ the associated toric line bundle with toric section.


\begin{defn} \label{tordef}
A metric $\|\cdot\|$ on $L$ is called \emph{toric} if,
for all $p,q\in \Tan$ satisfying $|\chi^m(p)|=|\chi^m(q)|$ for each $m\in M$, 
we have $\|s(p)\|=\|s(q)\|$.
\end{defn}
It easily follows from \ref{torseq} that this definition is independent of the choice of the toric section $s$.

\begin{rem}
In \cite[4.2]{BPS}, the authors study the action of the analytic group $\Tan$ on $\XS^{\an}$ and in particular, the action of the compact analytic subgroup
\[\SSS=\{p\in \Tan\mid |\chi^m(p)|=1\ \text{for all} \ m\in M\}.\]
By \cite[Proposition 4.2.15]{BPS}, we have for $p\in \Tan$,
\[\SSS\cdot p=\{q\in \Tan\mid |\chi^m(p)|=|\chi^m(q)|\ \text{for all} \ m\in M\}.\]
Hence, a metric $\|\cdot\|$ is toric if and only if the function $p\mapsto\|s(p)\|$ is invariant under the action of $\SSS$.
\end{rem}

\begin{art}\label{torification}
Given an arbitrary metric $\|\cdot\|$ on $L$, we can associate to it a toric metric 
in the following way:
For $\sigma\in \Sigma$, let $s_\sigma$ be a toric section of $L$ which is regular and non-vanishing in $U_\sigma$.
Then we set, for $p\in \Uan_\sigma$,
\[\|s_\sigma(p)\|_{\SSS}\coloneq\|s_\sigma(\tilde{p})\|,\]
where $\tilde{p}\in \Uan_\sigma$ is given by \[\sum_{m\in M_\sigma}{\alpha_m\chi^m}\longmapsto \max_m|\alpha_m||\chi^m(p)|.\]
It is easy to check that this defines a toric metric $\|\cdot\|_{\SSS}$ on $L$.
This process is called \emph{torification} of $\|\cdot\|$.
\end{art}

\begin{prop}\label{torprop}
Toric metrics are invariant under torification. Moreover, torification is multiplicative with respect to products of metrized line bundles
and continuous with respect to uniform convergence of metrics.
\end{prop}
\begin{proof} This is established in \cite[Proposition 4.3.4]{BPS} and follows easily from the definition.
\end{proof}

\begin{art}\label{trop}
We have the \emph{tropicalization map} $\trop\colon\Tan\rightarrow \NR$, $ p\mapsto \trop(p)$,
where $\trop(p)$ is the element of $\NR=\Hom(M,\RR)$ given by
\[\la m,\trop(p)\ra\coloneq -\log|\chi^m(p)|.\]
This defines a proper surjective continuous map. For details, we refer to \cite[§\,3]{Pay}.

Let $\|\cdot\|$ be a toric metric on $L$. Then consider the following diagram
\[
\begin{xy}
  \xymatrix{
      \Tan \ar[rr]^{\log\|s(\cdot)\|} \ar[rd]_{\trop}  &     &  \RR   \\
                             &  \NR \ar@{-->}[ru] & \hspace{2em}.
  }
\end{xy}
\]
Since $\|\cdot\|$ is toric, $\log\|s(\cdot)\|$ is constant along the fibers of $\trop$. Moreover, $\trop$ is surjective and closed, and hence, there exists
a unique continuous function on $\NR$ making the above diagram commutative. This leads to the following definition.
\end{art}

\begin{defn}\label{torfunc}
Let $\ML=(L,\|\cdot\|)$ be a metrized toric line bundle on $\XS$ and $s$ a toric section of $L$.
We define the function
\[\psi_{\ML,s}\colon\NR\longrightarrow \RR, \quad u \longmapsto \log\|s(p)\|_{\SSS},\]
where $p\in \Tan$ is any element with $\trop(p)=u$. 
The line bundle and the toric section are usually clear from the context
and we alternatively denote this function by $\psim$
\end{defn}

\begin{art}\label{torfunc2}
For an alternative description of $\psi_{\ML,s}$, we consider the continuous map $\rho\colon\NR\rightarrow\Tan$ defined, for each $u\in\NR$, by the multiplicative norm
\[\rho(u)\colon K[M]\longrightarrow\RR_{\geq 0},\quad \sum_{m\in M}{\alpha_m\chi^m}\longmapsto \max_{m\in M}|\alpha_m|\exp(-\la m,u\ra).\]
Then it is easy to see that $\psi_{\ML,s}(u)=\log\|s(\rho(u))\|$ for all $u\in \NR$. 

We note that $\rho$ is a homeomorphism of $N_\RR$ onto a canonical closed subset $S(\TT):= \rho(N_\RR)$ of $\TT^{\an}$ called the {\it skeleton of $\TT^{\an}$}. Berkovich showed in \cite[\S 6.3]{Ber} that $\tau :=\trop \circ \rho$ is a strong proper deformation retraction from $\TT^{\an}$ onto $S(\TT)$. 
\end{art}


\begin{prop}\label{torfact}
Let $\ML=(L,\|\cdot\|)$ and $\ML'$ be metrized toric line bundles on $\XS$ with toric sections $s$ and $s'$, respectively. Let $\varphi\colon X_{\Sigma'}\rightarrow\XS$ be
a toric morphism with corresponding linear map $H$ as in \ref{tormor-com}. Then
\[ \psi_{\ML\otimes\ML',s\otimes s'}=\psi_{\ML,s}+\psi_{\ML',s'}\,,\quad
\psi_{\ML^{-1},s^{-1}}=-\psi_{\ML}\quad\text{and}\quad \psi_{\varphi^*\ML,\varphi^*s}=\varphi_{\ML,s}\circ H\,.\]
Moreover, if $(\|\cdot\|_n)_{n\in\NN}$ is a sequence of metrics on $L$ that converges uniformly to $\|\cdot\|$, then
$\bigl(\psi_{\|\cdot\|_n}\bigr)_{n\in\NN}$ converges uniformly to $\psi_{\|\cdot\|}$.
\end{prop}
\begin{proof}
This is established in the propositions 4.3.14 and 4.3.19 in \cite{BPS} and follows easily from the definitions.
\end{proof}

\begin{art}
In order to characterize toric metrics by functions on $\NR$,
we need the \emph{Kajiwara-Payne tropicalization} of $\XS$ introduced by \cite{Kaj} and \cite{Pay}.
This is a topological space $N_\Sigma$ together with a tropicalization map $ \XS^{\an}\rightarrow N_\Sigma$.
As a set, $N_\Sigma$ is a disjoint union of linear spaces
\[N_\Sigma=\coprod_{\sigma\in \Sigma}{N(\sigma)_{\RR}}\,,\]
where $N(\sigma)=N/\la N\cap \sigma\ra$ is the quotient lattice as in \ref{orbit-cone}.
We refer to \cite{Pay} for a description of the topology. 

The toric variety $X_\Sigma$ is the disjoint union of tori $\TT_{N(\sigma)}=\spec K[M(\sigma)],\sigma\in \Sigma$.
Hence, we can define the \emph{tropicalization map} 
\[\trop\colon X_\Sigma^{\an}\longrightarrow N_\Sigma\] as the disjoint union of tropicalization maps
$\trop\colon\TT_{N(\sigma)}^{\an}\rightarrow N(\sigma)_\RR$ as defined in \ref{trop}. This is also a proper surjective continuous map.
Especially, $N_\Sigma=\trop(\XS^{\an})$ is a compact space. 

By \cite[4.2.12]{BPS}, the canonical section $\rho:\NR \to \TT^{\an}$ from \ref{torfunc2} extends uniquely to a continuous proper section $\rho_\Sigma:N_\Sigma \to \XS^{\an}$.
\end{art}

\begin{prop}\label{tor-conc}
Let $\Sigma$ be a complete fan in $\NR$ and $\Psi$ a virtual support function on $\Sigma$. We set $L=L_\Psi$.
Then, for any metric $\|\cdot\|$ on $L$, the function $\psi_{\|\cdot\|}-\Psi$ extends to a continuous function on $N_\Sigma$.
In particular, the function $|\psi_{\|\cdot\|}-\Psi|$ is bounded.

Moreover, the assignment $\|\cdot\|\mapsto \psi_{\|\cdot\|}$ is a bijection between the sets of 
\begin{enumerate}
\item toric metrics on $L$;
\item continuous functions $\psi\colon\NR\rightarrow \RR$ sucht that $\psi-\Psi$ can be extended to a continuous function on $N_\Sigma$.
\end{enumerate}
\end{prop}
\begin{proof}
This is proved in Proposition 4.3.10 and Corollary 4.3.13 in \cite{BPS}.
The inverse map is given as follows: Let $\psi$ be a function as in (ii) and $\{m_\sigma\}$ a set of defining vectors of $\Psi$. For each cone $\sigma\in \Sigma$, the section
$s_\sigma=\chi^{m_\sigma}s$ is a non-vanishing regular section on $U_\sigma$.
Then we obtain a toric metric $\|\cdot\|_\psi$ on $L$ characterized by
\begin{align}\label{toricm}
\|s_\sigma(p)\|_\psi\coloneq \exp\bigl((\psi-m_\sigma)(\trop(p))\bigr)
\end{align}
on $U_\sigma$.
\end{proof}

\begin{defn} \label{candef}
Let $L$ be a toric line bundle on $\XS$ with toric section $s$ and let $\Psi$ be the associated virtual support function on $\Sigma$.
By Proposition \ref{tor-conc}, the function $\psi \coloneq \Psi$ defines a toric metric on $L$.
This metric is called the \emph{canonical metric} of $L$. We denoted it by $\|\cdot\|_{\can}$ and write $\MLcan=(L,\|\cdot\|_{\can})$.
\end{defn}

\begin{rem}
By \cite[Proposition 4.3.15]{BPS}, the canonical metric only depends on the structure of toric line bundle of $L$ and not on the choice of $s$.
\end{rem}

\begin{prop}\label{canprop}
Let $L$, $L'$ be toric line bundles on $\XS$ and let $\varphi\colon\XS'\rightarrow \XS$ be a toric morphism.
Let $\sigma\in\Sigma$ and $\iota\colon V(\sigma)\rightarrow \XS$ the closed immersion of \ref{orbit-cone}.
Then
\[
\overline{L\otimes L'}^{\can}\!=\MLcan\!\otimes \overline{L'}^{\can},\; \overline{L^{-1}}^{\can}\!=(\MLcan)^{-1},\;
\overline{\varphi^*L}^{\can}\!=\varphi^*\MLcan\!,
\;
\overline{\iota^*L}^{\can}\!=\iota^*{\MLcan}\!.\]
\end{prop}
\begin{proof}
The first two statements are established in \cite[Proposition 4.3.16]{BPS}. The last two statements are the corollaries 4.3.20 and 4.3.18 in \cite{BPS}.
\end{proof}

\subsection{Semipositive metrics and measures on toric varieties} \label{Semipositive metrics and measures on toric varieties}

We  continue our study of metrized line bundles on a toric variety. We assume that the reader is familiar with the notation introduced in §\,\ref{Metrized line bundles on toric varieties}. We  give a complete characterization of semipositive toric metrics in terms of concave functions. Moreover, we  describe the Chambert-Loir measure of a semipositive toric metric as the Monge--Amp\`ere measure of the associated concave function. 

In this subsection, we will use the following setup. 
Let $K$ be an algebraically closed field which is complete with respect to a non-trivial non-archimedean absolute value $|\cdot |$. Then we have a valuation $\val\coloneq -\log |\cdot |$ and a value group
 $\Gamma \coloneq \val(K^{\times})$ of rank one. 
We fix a free abelian group $M$ of rank $n$ with dual $N$ and denote by $\TT=\spec(K[M])$ the $n$-dimensional split torus over $K$. We consider a complete rational fan $\Sigma$ in $\NR$ with associated $\TT$-toric variety $X_\Sigma$. For a virtual support function $\Psi$ on $\Sigma$, we denote by $D_\Psi$ the associated $\TT$-toric Cartier divisor on $X_\Sigma$ and by $L \coloneq \OO(D_\Psi)$ the associated toric line bundle.



\begin{art} \label{algebraic metrics and toric models}
Let $\Pi$ be a complete $\Gamma$-rational polyhedral complex in $\NR$ with $\rec(\Pi)=\Sigma$, and let $\phi$ be a $\Gamma$-rational piecewise affine function on $\Pi$
with $\rec(\phi)=\Psi$.
Let $e>0$ be an integer such that $e\phi$ is a $\Gamma$-lattice function given by the defining vectors $\{(m_\Lambda,l_\Lambda)\}_{\Lambda\in \Pi}$ in $M\times \Gamma$.
Then $e\phi$ defines a $\TT_S$-Cartier divisor
\[D_{e\phi}=\left\{\left(\UUU_\Lambda,\alpha_{\Lambda}^{-1}\chi^{-m_\Lambda}\right)\right\}_{\Lambda\in\Pi},\]
where $\alpha_{\Lambda}\in K^\times$ with $\val(\alpha_{\Lambda})=l_\Lambda$,
and the pair $(\Pi,e\phi)$ defines a toric model $(\XXX_\Pi,D_{e\phi},e)$ of $(\XS,D_\Psi)$ (see Theorem \ref{class-tor}). 
We write
$\LLL=\OO(D_{e\phi})$  for the corresponding toric line bundles.
The model $(\XXX_\Pi,\LLL,e)$ induces an algebraic metric $\|\cdot\|_{\LLL}$ on $L$ (see \ref{algmod2}).
\end{art}

\begin{prop}\label{circ}
In the above notation, the metric $\|\cdot\|_\LLL$ is toric. Moreover, the equalities $\psi_{\|\cdot\|_\LLL}=\psi$ and $\|\cdot\|_\LLL=\|\cdot\|_\psi$ hold.
\end{prop}
\begin{proof}
Let $\Lambda\in \Pi$. Recall that $\UUU_\Lambda := \spec(K[M]^\Lambda)$ is an algebraic $K^\circ$-model of $U_{\rec(\Lambda)}$. The associated formal scheme has generic fiber 
$$U_{\rec(\Lambda)}^\circ \coloneq \{p \in U_{\rec(\Lambda)}^{\an}\mid p(f) \leq 1 \; \forall f \in K[M]^\Lambda\}.$$ 
Then $\UUU_\Lambda$ is a trivialization of $\LLL$ on which $s_\Psi^{\otimes e}$, considered as a rational section of $\LLL$, corresponds
to the rational function $\alpha_\Lambda^{-1}\chi^{-m_\Lambda}$  
(see \cite[4.9]{Guide}). Hence we have
\[\|s_\Psi(p)\|_\LLL=|\alpha_\Lambda^{-1}\chi^{-m_\Lambda}(p)|^{1/e}\]
for all $p\in 
U_{\rec(\Lambda)}^\circ$.
Let $u\in\Lambda$ and $p\in\Tan$ with $\trop(p)=u$.
Lemma \ref{trop-rec} below implies that $p\in U_{\rec(\Lambda)}^\circ$ and we obtain
\[\log\|s_\Psi(p)\|_\LLL=\log|\alpha_\Lambda^{-1}\chi^{-m_\Lambda}(p)|^{1/e}=\frac{1}{e}\left(\la m_\Lambda,u\ra+l_\Lambda\right)=\psi(u).\]
This shows that the metric $\|\cdot\|_\LLL$ is toric.
We deduce, by Definition \ref{torfunc}, that $\psi_{\|\cdot\|_\LLL}=\psi$ and, by Proposition \ref{tor-conc}, that $\|\cdot\|_\LLL=\|\cdot\|_\psi$.
\end{proof}

\begin{lemma}\label{trop-rec}
Let $\Pi$ be a complete $\Gamma$-rational polyhedral complex in $\NR$  and $\rec(\Pi)=\Sigma$,
and let $\red\colon \XS^{\an}\rightarrow {\XXX}_{\Pi,s}$ be the reduction map from \ref{f-alg}.
Let $\Lambda\in\Pi$ and $p\in \Tan$. Then 
\[\trop(p)\in\Lambda\ \Longleftrightarrow\ p\in U_{\rec(\Lambda)}^\circ\  \Longleftrightarrow\ \red(p)\in{\UUU}_{\Lambda,s}\,.\]
\end{lemma}
\begin{proof}
By \cite[Lemma 6.21]{Guide}, we have $\trop(p)\in\Lambda$ if and only if 
$|p(f)|\leq 1$ for all $f\in K[M]^\Lambda$ or in other words $p\in  U_{\rec(\Lambda)}^\circ$.
By the description of the reduction map in \ref{f-alg}, this is equivalent to $\red(p)\in {\UUU}_{\Lambda,s}$.
\end{proof}

\begin{cor}\label{pa-alg}
Let $\psi$ be a $\Gamma$-rational piecewise affine concave function on $\NR$ with $\rec(\psi)=\Psi$.
Then the metric $\|\cdot\|_\psi$ on $L$ is induced by a toric model.
\end{cor}
\begin{proof}
As in the proof of \cite[Theorem 3.7.3]{BPS}, we can show that
there exists a complete $\Gamma$-rational polyhedral complex $\Pi$ in $\NR$ such that $\rec(\Pi)=\Sigma$ and $\psi$ is piecewise affine on $\Pi$. 
Since $\Gamma$ is discrete or divisible, the complex $\Pi$  induces a proper toric scheme $\XXX_\Pi$ (see \ref{sumihiro}). 
Then Proposition \ref{circ} says that $\|\cdot\|_\psi$ is induced by a toric model $(\XXX_\Pi,D_{e\psi},e)$ of $(X_\Sigma,D_\Psi)$.
\end{proof}

\begin{prop}\label{alg-pa}
Let $\|\cdot\|$ be an algebraic metric on $L$.
Then the function $\psi_{\|\cdot\|}$ is $\Gamma$-rational piecewise affine.
\end{prop}
\begin{proof}
There exists a proper $\Kval$-model $(\XXX,\LLL,e)$ of $(X_\Sigma,L)$ inducing the metric $\|\cdot\|$.
Let $\{\UUU_i\}_{i\in I}$ be a trivialization of $\LLL$.
Then the subsets $U_i^\circ=\red^{-1}\bigl(\UUU_i\cap\widetilde{\XXX}\bigr)$ form a finite closed cover of $\XS^{\an}$.
On $\UUU_i$ the rational section $s^{\otimes e}$ corresponds to a rational function $\lambda_i\in K(M)^\times$ such that on $U_i^\circ$
we have
\[\|s(p)\|=|\lambda_i(p)|^{1/e}.\]
We write $\textstyle\lambda_i=\frac{\sum_{m\in M}{\alpha_m\chi^m}}{\sum_{m\in M}{\beta_m\chi^m}}$. Using the continuous map $\rho\colon \NR\rightarrow \Tan$ from \ref{torfunc2}, 
we have on the closed subset $\Lambda_i\coloneq \rho^{-1}(U_i^\circ\cap\Tan)\subseteq \NR$,
\begin{align*}
\psi_{\|\cdot\|}(u)&=\log\|s(\rho(u))\|\\
&=\log|\lambda_i(\rho(u))|^{1/e}\\
&=\frac{1}{e}\log \left(\max_{m\in M}|\alpha_m|\exp(-\la m,u \ra)\right)- \frac{1}{e}\log \left(\max_{m\in M}|\beta_m|\exp(-\la m,u \ra)\right)\\ 
&= \frac{1}{e}\min_{m\in M}\left(\la m,u \ra+\val(\beta_m)\right)-\frac{1}{e}\min_{m\in M}\left(\la m,u \ra+\val(\alpha_m)\right).
\end{align*}
We see that $\psi_{\|\cdot\|}|_{\Lambda_i}$ is the difference of two $\Gamma$-rational piecewise affine concave functions.
Since $\{\Lambda_i\}_{i\in I}$ is a finite closed cover of $\NR$, we deduce that $\psi_{\|\cdot\|}$ is $\Gamma$-rational piecewise affine (see \ref{definition piecewise affine function} and \ref{rationality for piecewise affine functions}).
\end{proof}

Next we study semipositive toric metrics on $L$.
\begin{prop}\label{walter-prop}
Let $\|\cdot\|$ be an algebraic metric on $L$.
\begin{enumerate}
\item If $\|\cdot\|$ is semipositive, then $\psi_{\|\cdot\|}$ is concave.
\item We assume that $\|\cdot\|$ is toric. Then $\|\cdot\|$ is semipositive if and only if $\psi_{\|\cdot\|}$ is concave.
\end{enumerate}
\end{prop}
\begin{proof}
(ii) Because each algebraic metric is formal (see \ref{f-alg}), this follows from Corollary 8.12 in \cite{GuKue}.

(i) For $\metr$ semipositive, we have to show that $\psi_\metr$ is concave along any affine line. By a density argument, we may assume that the line is  $\Gamma$-rational. Similarly as in \cite[proof of Proposition 4.7.1]{BPS}, we use pull-back with respect to a suitable equivariant morphism to reduce the concavity on the affine line to the case of $\PP_K^1$ and hence the claim follows from Corollary \ref{torification of semipositive} and (ii).
\end{proof}

\begin{cor}\label{tor-alg}
Let $\|\cdot\|$ be a semipositive algebraic metric on $L$. Then the toric metric $\|\cdot\|_{\SSS}$ is also algebraic and semipositive.
\end{cor}
\begin{proof}
By the propositions \ref{walter-prop}\,(i), \ref{alg-pa} and \ref{tor-conc}, the function $\psi=\psi_{\|\cdot\|}$ is a concave $\Gamma$-rational piecewise affine function with $\rec(\psi)=\Psi$.
Then Corollary \ref{pa-alg} says that the metric $\|\cdot\|_{\SSS}=\|\cdot\|_{\psi}$ is algebraic and Proposition \ref{walter-prop}\,(ii) implies semipositivity.
\end{proof}

\begin{thm}\label{sem-con}
Let $\Psi$ be a  support function on the complete fan $\Sigma$ in $\NR$ and set $L=L_\Psi$.
Then there is a bijection between the sets of 
\begin{enumerate}
\item semipositive toric metrics on $L$;
\item concave functions $\psi$ on $\NR$ such that  the function $|\psi-\Psi|$ is bounded;
\newcounter{store}
\setcounter{store}{\theenumi}
\end{enumerate}
\begin{enumerate}
\setcounter{enumi}{\thestore}
\item continuous concave functions on $\Delta_\Psi$.
\end{enumerate}
The bijections are given by $\|\cdot\|\mapsto\psi_{\|\cdot\|}$ and by $\psi\mapsto\psi^\vee$.
\end{thm}

This theorem was proven by Burgos--Philippon--Sombra \cite[Theorem 4.8.1]{BPS} in the case of a discrete or an archimedean absolute value. Note that the bijection between (i) and (ii) holds also in case of a non-concave virtual support function $\Psi$ as then both sets are empty. This follows from the arguments in the proof below. 

\begin{proof}
The bijection between (ii) and (iii) follows from Proposition \ref{con-func}. To prove the bijection between (i) and (ii),
let $\|\cdot\|$ be a semipositive toric metric on $L$. By Proposition \ref{tor-conc}, the function $|\psi_{\|\cdot\|}-\Psi|$ is bounded.
Furthermore, there exists a sequence $(\|\cdot\|_{n})_{n\in \NN}$ of semipositive algebraic metrics converging to the toric metric $\|\cdot\|$. 
Proposition \ref{walter-prop}\,(i) says that
the functions $\psi_{\|\cdot\|_n}$ are concave.
By Proposition \ref{torfact}, the sequence $(\psi_{\|\cdot\|_n})_{n\in\NN}$ converges uniformly to $\psi_{\|\cdot\|}$ and hence the latter is also concave.

Conversely, let $\psi$ be a concave function on $\NR$ such that $|\psi-\Psi|$ is bounded. 
Then 
by Proposition \ref{sup-conc}, there is a sequence of $\Gamma$-rational piecewise affine concave functions $(\psi_k)_{k\in \NN}$ with $\rec(\psi_k)=\Psi$, that uniformly
converges to $\psi$.
Because $\psi_k$ is a piecewise affine concave function with $\rec(\psi_k)=\Psi$, the function $\psi_k-\Psi$ continuously extends to $N_\Sigma$.
We conclude that $\psi-\Psi$ continuously to  $N_\Sigma$.
By Proposition \ref{tor-conc}, we obtain toric metrics $\|\cdot\|_{\psi}$ and $\|\cdot\|_{\psi_k},k\in \NN$, given as in (\ref{toricm}).
Then the sequence of metrics $(\|\cdot\|_{\psi_k})_{k\in \NN}$ converges to $\|\cdot\|_{\psi}$.
By Proposition \ref{circ}, the metric $\|\cdot\|_{\psi_k}$ is algebraic and therefore, by Proposition \ref{walter-prop}\,(ii),
semipositive. It follows that the metric $\|\cdot\|_{\psi}$ is also semipositive.
\end{proof}


\begin{rem} \label{piecewise affine vs algebraic}
Theorem \ref{sem-con} induces a bijective correspondence between semipositive algebraic toric metrics $\metr$ on $L$ and concave $\Gamma$-rational piecewise affine functions $\psi$ on $\NR$ with $\rec(\psi)=\Psi$. Moreover, such metrics always have a toric model. One direction follows from the propositions \ref{alg-pa}, \ref{walter-prop}\,(i) and \ref{tor-conc}. The converse and the last claim are a consequence of Corollary \ref{tor-alg}.
\end{rem}

Now we characterize the Chambert-Loir measure associated to a semipositive toric metric.
Let $\psi\colon \NR\rightarrow \RR$ be a concave function. We extend the Monge-Ampère measure $\M_M(\psi)$ on $\NR$ (Definition \ref{a8})
to a measure $\overline{\M}_M(\psi)$ on $N_\Sigma$ by setting
\[\overline{\M}_M(\psi)(E)=\M_M(\psi)\left(E\cap \NR\right)\]
for any Borel subset $E$ of $N_\Sigma$.

\begin{thm}\label{trop-mong} \label{trop-mong2}
Let $\|\cdot\|$ be a semipositive 
toric metric on $L$ and $\psi=\psi_{\|\cdot\|}$ the associated concave function on $\NR$.
Then
\[\trop_*\left(\c1\left(L,\|\cdot\|\right)^n\right)=n!\,\overline{\M}_M(\psi).\]
Moreover, we have 
$$c_1(L,\metr)^n = (\rho_\Sigma)_*\left(n!\,\overline{\M}_M(\psi) \right).$$
\end{thm}

\begin{proof}
In the case of an algebraic metric, the claim follows from the same arguments as in \cite[Theorem 4.7.4]{BPS} based in our case on Remark \ref{piecewise affine vs algebraic}, \ref{ver-irr} and Proposition \ref{mult-deg-vol}. Note that our assumption $K$ algebraically closed yields that $\Gamma$ is divisible and hence the special fiber of a toric model $\XXX_\Pi$ is reduced (see \ref{pol-scheme}) and the multiplicity of a vertex of $\Pi$ is one as well. This simplifies this argument a little bit.

The general case is based on the algebraic case as in \cite[Corollary 4.7.5]{BPS}. It is used here  that the boundary $X_\Sigma^{\an} \setminus \Tan$ is a set of measure zero with respect to $\c1\left(L,\|\cdot\|\right)^n$ by Corollary \ref{MeasZero}.
\end{proof}

At the end of this subsection, we quote a result about the restriction of semipositive metrics to toric orbits
which will be useful in the proof of the local height formula.
Recall that $\Psi$ is a support function on $\Sigma$ with associated toric line bundle $(L,s)$ with toric section.
Let $\sigma$ be a cone of $\Sigma$ and $V(\sigma)$ the corresponding orbit closure
with the structure of a toric variety (cf. \ref{orbit-cone}).
We denote by $\iota\colon V(\sigma)\rightarrow \XS$ the closed immersion.
Let $m_\sigma\in M$ be a defining vector of $\Psi$ at $\sigma$ and set $s_\sigma=\chi^{m_\sigma}s$.
By \ref{inter-clos}, the divisor $D_{\Psi-m_\sigma}=\dvs(s_\sigma)$ intersects $V(\sigma)$ properly and 
we can restrict $s_\sigma$ to $V(\sigma)$ to obtain a toric section $\iota^* s_\sigma$ of 
the toric line bundle $\OO\bigl(D_{\left(\Psi-m_\sigma\right)(\sigma)}\bigr)\simeq\iota^*L$.

\begin{prop}\label{legen-restr}
Let notation be as above and denote by $F_\sigma$ the face of $\Delta_\Psi$ associated to $\sigma$ (see \ref{pol-fan}).
Let $\|\cdot\|$ be a semipositive toric metric on $L$.
Then, for all $m\in F_\sigma-m_\sigma$,
\[\psi_{\iota^*\ML,\iota^* s_\sigma}^\vee(m)=\psi_{\ML,s}^\vee(m+m_\sigma)\, .\]
\end{prop}
\begin{proof}
We can prove the statement as in \cite[Proposition 4.8.8]{BPS} since the discreteness of the valuation doesn't play a role in the argument.
\end{proof}


\subsection{Local heights of toric varieties} \label{Local heights of toric varieties}

We prove a formula to compute the local height of a normal toric variety over an  algebraically closed non-archimedean field. Note that this is no restriction of generality as we can always achieve that by a base extension and as local heights are invariant under such a base extension.
This generalizes work by Burgos, Philippon and Sombra who showed this formula over fields with a discrete valuation (cf. \cite[Theorem 5.1.6]{BPS}).

Let $K$ be an algebraically closed field which is complete with respect to a non-trivial absolute value $|\cdot|$ and denote by $\Gamma=-\log|K^\times|$ the associated value group.
We fix a lattice $M\simeq \ZZ^n$ with dual $M^\vee=N$ and denote by $\TT=\spec(K[M])$ the $n$-dimensional split torus over $K$.
Let $\Sigma$ be a complete fan on $\NR$ and $\XS$ the associated proper $\TT$-toric variety.

Following \cite[§\,5.1]{BPS} we define a local height for toric metrized line bundles that does not depend on the choice of sections. 
\begin{defn} \label{toriclocaldef}
Let $\ML_i$, $i=0,\dotsc,t$, be toric line bundles on $\XS$ equipped with DSP toric metrics. We denote by $\MLcan_i$ the same line bundle endowed 
with the canonical metric.
Let $Y$ be a $t$-dimensional prime cycle of $\XS$ and let $\varphi\colon Y'\rightarrow Y$ be a birational morphism such that $Y'$ is projective.
Recall the definition of local heights in \ref{locallimit}.
Then the \emph{toric local height} of $Y$ with respect to $\ML_0,\dotsc, \ML_t$ is defined as
\[ \lambdator_{\ML_0,\dotsc,\ML_t}(Y)=\lambda_{(\varphi^*\ML_0,s_0),\dotsc,(\varphi^*\ML_t,s_t)}(Y')-\lambda_{(\varphi^*\MLcan_0,s_0),\dotsc,(\varphi^*\MLcan_t,s_t)}(Y')\, ,
\]
where $s_0\dotsc,s_t$ are invertible meromorphic sections of $L$ 
with 
\begin{align}\label{intersect}
|\dvs(s_0)|\cap\dotsb\cap|\dvs(s_t)|\cap Y=\emptyset\, .
\end{align}
This definition extends to cycles by linearity.
When $\ML_0=\dotsb=\ML_t=\ML$, we write shortly 
$\lambdator_{\ML}(Y)=\lambdator_{\ML_0,\dotsc,\ML_t}(Y)$.
\end{defn}

\begin{rem}
Proposition \ref{proplocal2}\,(iii, v) implies that the toric local height does not depend on the choice of $\varphi$ and $Y'$
nor on the choice of the meromorphic sections.
When $\left|\dvs(s_0)\right|,\dotsc,\left|\dvs(s_t)\right|$
intersect properly on $Y$, then condition (\ref{intersect}) is fullfilled.
\end{rem}

\begin{prop}
The toric local height is symmetric and multilinear in the metrized line bundles.
\end{prop}
\begin{proof}
This follows easily from Proposition \ref{proplocal2}\,(ii).
\end{proof}

\begin{defn} \label{vartheta}
Let $\ML=(L,\|\cdot\|)$ be a semipositive metrized toric line bundle with a toric section $s$. Let $\Psi$ be the corresponding support function on $\Sigma$ and $\psi_{\ML,s}$ 
the associated concave function on $\NR$.
The \emph{roof function} associated to $(\ML,s)$ is the concave function $\vartheta_{\ML,s}\colon\Delta_\Psi\rightarrow\RR$ given by
\[\vartheta_{\ML,s}=\psi_{\ML,s}^\vee,\]
where $\psi_{\ML,s}^\vee$ denotes the Legendre-Fenchel dual (see \ref{a5}).
We will denote $\vartheta_{\ML,s}$ by $\thetam$ if the line bundle and section are clear from the context.
\end{defn}

\begin{art}
Let notation be as above. 
If $\|\cdot\|$ is an algebraic metric, then, by Proposition \ref{alg-pa} and \ref{a12}, the roof function $\thetam$ is 
piecewise affine concave.
\end{art}

\begin{thm}\label{thetheorem}
Let $\Sigma$ be a complete fan on $\NR$. Let $\ML=(L,\|\cdot\|)$ be a toric line bundle on $\XS$ equipped with a semipositive toric metric.
We choose any toric section $s$ of $L$ and denote by $\Psi$ the corresponding support function on $\Sigma$.
Then, the toric local height of $\XS$ with respect to $\ML$ is given by
\begin{align}\label{theformula}
\lambdator_{\ML}(\XS)=(n+1)!\int_{\Delta_\Psi}{\vartheta_{\ML,s}\,\dint\,\vol_M}\,
\end{align}
where $\Delta_\Psi$ is the stability set of $\Psi$ and $\vol_M$ is the Haar measure on $\MRR$ such that $M$ has covolume one.
\end{thm}
\begin{proof}
The proof is completely analogous to \cite[Theorem 5.1.6]{BPS}. It is based on induction relative to $n$ and uses the induction formula \ref{ifDSP} in an essential way.
\end{proof}

\begin{rem}
The above theorem also holds in the archimedean case (see \cite[Theorem 5.1.6]{BPS}).
In \cite[§\,5.1]{BPS}, the formula in \eqref{theformula} is extended to
toric local heights with respect to distinct line bundles in the archimedean case or in case of a discrete valuation.
Moreover, 
the toric local height of a translated toric subvariety
and its behavior with respect to equivariant morphisms is studied.
For arbitrary non-archimedean fields, all these results remain valid and the arguments are the same.
\end{rem}



\section{Global heights of varieties over finitely generated fields} \label{global heights chapter}

In this section, we apply the results about toric local heights from the previous section to global heights. First, we recall the theory of global heights over an $M$-field and we give an induction formula for global heights. In §\,\ref{section3.2}, we consider Moriwaki's $M$-field structure on a finitely generated field over a global field and we prove Theorem \ref{thm2} from the introduction. 
In §\,\ref{GlobalToric}, we apply Theorem \ref{thm3} and Theorem \ref{thm2} to a fibration with toric generic fiber leading to a combinatorial formula for a suitable global height of the fibration. In the last subsection, we illustrate this formula for  projective toric varieties over the function field of an elliptic curve.

\subsection{Global heights of varieties over an $M$-field} \label{global height over M}

Diophantine geometry is usually considered over a number field or a function field. Osgood and Vojta realized a stunning similarity to Nevanlinna theory where the base field is the field of meromorphic function on $\CC$. The characteristic function in Nevanlinna theory corresponds to the height in diophantine geometry. Inspired by the analogy and the proof of the second main theorem of Nevanlinna theory, Vojta found a new proof of the Mordell conjecture which was later generalized by Faltings to prove Lang's conjecture for subvarieties of abelian varieties.
 
In \cite[Definition 2.1]{GuM}, the notion of an $M$-field was introduced to capture all three situations simultaneously. In this subsection, we will recall the definition and the theory of global  heights over an $M$-field. This will be later applied to heights over finitely generated fields introduced by Moriwaki.

\begin{defn} \label{definition of M-field}
Let $K$ be a field and $(M,\mu)$ be a measure space endowed with a positive measure $\mu$. We assume that for $\mu$-almost every $v \in M$, we fix a non-trivial absolute value $|\phantom{a}|_v$ on $K$. Then $K$ is called an $M$-field if 
$\log|f|_v \in L^1(M,\mu)$ for all $f \in K^\times$. We say that $K$ satisfies the {\it product formula} if $\int_M \log |f|_v d\mu(v)=0$ for all $f \in K^\times$.
\end{defn}

\begin{rem} \label{more general definition of M-field}
In \cite[Definition 2.1]{GuM}, a more general definition of an $M$-field is given which is more suitable to uncountable fields as occurring in Nevanlinna theory. For countable fields, the definitions are the same. For our purposes in this paper, the above definition will suffice. We refer to \cite{GuM} for the theory of global heights over the more general $M$-fields noticing that this much more technical and that one reduces always to the above special case by passing to a sufficiently large finitely generated subfield.
\end{rem}

\begin{ex} \label{number field} 
Let $F$ be a number field and $M_F$ be the set of places endowed with the  discrete measure $\mu$  given by $\mu(v)=\frac{[F_v: \QQ_v]}{[F: \QQ]}$ for $v \in M_F$ with completions $F_v$ and $\QQ_v$.   
We denote by $|\phantom{a}|_v$ the absolute value for the place $v \in M_F$ which extends the standard euclidean or $p$-adic absolute value on $\QQ$. Then $F$ is an $M_F$-field satisfying the product formula.
\end{ex}

\begin{ex} \label{function field}
Let us fix the function field $F_0:=k(C)$ of a regular projective curve $C$ over an arbitrary field $k$ and let $M_{F_0}$ be the set of places of $F_0$ corresponding to the closed points of $C$. We fix $q \in \RR$ with $q >1$. For $v_0 \in M_{F_0}$, we choose the standard absolute value given by
\begin{equation} \label{absolute value for function} 
|\alpha|_{v_0}=q^{-\ord_{v_0}(\alpha)}
\end{equation}
for $\alpha \in F_0$. We endow $M_{F_0}$ with the discrete measure given by $\mu(v_0)=[k(v_0):k]$. Then $F_0$ is an $M_{F_0}$-field satisfying the product formula.

More generally, we consider a finite extension $F$ of $F_0$. Then $F$ is the function field of a regular projective curve $Z$ over $C$. Let $M_F$ be the set of places corresponding to the closed points of $Z$. For $v \in M_F$, we use the representative $|\phantom{a}|_v$ which extends $|\phantom{a}|_{v_0}$ to $K$, where $v_0:=v|_{F_0}$. We endow $F$ with the discrete measure given by 
$$\mu(v)=\frac{[F_v: F_{0,v_0}]}{[F: F_0]}\mu_0(v_0).$$ 
Then $F$ is an $M_F$-field satisfying the product formula. 
\end{ex}

\begin{art}\label{globalf}
%
A \emph{global field} $F$ is either a number field or a finite extension of the  function field of a fixed regular projective curve over a finite field $k$.
We endow $F$ always with the $M_F$-field structure given in Examples \ref{number field} and \ref{function field}. In the latter case, we choose $q$ as the number of elements of $k$. In fact, we never use that $k$ is finite and we could choose any constant $q > 1$. For some results of the next subsection, it is required that $k$ is countable (otherwise one has to state them differently restricting to a countable subfield of $k$). 
\end{art}

\begin{defn} \label{completion of algebraic closure}
Let $K$ be a field with an absolute value $|\phantom{a}|_v$. We denote by $\KK_v$  the completion of an algebraic closure of the completion of $K$ with respect to $v\in M$. Note that $\KK_v$ is a minimal algebraically closed complete field extending $(K, |\phantom{a}|_v)$ with residue field equal to an algebraic closure of $\tilde K$ (see \cite[Proposition 3.4.1/3, Lemma 3.4.1/4]{BGR}). By abuse of notation, we denote the absolute value of $\KK_v$ also by $|\phantom{a}|_v$. 

Let $X$ be a  variety over $K$.
We set $X_v\coloneq X\times_K\spec(\KK_v)$. 
If $v$ is archimedean, then $\KK_v = \CC$ and we denote by $\Xan_v=X_v(\KK_v)$ the complex analytic space associated to $X$.
If $v$ is non-archimedean, then $\Xan_v$ is the Berkovich analytic space associated to $X_v$ over $\KK_v$ as defined in \ref{BerkAn}.
We call $\Xan_v$ the \emph{analytification of} $X$ with respect to $v$ (or $|\cdot|_v$).
\end{defn}

\begin{art} \label{M-metrics}
In the following, $K$ is always an $M$-field. Our goal is to define an $M$-metric on a line bundle $L$ over the proper variety $X$ over $K$.  For $\mu$-almost every $v\in M$, we have an associated absolute value $|\phantom{a}|_v$. 
An \emph{($M$-)metric} on $L$ is a family of metrics $\|\cdot\|_v$ on $\Lan_v$ for $\mu$-almost every  $v \in M$ as above. 
The corresponding metrized line bundle is denoted by $\ML=(L,(\|\cdot\|_v)_v)$.

An $M$-metric on $L$ is said to be \emph{semipositive} if $\|\cdot\|_v$ is semipositive for $\mu$-almost all $v\in M$ (cf. Definition \ref{semip} and Remark \ref{archimedean local heights}).
Moreover, a metrized line bundle $\ML$ is \emph{DSP} if there are semipositive metrized line bundles $\MM$, $\MN$ on $X$ such that
$\ML=\MM\otimes\MN^{-1}$. 
\end{art}

\begin{art} \label{heights and integrable}
Let $Z$ be a $t$-dimensional cycle on $X$ and $(\ML_i,s_i)$, $i=0,\dotsc,t$, DSP metrized line bundles on $X$ with invertible meromorphic sections such that
$|\dvs(s_0)|\cap\dotsb\cap|\dvs(s_t)|\cap |Z|=\emptyset$.
For $v\in M$, we set for the local height at $v$,
\[\lambda_{(\ML_0,s_0),\dotsc,(\ML_{t},s_t)}(Z,v)\coloneq \lambda_{{\widehat{\dvs}(s_0)_v},\dotsc,\widehat{\dvs}(s_t)_v}(Z_v),\]
where ${\widehat{\dvs}(s_i)_v}$ is the pseudo-divisor on $X_v$ induced by $\widehat{\dvs}(s_i)$ (cf. Example \ref{exdvs}).
\end{art}

\begin{art}\label{integrable}
A $t$-dimensional prime cycle $Y$ of $X$ is called \emph{integrable} with respect to DSP metrized line bundles $\ML_i$, $i=0,\dotsc,t$, on $X$
if there is a birational  map $\varphi\colon Y'\rightarrow Y$ from a projective variety $Y'$ 
and invertible meromorphic 
sections $s_i$ of $\varphi^*L_i$, $i=0,\dotsc,t$, with $\dvs(s_0), \dots, \dvs(s_t)$ intersecting properly, such that the function
\begin{align}\label{int0}
M\longrightarrow \RR,\quad  v\longmapsto 
\lambda_{(\varphi^*\ML_0,s_0),\dotsc,(\varphi^*\ML_{t},s_t)}(Y',v)
\end{align}
is $\mu$-integrable on $M$. 
A $t$-dimensional cycle is \emph{integrable} if its components are integrable.
\end{art}

\begin{art}\label{propint}
Let $Y$ be  a prime cycle. Then there exists a generically finite surjective morphism $\varphi\colon Y'\rightarrow Y$ from a proper variety $Y'$ 
and invertible meromorphic sections $s_i$ of $\varphi^*L_i$, $i=0,\dotsc,t$,  satisfying 
\[\left|\dvs(s_0)\right|\cap\dotsb\cap\left|\dvs(s_t)\right|=\emptyset.\]
By Chow's lemma, we may even assume that $Y$ is projective, $\varphi$ is birational and that  $\dvs(s_0), \dots, \dvs(s_t)$ intersect properly, but we don't want to make these additional assumptions here. Then $Y$ is $\mu$-integrable with respect to $\ML_0, \dots, \ML_t$ if and only if 
the $\mu$-integrability of (\ref{int0}) holds. 
Moreover, the notion of integrability of cycles is closed under tensor product and pullback of DSP metrized line bundles. 
This follows from  \cite[11.4, 11.5]{GuPisa}.
\end{art}


\begin{defn}
Let $X$ be a proper variety over an $M$-field $K$ and $Y$ a $t$-dimensional prime cycle on $X$ which is integrable with respect to DSP metrized line bundles $\ML_0,\dotsc,\ML_t$ on $X$.
Let $Y'$ and $s_0,\dotsc,s_t$ be as in  \ref{integrable}.
Then the \emph{global height} of $Y$ with respect to $\ML,\dotsc,\ML_t$ is defined as
\[\h_{\ML_0,\dotsc,\ML_t}(Y)=\int_{ M}{\lambda_{(\varphi^*\ML_0,s_0),\dotsc,(\varphi^*\ML_{t},s_t)}(Y',v)}\,\dint \mu(v).\]
By linearity, we extend this definition to all $t$-dimensional cycles on $X$.

Using Corollary \ref{proplocal2}\,(iii), the archimedean analogon mentioned in Remark \ref{archimedean local heights} and the product formula of $K$, we see that
this definition is independent of the choice of the sections.
\end{defn}

\begin{prop}\label{propheights}
The global height of integrable cycles has the following basic properties:
\begin{enumerate}
\item It is multilinear and symmetric with respect to tensor products of DSP metrized line bundles.
\item Let $\varphi\colon X'\rightarrow X$ be a morphism of proper varieties over $K$ and let $Z'$ be a
$t$-dimensional cycle such that $\varphi_* Z'$ is integrable with respect to DSP metrized line bundles
$\ML_0,\dotsc,\ML_t$ on $X$. Then we have
\[\h_{\varphi^*\ML_0,\dotsc,\varphi^*\ML_t}(Z')=\h_{\ML_0,\dotsc,\ML_t}(\varphi_*Z').\]
\end{enumerate}
\end{prop}
\begin{proof}
Using \ref{propint}, we get the results by integrating the corresponding 
formulas stated in Proposition \ref{proplocal2}.
\end{proof}

\begin{thm}[Global induction formula] \label{ifMglobal} 
Let $X$ be a $d$-dimensional proper variety over the $M$-field $K$. Let ${\ML}_0,\dotsc,{\ML}_d$ be DSP metrized line bundles 
on $X$ and let $s_d$ be any invertible meromorphic section of $L_d$. If $X$ is integrable with respect to ${\ML}_0,\dotsc,{\ML}_d$ 
and if $\cyc(s_d)$ is integrable with respect to ${\ML}_0,\dotsc,{\ML}_{d-1}$, then 
$$\phi(v):= {\int_{\Xan_v}{\log\|s_d\|_{d,v}\c1({\ML}_{0,v}) 
\wedge\dotsb\wedge\c1({\ML}_{d-1,v})}}$$
is in $L^1(M,\mu)$ and 
\[ h_{{\ML}_0,\dotsc,{\ML}_d}(\XX) =\h_{{\ML}_0,\dotsc,{\ML}_{d-1}}(\cyc(s_d))
-\int_{M_F} \phi(v)\, d\mu(v) .\]
\end{thm}

\begin{proof}
We may assume that $K$ is infinite as otherwise $M$ has measure zero and all heights are zero as well. 
By Proposition \ref{propheights} and Chow's lemma, we may assume that $X$ is projective and there are invertible meromorphic sections $s_j$ of $L_j$ for $j=0,\dots, d$ with 
$$|\dvs(s_0)| \cap \dots \cap |\dvs(s_d)| = \emptyset.$$
Then the claim follows from Theorem \ref{ifDSP}. 
\end{proof}

\begin{defn}
Let $F$ be a global field with the structure $(M_F,\mu)$ of an $M_F$-field as in Example \ref{globalf}.
Let $\XX$ be a proper variety over a global field $F$ and $\LL$ a line bundle on $\XX$.
We call an $M_F$-metric on $\LL$  \emph{quasi-algebraic} if there exist a finite subset $S\subseteq M_F$ containing the archimedean places
and a proper algebraic model $(\XXX, \LLL,e)$ of $(\XX,\LL)$ over the ring \[F^\circ_S=\{\alpha \in F\,|\,|\alpha|_v\leq 1 \,\forall v\notin S\},\]such that, for each $v\notin S$, the metric $\|\cdot\|_v$ is induced by
the localization 
\[(\XXX\times_{F^\circ_S}\spec \FF^\circ_v,\LLL\otimes_{F^\circ_S}\FF^\circ_v,e).\]
\end{defn}

\begin{prop}\label{qu-al}
Let $\XX$ be a proper  variety over a global field $F$. Then every $d$-dimensional cycle of $\XX$ is $\mu$-integrable with respect to 
DSP quasi-algebraic $M_F$-metrized line bundles $\overline{\LL}_0,\dotsc,\overline{\LL}_d$ on $\XX$.
\end{prop}
\begin{proof}
This is \cite[Proposition 1.5.14]{BPS}.
\end{proof}

\begin{rem}\label{ifglobal}
In the situation of Proposition \ref{qu-al} and with $\dim(\XX)=d$, the hypotheses of the global induction formula \ref{ifMglobal}  are always satisfied for 
any invertible meromorphic section $s_d$ of $\LL_d$ and hence 
there is only a finite number of $v\in M_F$ such that
\[ \int_{\XX^{\an}_v}{\log\|s_d\|_{d,v}\c1(\overline{\LL}_{0,v})\wedge\dotsb\wedge\c1(\overline{\LL}_{d-1,v})}\neq 0\]
and the global induction formula holds.
\end{rem}

\begin{prop}\label{BCheight}
Let $F$ be a global field and $F'$ a finite extension of $F$ with the induced structure of an $M_{F'}$-field (see Example \ref{globalf}).
Let $\XX$ be an $F$-variety, $\overline{\LL}_i$, $i=0,\dotsc,t$, quasi-algebraic DSP metrized line bundles on $\XX$ and $\mathcal{Z}$ a $t$-dimensional cycle on $\XX$.
We denote by $\pi\colon \XX'\rightarrow \XX$ the morphism, by $\mathcal{Z}'$ the cycle and by $\pi^*\overline{\LL}_i$ the $M_{F'}$-metrized line bundles obtained by base change to $F'$.
Then
\[\h_{\pi^*\overline{\LL}_0,\dotsc,\pi^*\overline{\LL}_t}\left(\mathcal{Z}'\right)=\h_{\overline{\LL}_0,\dotsc,\overline{\LL}_t}(\mathcal{Z}).\]
\end{prop}
\begin{proof}
This follows from \cite[Proposition 1.5.10]{BPS}.
\end{proof}

\begin{defn}
Let $F$ be a global field and let $\overline{\LL}$ be a quasi-algebraic $M_F$-metrized line bundle on the proper variety $\XX$ over $F$. We say that ${\overline{\LL}}$ is \emph{nef}\index{M-metrized line bundle@($M$-)metrized line bundle!nef} if $\|\cdot\|$ is semipositive and, for each point $p\in {\XX}(\overline{F})$,
the global height $\h_{{\overline{\LL}}}(p)$ is non-negative.
\end{defn}

\begin{ex}
Let ${\overline{\LL}}=(L,(\|\cdot\|_v)_v)$ be a semipositive quasi-algebraic metrized line bundle. We assume that ${\overline{\LL}}$ is generated by small global sections, i.\,e.
for each point $p\in {\XX}(\overline{F})$, there exists a global section $s$ such that $p\notin \left|\dvs(s)\right|$ and
$\sup_{x\in {\XX}^{\an}_v }\|s(x)\|_v\leq 1$ for all $v\in M_F$. Then ${\overline{\LL}}$ is nef.
\end{ex}

The idea of the following proof was suggested to us by José Burgos Gil.
\begin{lemma}\label{LAST}
Let $V$ be a $d$-dimensional subvariety of ${\XX}$ and 
let ${\overline{\LL}}_1,\dotsc,{\overline{\LL}}_d$ be nef quasi-algebraic $M_F$-metrized line bundles on ${\XX}$.
Then,
\[\h_{{\overline{\LL}}_1,\dotsc,{\overline{\LL}}_d}(V)\geq 0.\]
\end{lemma}
\begin{proof}
We may assume that $V={\XX}$ and, by Chow's Lemma and Proposition \ref{propheights}\,(ii), that there is a closed immersion $\varphi\colon {\XX}\hookrightarrow \PP^m_F$.
Consider the line bundle $\varphi^*\OO_{\PP^m_F}(1)$
on ${\XX}$, equipped with the metric $\textstyle{\frac{1}{2}}\varphi^*\|\cdot\|_{\can,v_0}$
at one place $v_0 \in M_F$ and with the metric
$\varphi^*\|\cdot\|_{\can,v}$ at all other places $v\neq v_0$.
This $M_F$-metrized line bundle is denoted by ${\overline{\LL}}$.
For each point $p\in {\XX}(\overline{F})$ with function field $F(p)$, there
exists a homogeneous coordinate $x_j$, considered 
as a global section of $\OO_{\PP^m_F}(1)$, such that $p\notin\left|\dvs(\varphi^*x_j)\right|$
and hence, 
\begin{align}\label{leq:1}
\h_{{\overline{\LL}}}(p)=-\!\sum_{w\in M_{F(p)}}\!\mu(w)\log \|x_j\circ\varphi(p)\|_{\can,w}
+\!\sum_{\substack{w\in M_{F(p)}\\w|v_0}}\!\mu(w)\log 2
\geq \log 2> 0.
\end{align}

We extend the group of isomorphism classes of $M_F$-metrized line bundles on ${\XX}$ by $\QQ$-coefficients and write its group structure additively.
For $i=1,\dotsc,d$, and a positive rational number $\varepsilon$, we set ${\overline{\LL}}_{i,\varepsilon}\coloneq {\overline{\LL}}_i+\varepsilon{\overline{\LL}}$.
Since ${\overline{\LL}}_i$ is nef, we obtain, by \eqref{leq:1} and the multilinearity of the heights, for each point $p\in {\XX}(\overline{F})$,
\begin{align}\label{leq:2}
\h_{{\overline{\LL}}_i,\varepsilon}(p)=\h_{{\overline{\LL}}_i}(p)+\varepsilon\h_{{\overline{\LL}}}(p)\geq \varepsilon \log 2> 0.
\end{align}

Now, we distinguish between number fields and function fields.
First, let $F$ be a number field. Since ${\overline{\LL}}_{i,\varepsilon}$ is semipositive quasi-algebraic, there exists a sequence
$({\overline{\LL}}_{i,\varepsilon,k})_{k\in \NN}$ that converges to ${\overline{\LL}}_{i,\varepsilon}$ and that consists of $M_F$-metrized line bundles which are induced
by vertically nef smooth hermitian $\QQ$-line bundles $\overline{\LLL}_{i,\varepsilon,k}$, $k\in \NN$, on a common model ${\XXX_{\varepsilon,k}}$ over the ring of integers $\OO_F$.
By Proposition \ref{proplocal2}\,(iv), we have, for all $k\in \NN$ and all $p\in {\XX}(\overline{F})$, 
\begin{align*}
\left|\h_{{\overline{\LL}}_{i,\varepsilon,k}}(p)-\h_{{\overline{\LL}}_{i,\varepsilon}}(p)\right|\leq \sum_{v\in M_{F}}\mu(v)\dist \left(\|\cdot\|_{i,\varepsilon,k,v},\|\cdot\|_{i,\varepsilon,v}\right).
\end{align*} 
Note that the sum is finite and does not depend on $p$.
Hence, by \eqref{leq:2}, there is a $k_0\in \NN$ such that for all $k\geq k_0$ and all $p\in {\XX}(\overline{F})$, 
\begin{align*}
\h_{\overline{\LLL}_{i,\varepsilon,k}}(\overline{p})=\h_{{\overline{\LL}}_{i,\varepsilon, k}}(p)\geq 0.
\end{align*}
Thus, for all $k\geq k_0$, we have nef smooth hermitian $\QQ$-line bundles  $\overline{\LLL}_{1,\varepsilon,k},\dotsc,\overline{\LLL}_{d,\varepsilon,k}$. By a result 
of Zhang (see Moriwaki's paper \cite[Proposition 2.3\,(1)]{Mor}), the proposition holds for such line bundles and hence we  get
\begin{align}\label{leq:4}
\h_{{\overline{\LL}}_{1,\varepsilon,k},\dotsc,{\overline{\LL}}_{d,\varepsilon,k}}({\XX})=\h_{\overline{\LLL}_{1,\varepsilon,k},\dotsc,\overline{\LLL}_{d,\varepsilon,k}}({\overline{\XX}})\geq 0.
\end{align}

Next, let $F$ be the function field of a smooth projective curve $C$ over any field.
Since ${\overline{\LL}}_{i,\varepsilon}$ is semipositive quasi-algebraic, there exists a sequence
$({\overline{\LL}}_{i,\varepsilon,k})_{k\in \NN}$ that converges to ${\overline{\LL}}_{i,\varepsilon}$ and that consists of $M_F$-metrized line bundles which are induced
by vertically nef $\QQ$-line bundles $\LLL_{i,\varepsilon,k}$, $k\in \NN$, on a common model $\pi_{\varepsilon,k}\colon{\XXX_{\varepsilon,k}}\rightarrow C$.
As in the number field case, we can deduce, for sufficiently large $k$'s and for all $p\in {\XX}(\overline{F})$,
\begin{align}\label{leq:4.5}
\h_{{\overline{\LL}}_{i,\varepsilon, k}}(p)\geq 0.
\end{align}
By \cite[Theorem 3.5\,(d)]{GuEq}, the height with respect to such algebraic metrized line bundles
is given as an algebraic intersection number of the associated models.
So, the inequality \eqref{leq:4.5} just says that the line bundles $\LLL_{1,\varepsilon,k},\dotsc, \LLL_{d,\varepsilon,k}$ 
on the model $\XXX_{\varepsilon,k}$
are horizontally nef. Using that they are also vertically nef, it follows from
Kleiman's Theorem \cite[Theorem III.2.1]{Kle} that
\begin{align}\label{leq:5}
\h_{{\overline{\LL}}_{1,\varepsilon,k},\dotsc,{\overline{\LL}}_{d,\varepsilon,k}}({\XX})=
\deg_C\left((\pi_{\varepsilon,k})_*(\c1(\LLL_{1,\varepsilon,k})\dotsc\c1(\LLL_{d,\varepsilon,k}))\right)
\geq 0.
\end{align}

%
%

Finally, by \eqref{leq:4} for number fields and by \eqref{leq:5} for function fields, 
continuity of heights in the metrized line bundles yields 
\[\h_{{\overline{\LL}}_1,\dotsc,{\overline{\LL}}_n}({\XX})=\lim_{\varepsilon\to 0}\h_{{\overline{\LL}}_{1,\varepsilon},\dotsc,{\overline{\LL}}_{d,\varepsilon}}({\XX})
=\lim_{\varepsilon \to 0}\lim_{k \to \infty}\h_{{\overline{\LL}}_{1,\varepsilon,k},\dotsc, {\overline{\LL}}_{d,\varepsilon,k}}({\XX})\geq 0,\]
proving the lemma.
\end{proof}


\subsection{Relative varieties over a global field} 
\label{section3.2} 

Let $F$ be a global field with the canonical $M_F$-field structure from Example \ref{globalf}.
Let $B$ be a $b$-dimensional normal proper variety over $F$ with function field $K=F(B)$.

We first endow the field $K$ with the structure of an $\Mf$-field where $\Mf$ is a natural set of places and where the positive measure is induced by fixed nef quasi-algebraic $M_F$-metrized line bundles $\MH_1,\dotsc,\MH_b$ on $B$.
This generalizes the $\Mf$-fields obtained by Moriwaki's construction in \mbox{\cite[§\,3]{Mor}}
where the function field of an arithmetic variety and a family of nef hermitian line bundles on $B$ are considered
(see also \cite[Example 11.22]{GuPisa} and \cite[\S 2]{BPS3}).

We consider a dominant morphism
 $\pi\colon \XX\rightarrow B$  of proper varieties over $F$ of relative dimension $n$. 
The generic fiber  $X=\XX\times_B\spec(K)$ of $\pi$ is a proper variety over $K$. 
Then we prove the main result of this section (Theorem \ref{THM}) showing that the intersection number
$\h_{\pi^*\MH_1,\dotsc,\pi^*\MH_b,\overline{\LL}_0,\dotsc,\overline{\LL}_n}(\XX)$ with respect to DSP quasi-algebraic $M_F$-metrized line bundles $\overline{\LL}_i$
is equal to the height $\h_{\ML_0,\dotsc,\ML_n}(X)$ with respect to induced $\Mf$-metrized line bundles $\ML_i$.
Note that the first height is a sum of local heights over $M_F$ whereas the second is an integral over $\Mf$.
This generalizes Theorem 2.4 in \cite{BPS} where the global field is a number field and only hermitian line bundles are considered. Our more general assumptions above on the metrics of the polarizations $\MH_1,\dotsc,\MH_b$ lead to the problem that also non-discrete non-archimedean places in $\Mf$ have to be considered leading to considerable difficulties in the proof of the theorem.




\begin{art}\label{Mmu}
Let $\MH_1,\dotsc,\MH_b$ be nef quasi-algebraic line bundles 
 on $B$.
By Lemma \ref{LAST}, we deduce, for every one-codimensional prime cycle $V$ on $B$, 
\begin{align}\label{hpos}
\h_{\MH_1,\dotsc,\MH_b}(V)\geq 0.
\end{align}
Let $B^{(1)}$ be the set of prime cycles of $B$ of codimension $1$. By (\ref{hpos}), the cycle $V\in B^{(1)}$ induces a non-archimedean absolute value on $K$
given, for $f\in K$, by
\begin{align}\label{AbsV}
|f|_V=\euler^{-\h_{\MH_1,\dotsc,\MH_b}(V)\,\ord_V(f)},
\end{align}
where $\ord_V$ is the discrete valuation associated to the regular local ring $\OO_{B,V}$.
We equip $B^{(1)}$ with the counting measure $\mu_{\text{fin}}$.

Let us fix a place $v\in M_F$. Then we define the \emph{generic points}\index{generic points} of $B_v^{\an}$ as
\[B_v^{\text{gen}}=B_v^{\an}\setminus \bigcup_{V\in B^{(1)}}{V_v^{\an}}.\]
Since each $V\in B^{(1)}$ is contained in the support of the divisor of a rational function,
a point $p\in B_v^{\an}$ lies in $B_v^{\text{gen}}$ if and only if, for each $f\in K^\times$, $p$ does not lie in the analytification (with respect to $v$)
of the support of $\dvs(f)$.
Thus every $p \in B_v^{\text{gen}}$ defines a well-defined absolute value on $K$ given by
\begin{align}\label{Absp}
|f|_{v,p}=|f(p)|.
\end{align}
If $v$ is non-archimedean, then this absolute value is just $p$. Let $\mu(v)$ be the weight of the product formula for $F$ in $v$ as given in Example \ref{number field} and Example \ref{function field}. 
On $B_v^{\an}$ we have the positive measure
\[\mu_v = \mu(v) \cdot \c1(\MH_{1,v})\wedge \dotsb\wedge \c1(\MH_{b,v}),\]
which is a weak limit of smooth volume  forms in the archimedean case (cf. \cite[Definition 1.4.6]{BPS}) and
defined as in Definition \ref{MeasSem} in the non-archimedean case.
Each $V^{\an}_v$, $V\in B^{(1)}$, has measure zero with respect to $\mu_v$ by Corollary \ref{MeasZero} (non-archimedean case) and by \cite[Corollaire 4.2]{CT} (archimedean case).
Since $F$ is countable, the set $B^{(1)}$ is also countable and therefore $B_v^{\an}\setminus B_v^{\text{gen}}$ has measure zero with respect to $\mu_v$.
So we get a positive measure on $B_v^{\text{gen}}$, which we also denote by $\mu_v$.

In conclusion, we obtain a measure space
\begin{align}\label{Mfield}
(\Mf,\mu)=(\Bone,\mufin)\sqcup(\bigsqcup_{v\in M_F}{\Bgen_v},\bigsqcup_{v\in M_F}{\mu_v}),
\end{align}
which is in bijection with a set of absolute values on $K$.
\end{art}

The following shows that $(K, \Mf,\mu)$ is an $\Mf$-field:
\begin{prop}\label{prodform}
Let $f\in K^\times$, then the function $\Mf\rightarrow \RR$, $w\mapsto \log|f|_w$ is integrable with respect to $\mu$ and we have the product formula
\[\int_{\Mf}{\log|f|_w \, \dint\mu(w)}=0.\]
\end{prop}
\begin{proof}
Let $f\in K^\times$ be a non-zero rational function on $B$. Then for almost every $V\in \Bone$, we have $f\in \OO_{B,V}^\times$.
Hence, the function on $\Bone$ given by $V\mapsto \log|f|_V$ is $\mufin$-integrable.

For every $v\in M_F$, the function on $\Bgen_v$ given by $p\mapsto \log|f(p)|$ is $\mu_v$-integrable
(see Theorem \ref{ifDSP}).
Since the trivially metrized line bundle $\overline{\OO}_B$ and $\MH_1,\dotsc,\MH_b$ are quasi-algebraic,
there is, by Remark \ref{ifglobal}, only a finite number of $v\in M_F$ such that
\[\int_{\Bgen_v}{\log|f(p)|\,\dint\mu_v(p)}\neq 0 .\]
Summing up, the function $\Mf\rightarrow \RR$, $w\mapsto \log|f|_w$, is $\mu$-integrable.

By the global induction formula \ref{ifglobal}, 
\begin{align*}
\int_{\Mf}{\!\log|f|_w\dint\mu(w)}&= \!\sum_{V\in \Bone}\!-\ord_V(f)\h_{\MH_1,\dotsc,\MH_b}(V)+\!\sum_{v\in M_F}\int_{B_v^{\an}}{\!\log|f(p)|\,\dint\mu_v(p)}\\
&= -\h_{\MH_1,\dotsc,\MH_b}(\cyc(f))+\sum_{v\in M_F}\int_{B_v^{\an}}{\log|f(p)|\,\dint\mu_v(p)}\\
&=-\h_{\overline{\OO}_B,\MH_1,\dotsc,\MH_b}(B)\\
&=0,
\end{align*}
which concludes the proof.
\end{proof}

\begin{art}\label{analyt}
Let $\overline{\LL}=(\LL,(\|\cdot\|_v)_{v})$ be an $M_F$-metrized line bundle on $\XX$. Then $\overline{\LL}$ induces an $\Mf$-metric on the line bundle $L=\LL\otimes K$ on $X$
given as follows:

For each $V\in \Bone$, consider 
the non-archimedean absolute value $|\cdot|_V$ on $K$ from \eqref{AbsV}
and let $\KK_V$ be the completion of an algebraic closure of the completion of $K$ with respect to $|\cdot|_V$.
We get a proper $\KK_V^\circ$-model $(\XX_V,\LL_V)\coloneq\left(\XX\times_B\spec\KK_V^\circ,\LL\otimes \KK_V^\circ\right)$ of $(X,L)$.
By Definition \ref{algmod2}, the model $(\XX_V,\LL_V)$ induces a metric $\|\cdot\|_V$ on the analytification $\Lan_V$ over $\Xan_V$ with respect to $|\cdot|_V$.

Let us fix a place $v\in M_F$. 
By \eqref{Absp}, a generic point $p\in \Bgen_v$ induces an absolute value $|\cdot|_{v,p}$ on $K$.
We denote by $\KK_{v,p}$ the completion 
of an algebraic closure of the completion 
of $K$ with respect to $|\cdot|_{v,p}$ and
by $\Xan_{v,p}$ the analytification of $X$ 
with respect to $|\cdot|_{v,p}$.
Then the projection $\XX_v\times_{B_v}\spec \KK_{v,p}\rightarrow \XX_v$ induces a morphism 
\begin{align}\label{i_p}
i_p\colon \Xan_{v,p}\rightarrow \XX^{\an}_v.
\end{align}
Note that $i_p$ is injective if $v$ is an archimedean place (cf. \cite[(2.1)]{BPS3}), but not necessarily in the non-archimedean case.
The analytification $\Lan_{v,p}$ of $L$ with respect to $|\cdot|_{v,p}$ can be identified with the line bundle $i_p^*\LL^{\an}_v$
and we equip it with the metric $\|\cdot\|_{v,p}\coloneq i_p^*\|\cdot\|_v$.
This leads to an $\Mf$-metrized line bundle 
\begin{align}\label{M-metrized}
\ML=(L,(\|\cdot\|_w)_{w\in \Mf})
\end{align}
on $X$.
\end{art}
\begin{lemma}
Let 
$\pi_v\colon\XX^{\an}_v\rightarrow\Ban_v$ be the morphism of $\FF_v$-analytic spaces induced by $\pi\colon\XX\rightarrow B$ and let
$i_p\colon \Xan_{v,p} \rightarrow \XX^{\an}_v$ be the morphism from (\ref{i_p}).
Then we have \[i_p(\Xan_{v,p})=\pi^{-1}_v{(p)}.\]
\end{lemma}
\begin{proof}
We only show this for a non-archimedean place $v$, the archimedean case is established at the beginning of \cite[§\,2]{BPS3}.
We may assume that $B=\spec(A)$ resp. $\XX=\spec(C)$ for finitely generated $F$-algebras $A$ and $C$.
Then 
$\pi$ corresponds to an injective $F$-algebra homomorphism $A\hookrightarrow C$
and we have $X=\spec(C\otimes_A K)$ with $K=F(B)=\quot(A)$.

Let $q\in \XX^{\an}_v$, that means $q$ is a multiplicative seminorm on $C\otimes_F \FF_v$ satisfying $q|_{\FF_v}=|\cdot|_v$.
Then $q$ lies in $i_p(\Xan_{v,p})$ if and only if it extends to a multiplicative seminorm $\tilde{q}$ on $C\otimes_A \KK_{v,p}$
with $\tilde{q}|_{\KK_{v,p}}=|\cdot|_{v,p}$. This is illustrated in the following diagram,
\[
\begin{xy}
  \xymatrix{
       A \otimes_{F\vphantom{\sum}}\FF_v \ar@{^{(}->}[r]  \ar@{^{(}->}[d]  &  \KK_{{v,p}\vphantom{\sum}}  \ar@{^{(}->}[d] \ar@/^/[ddr]^{|\cdot|_{v,p}}  &  \\
      C\otimes_F \FF_v \ar@{^{(}->}[r] \ar@/_/[drr]_q  &  C\otimes_A \KK_{v,p} \ar@{-->}[dr]^{\tilde{q}}  &  \\
      &  &  \RR_{\geq 0}\ .
  }
\end{xy}
\]
On the one hand, if we have such a commutative diagram, then
\[\pi_v(q)=q|_{A\otimes \FF_v}=|\cdot|_{v,p}|_{A\otimes \FF_v}=p.\]
On the other hand, if $\pi_v(q)=p$, then we have a multiplicative seminorm $\tilde{q}$ given by
\[C\otimes_A\KK_{v,p}=\left(C\otimes_F \FF_v\right)\otimes_{\left(A\otimes \FF_v\right)}\KK_{v,p}\longrightarrow \mathscr{H}\!\left(q\right)\widehat{\otimes}_{\mathscr{H}\left(p\right)} \KK_{v,p}\stackrel{y}{\longrightarrow} \RR_{\geq 0}\,,\]
where $y$ is some element of the non-empty Berkovich spectrum $\MMM\left( \mathscr{H}\!\left(q\right)\widehat{\otimes}_{\mathscr{H}\left(p\right)} \KK_{v,p}\right)$
(cf. \cite[0.3.2]{Duc}).
It follows easily that we obtain a commutative diagram as above.
This proves the result. 
\end{proof}

We need the following projection formula for heights. This formula is possible because we have one more line bundle $\MH_j$ on the base $B$ than usual (compare with Theorem \ref{THM}). 

\begin{prop}\label{genheight}
Let $\pi\colon\WW\rightarrow V$ be a morphism of proper varieties over a global field $F$ of dimensions $n+b-1$ and $b-1$ respectively, with $b,n\geq 0$.
Let $\overline{H}_i$, $i=1,\dotsc,b$, and $\overline{\LL}_j$, $j=1,\dotsc,n$, be DSP quasi-algebraic line bundles on $V$ and $\WW$ respectively.
Then
\[\h_{\pi^*\MH_1,\dotsc,\pi^*\MH_b,\overline{\LL}_1,\dotsc,\overline{\LL}_{n}}(\WW)= \deg_{\LL_1,\dotsc,\LL_n}(\WW_\eta)\h_{\MH_1,\dotsc,\MH_b}(V),
\]
where $\WW_\eta$ denotes the generic fiber of $\pi$. In particular, if $\dim(\pi(\WW))\leq b-2$, then $\h_{\pi^*\MH_1,\dotsc,\pi^*\MH_b,\overline{\LL}_1,\dotsc,\overline{\LL}_{n}}(\WW)=0$.
\end{prop}
\begin{proof}
The proof is similar as for \cite[Proposition 2.3]{BPS3} and we only sketch the additional contributions. 
By continuity of the height, we may assume that the metrics in $\MH_i$ and $\overline{\LL}_j$ are smooth or algebraic 
for all $i,j$.
We prove this proposition by induction on $n$. The case $n=0$ follows from   functoriality of the height (Proposition \ref{propheights}). 
Let $n\geq 1$. We choose any invertible meromorphic section $s_n$ of $\LL_n$ and denote by $\|\cdot\|_n=(\|\cdot\|_{n,v})_v$ the metric of $\overline{\LL}_n$.
Then the global induction formula \ref{ifglobal} implies
\begin{align*}
\h_{\pi^*\MH_1,\dotsc,\pi^*\MH_b,\overline{\LL}_1,\dotsc,\overline{\LL}_n}(\WW)&=\ \h_{\pi^*\MH_1,\dotsc,\pi^*\MH_b,\overline{\LL}_1,\dotsc,\overline{\LL}_{n-1}}(\cyc(s_n))\\
&-\!\sum_{v\in M_F}{\!\mu(v)\!\int_{\WW^{\an}_v}\!{\log\|s_n\|_{n,v} \bigwedge_{i=1}^b \c1(\pi^*\MH_{i,v})\wedge \bigwedge_{j=1}^{n-1}\c1(\overline{\LL}_{j,v})}}.
\end{align*}
If $v$ is archimedean, then $\bigwedge_{i=1}^{b}{\c1(\MH_{i,v})}$ is the zero measure on $V_v^{\an}$ since $\dim(\Van_v)=b-1$. 
Thus, the measure in the above integral vanishes and so the integral is zero.

If $v$ is non-archimedean, then the metrics in $\MH_{i,v}$, $i=1,\dotsc, b$, are induced by models $\HHH_{i}$ of $H_{i,v}^{e_{i}}$
on a common model $\VVV$ of $V_v$ over $\spec\FF_v^\circ$. By linearity, we may assume that $e_{i}=1$ for all $i$.
Analogously, the metrics in $\overline{\LL}_{j,v}$, $j=1,\dotsc, n$, are induced by models $\LLL_{j}$ of $\LL_{j,v}$
on a common model $\WWW$ of $\WW_v$. Moreover, 
we may assume that the morphism $\pi_v\colon\WW_v\rightarrow V_v$ extends to a morphism $\tau\colon\WWW\rightarrow \VVV$ over $\spec\FF_v^\circ$.
Since the special fiber $\widetilde{\VVV}$ of $\VVV$ has dimension $b-1$, the degree 
with respect to $\HHH_1,\dotsc,\HHH_b$ of a cycle of $\widetilde{\VVV}$ is zero.
Hence, for every irreducible component $Y$ of the special fiber of $\WWW$, we have by means of the projection formula,
\[\deg_{\tau^*\HHH_1,\dotsc,\tau^*\HHH_b,\LLL_1,\dotsc,\LLL_{n-1}}(Y)=\deg_{\HHH_1,\dotsc,\HHH_b}(\tau_*(\c1(\LLL_1)\dotso \c1(\LLL_{n-1}).Y))=0.\]
Therefore, for each $v\in M_F$, the measure in the above integral vanishes and so the integral is zero.

Finally, we obtain by the induction hypothesis,
\begin{align*}
\h_{\pi^*\MH_1,\dotsc,\pi^*\MH_b,\overline{\LL}_1,\dotsc,\overline{\LL}_n}(\WW)&= \h_{\pi^*\MH_1,\dotsc,\pi^*\MH_b,\overline{\LL}_1,\dotsc,\overline{\LL}_{n-1}}(\cyc(s_n))\\
& =\deg_{\LL_1,\dotsc,\LL_{n-1}}(\cyc(s_n)_\eta)\,\h_{\MH_1,\dotsc,\MH_b}(V)\\
&=\deg_{\LL_1,\dotsc,\LL_{n}}(\WW_\eta)\,\h_{\MH_1,\dotsc,\MH_b}(V),
\end{align*}
proving the result.
\end{proof}

\begin{thm}\label{THM}
Let $B$ be a $b$-dimensional normal proper variety over a global field $F$ and let $\MH_1,\dotsc,\MH_b$ be nef quasi-algebraic line bundles on $B$.
Let $K=F(B)$ be the function field of $B$ and $(\Mf,\mu)$ the associated structure of an $\Mf$-field on $K$ as in (\ref{Mfield}).

Let $\pi\colon\XX\rightarrow B$ be a dominant morphism of proper varieties over $F$ and $X$ the generic fiber of $\pi$.
Let $Y$ be an $n$-dimensional prime cycle of $X$ and $\YY$ its closure in $\XX$.
For $j=0,\dotsc,n,$ let $\ML_j$ be an $\Mf$-metrized line bundle on $X$ which is induced by a DSP quasi-algebraic line bundle $\overline{\LL}_j$ on $\XX$
as in (\ref{M-metrized}).

Then $Y$ is integrable with respect to $\ML_0,\dotsc,\ML_n$ and we have
\begin{align}\label{height}
\h_{\ML_0,\dotsc,\ML_n}(Y)=\h_{\pi^*\MH_1,\dotsc,\pi^*\MH_b,\overline{\LL}_0,\dotsc,\overline{\LL}_n}(\YY).
\end{align}
\end{thm}

\begin{proof}
By Chow's lemma (e.\,g. \cite[Theorem 13.100]{GW}) and functoriality of the height (Proposition \ref{propheights}\,(ii)), we reduce to the case when the proper varieties are projective over $F$. Then 
$\pi$ is also projective.
By (multi-)linearity of the height (Proposition \ref{propheights}\,(i)), we may assume that the line bundles $\LL_j$ are very ample and their $M_F$-metrics are semipositive.
Making a finite base change and using Proposition \ref{BCheight},
we may suppose that
$B$ and $\XX$ are geometrically integral.

We prove this theorem by induction on the dimension of $Y$.
If $\dim(Y)=-1$, thus $Y=\emptyset$, then $Y$ is integrable since the local heights of $Y$ are zero. Equation (\ref{height}) holds in this case because  $\YY$ is empty as well.

From now on we suppose that $\dim(Y)=n\geq 0$. Then the restriction $\pi|_{\YY}\colon\YY\rightarrow B$ is dominant. By Proposition \ref{propheights}\,(ii), the height does not change if we
restrict the corresponding metrized line bundles to $\YY$. 
So we may assume that $\YY=\XX$, $Y=X$ and $n=\dim(Y)=\dim(X)$.

Let $s_0,\dotsc,s_n$ be global sections of $\LL_0,\dotsc,\LL_n$ respectively, whose divisors intersect properly on $X$, and let $\rho\colon \Mf\rightarrow \RR$ be the function given by
\[w\longmapsto \lambda_{(\ML_0,s_0|_X),\dotsc,(\ML_{n},s_{n}|_X)}(X,w).\]
We must show that $\rho$ is $\mu$-integrable and that
\begin{align*}
\int_{\Mf} {\rho(w)\dint \mu (w)}=\h_{\pi^*\MH_1,\dotsc,\pi^*\MH_b,\overline{\LL}_0,\dotsc,\overline{\LL}_n}(\XX).
\end{align*}

By the induction formula of local heights (Theorem \ref{ifDSP}), there is a decomposition $\rho=\rho_1+\rho_2$ into
well-defined functions $\rho_1,\rho_2\colon \Mf\rightarrow \RR$ given by
\[ \rho_1(w)= \lambda_{(\ML_0,s_0|_X),\dotsc,(\ML_{n-1},s_{n-1}|_X)}(\cyc(s_n|_X),w) \]
and
\[\rho_2(w)=\int_{\Xan_w}{\log\|s_n|_{X_w}\|_{n,w}^{-1}\ \c1(\ML_{0,w})\wedge\dotsb\wedge\c1(\ML_{n-1,w})}.\]
Moreover, we can write the cycle $\cyc(s_n)$ in $\XX$ as
\[\cyc(s_n)=\cyc(s_n)_{\hor/B}+\cyc(s_n)_{\vert / B},\]
where $\cyc(s_n)_{\hor/B}$ contains the components which are dominant over $B$ and
$\cyc(s_n)_{\vert / B}$ contains the components not meeting $X$.

By the induction hypothesis, the function $\rho_1$ is $\mu$-integrable and
\begin{align}\label{indhyp}
\int_{\Mf}{\rho_1(w)\dint\mu(w)}&=\h_{\ML_0,\dotsc,\ML_{n-1}}(\cyc(s_n|_X))\nonumber\\
&=\h_{\pi^*\MH_1,\dotsc,\pi^*\MH_b,\overline{\LL}_0,\dotsc,\overline{\LL}_{n-1}}(\cyc(s_n)_{\hor / B}).
\end{align}

If $w=V\in \Bone$, then we deduce as in the corresponding part of the proof of \cite[Theorem 2.4]{BPS3} that 
\begin{align}\label{rho_2b}
\rho_2(V)=\sum_{\substack{\WW \in \Xone \\\pi(\WW)=V}} {\h_{\MH_1,\dotsc,\MH_b}(V)\ord_{\WW}(s_n)\deg_{\LL_0,\dotsc,\LL_{n-1}}(\WW_V)},
\end{align}
where $\WW_V$ denotes the generic fiber of $\pi|_{\WW}\colon\WW\rightarrow V$. 
This formula implies the integrability of $\rho_2$ on $\Bone$ with respect to the counting measure $\mufin$
because there are only finitely many $\WW\in \Xone$ such that $\ord_{\WW}(s_n)\neq 0$.
The same arguments as in the corresponding part of the proof of \cite[Theorem 2.4]{BPS3} show that 
\begin{equation} \label{Bone part}
 \int_{\Bone}{\rho_2(w)\dint\mufin(w)}= \h_{\pi^*\MH_1,\dotsc,\pi^*\MH_b,\overline{\LL}_0,\dotsc,\overline{\LL}_{n-1}}(\cyc(s_n)_{\vert/ B}).
\end{equation}

\medskip

Now, let $v$ be a place of $M_F$ and $p$ a generic point of $\Ban_v$.
We claim that the function
\begin{equation} \label{definition}
\rho_2(p)=\int_{\Xan_{v,p}}{\log i_p^*\|s_n\|_{n,v}^{-1}\bigwedge_{j=0}^{n-1}{\c1(i_p^*\overline{\LL}_{j,v})}}
\end{equation}
is integrable with respect to $\mu_v = \mu(v) \cdot \c1(\MH_{1,v})\wedge \dotsb\wedge \c1(\MH_{b,v})$. Furthermore, we claim that
\begin{align}\label{genint}
\int_{\Bgen_v}{\rho_2(p)\dint\mu_v(p)}= \mu(v)
\int_{\XX^{\an}_{v}}{\log\|s_n\|_{n,v}^{-1}\bigwedge_{j=0}^{n-1}{\c1(\overline{\LL}_{j,v})}\wedge\bigwedge_{i=1}^{b}{\c1(\pi^*\MH_{i,v})}}.
\end{align}
This two claims will be shown in a rather elaborate argument below.

Assuming these two claims, we will first show that the theorem now follows easily. 
By Proposition \ref{ifglobal}, the integral in (\ref{genint}) is zero for all but finitely many $v\in M_F$ because the line bundles 
$\pi^*\MH_1,\dotsc,\pi^*\MH_b$, $\overline{\LL}_0,\dotsc,\overline{\LL}_n$ are quasi-algebraic.
We conclude that the function $\rho=\rho_1+\rho_2$ is $\mu$-integrable and obtain, by using the induction hypothesis (\ref{indhyp}), (\ref{Bone part}), (\ref{genint})
and the global induction formula \ref{ifglobal},
\begin{align*}
\h_{\ML_0,\dotsc,\ML_n}(X)=&\int_M{\rho_1(w)\dint\mu(w)}+\!\int_{\Bone}{\!\rho_2(w)\dint\mufin(w)}+\!\sum_{v\in M_F}{\int_{\Bgen_v}{\!\rho_2(p)\dint\mu_v(p)}}\\
= \,&\h_{\pi^*\MH_1,\dotsc,\pi^*\MH_b,\overline{\LL}_0,\dotsc,\overline{\LL}_{n-1}}(\cyc(s_n)_{\hor / B})\\
&+\h_{\pi^*\MH_1,\dotsc,\pi^*\MH_b,\overline{\LL}_0,\dotsc,\overline{\LL}_{n-1}}(\cyc(s_n)_{\vert/ B})\\
&+\sum_{v\in M_F}{\mu(v)\int_{\XX^{\an}_{v}}{\log\|s_n\|_{n,v}^{-1}\bigwedge_{j=0}^{n-1}{\c1(\overline{\LL}_{j,v})}\wedge\bigwedge_{i=1}^{b}{\c1(\pi^*\MH_{i,v})}}}\\
=\, &\h_{\pi^*\MH_1,\dotsc,\pi^*\MH_b,\overline{\LL}_0,\dotsc,\overline{\LL}_n}(\XX),
\end{align*}
proving the theorem.

We will prove more generally that for any non-trivial $s_n \in \Gamma(\XX_v,\LL_v)$, the function $\rho_2$ in \eqref{definition} is $\mu_v$-integrable and that \eqref{genint} holds. If $v\in M_F$ is an archimedean place, then the proof of \cite[Theorem 2.4]{BPS3} shows that $\rho_2$ is $\mu_v$-integrable on $\Bgen_v$
and that the equation (\ref{genint}) holds.

From now on, we consider the case where $v\in M_F$ is non-archimedean. 
We first assume that, for each $j=0,\dotsc,n-1$ and $i=1,\dotsb,b$, the metrics on $\LL_{j,v}$ and $H_{i,v}$ are algebraic.
Then the function $\rho_2$ is $\mu_v$-integrable because $\mu_v$ is a discrete finite measure.

We choose, for each $j$, a proper model $(\XXX_{j},\LLL_{j},{e_{j}})$ of $(\XX_{v},\LL_{j,v})$ over $\spec \FF_v^\circ$ that induces
the metric of $\overline{\LL}_{j,v}$.
Note that we omit the place $v$ in the notation of the models in order not to burden the notation.
By linearity, we may assume that $e_{j}=1$ for all $j$.
Furthermore, we can suppose that the models $\XXX_{j}$ agree with a common model $\XXX$ with reduced special fiber (cf. Remark \ref{commonmod}).
In the same way, we have a proper $\FF_v^\circ$-model $\BBB$ of $B_v$ with reduced special fiber and, for each $i=1,\dotsc, b$, a model $\HHH_{i}$ of $H_{i,v}$ on $\BBB$  
inducing the corresponding metric.
As in the proof of Proposition \ref{genheight}, we can asume that the morphism $\pi_v\colon\XX_v\rightarrow B_v$ extends to a morphism
$\tau\colon\XXX\rightarrow \BBB$ over $\FF_v^\circ$.

To construct a suitable model of $X_{v,p}=X\times_{K}\spec\KK_{v,p}$ over $\KK_{v,p}^\circ$,
we consider the commutative diagram
\[
\begin{xy}
  \xymatrix{
       \spec \KK_{v,p} \ar[r] \ar[d]    & B\ar[d]  \\
      \spec \FF_v              \ar[r]   & \spec F  \ .
  }
\end{xy}
\]
By the universal property of the fiber product, we have a unique morphism $\spec \KK_{v,p} \rightarrow  B_v$. 
Because $\BBB$ is proper over $\FF_v^\circ$ and by the valuative criterion, this morphism extends to
$\spec \KK_{v,p}^\circ\rightarrow \BBB$.
Let $\XXX_{p}$ be the fiber product $\XXX\times_{\BBB}\spec \KK_{v,p}^\circ$. This is 
a model of $X_{v,p}$ over $\KK_{v,p}^\circ$.
We denote the special fibers of $\BBB$, $\XXX$ and $\XXX_{p}$ by $\SB$, $\SX$ and $\SXG$ respectively.
By \ref{f-alg}, there exists a formal admissible scheme $\XXXX_p$ over $\KK_{v,p}^\circ$ with generic fiber
$\XXXX_p^{\an}=\Xan_{v,p}$ and with reduced special fiber $\widetilde{\XXXX}_p$ such that
the canonical morphism $\iota_p\colon\widetilde{\XXXX}_p\rightarrow \widetilde{\XXX}_p$
is finite and surjective.
We obtain the following commutative diagram
\begin{align*}
\begin{xy}
  \xymatrix{
  \Xan_{v,p} \ar[r]^{=\ } \ar[d]_{\red}   & \ \Xan_{v,p}\ \ar[r]^{i_p} \ar[d]_{\red}    &\ \XX_v^{\an} \ar[r]^{\pi_v} \ar[d]_{\red} &\ \Ban_v \ar[d]_{\red}\\
    \widetilde{\XXXX}_p\ar[r]^{\iota_p}& \ \SXG        \      \ar[r]^{j_p}   & \SX \ar[r]^{\tilde{\tau}} &\ \SB\ , 
  }
\end{xy}
\end{align*}
where $\red$ is the reduction map from \ref{red} and \ref{f-alg}.
We have $\SXG=\SX\times_{\SB}\spec\widetilde{\KK}_{v,p}$.

By Definition \ref{measure}, the left-hand side of equation (\ref{genint}) is equal to
\begin{align}\label{eq:1}
&\int_{\Bgen_v}\biggl(\ \int_{\Xan_{v,p}}{\log i_p^*\|s_n\|_{n,v}^{-1}\bigwedge_{j=0}^{n-1}{\c1(i_p^*\overline{\LL}_{j,v})}}\biggr)\bigwedge_{i=1}^{b}\c1(\MH_{i,v})(p)\\
=&\!\sum_{Z\in\SB^{(0)}}
\!\biggl(\sum_{V\in \widetilde{\XXXX}_{\xi_Z}^{(0)}}{\log\|s_n( i_{\xi_Z}(\xi_V))\|_{n,v}^{-1}\deg_{\iota_{\xi_Z}^*j_{\xi_Z}^*\widetilde{\LLL}_{0},\dotsc,\iota_{\xi_Z}^*j_{\xi_Z}^*\widetilde{\LLL}_{n-1}}(V)}\!\biggr)
\deg_{\HHH_{1},\dotsc,\HHH_{b}}(Z),\nonumber
\end{align}
where $\xi_Z$ (resp. $\xi_V$) denotes the unique point whose reduction is the generic point of $Z$ (resp. $V$).

First, we consider the inner sum.
Let $Z$ be an irreducible component of $\SB$ with generic point $\eta_Z=\red(\xi_Z)$.
For $V \in \XXXX_{\xi_Z}^{(0)}$, we consider the irreducible component $W := \iota_{\xi_Z}(V)$ of $\widetilde{\XXX}_{\xi_Z}$ and the irreducible component $Y:= j_{\xi_Z}(W)$ of $\SX$. It follows from the compatibility of the reduction with morphisms of models that $i_{\xi_Z}(\xi_V)$ is the unique point $\xi_Y$ of $\XX_v^{\an}$ with reduction equal to the generic point of $Y$. 
We conclude that 
\begin{equation} \label{pushdown}
 \log\|s_n(i_{\xi_Z}(\xi_V))\|_{n,v}^{-1} = \log\|s_n(\xi_Y)\|_{n,v}^{-1}.
\end{equation}
Applying the projection formula in \cite[Proposition 4.5]{GuLocal}, we deduce that 
\begin{equation} \label{application of projection formula}
 (\iota_{\xi_Z})_*(\cyc(\widetilde{\XXXX}_{\xi_Z})) = \cyc(\widetilde{\XXX}_{\xi_Z})
\end{equation}
Now the geometric projection formula,  \eqref{pushdown} and \eqref{application of projection formula} yield 
\begin{align}\label{eq:0}
\sum_{V\in \widetilde{\XXXX}_{\xi_Z}^{(0)}}\!\log\|s_n&(i_{\xi_Z}(\xi_V))\|_{n,v}^{-1}\,\deg_{\bigl(\iota_{\xi_Z}^*j_{\xi_Z}^*\widetilde{\LLL}_{k}\bigr)_{k=0,\dotsc,n-1}}\!(V)\nonumber\\
=&\sum_{W\in \widetilde{\XXX}_{\xi_Z}^{(0)}}{\!\log\|s_n(\xi_Y)\|_{n,v}^{-1}\,\m\bigl(W,\widetilde{\XXX}_{\xi_Z}\bigr)
\,\deg_{\bigl(j_{\xi_Z}^*\widetilde{\LLL}_{k}\bigr)_{k=0,\dotsc,n-1}}\!(W)}\,,
\end{align}
where $\m\bigl(W,\widetilde{\XXX}_{\xi_Z}\bigr)$ denotes the multiplicity of $W$ in $\widetilde{\XXX}_{\xi_Z}$ and where $Y = j_{\xi_Z}(W)$.
By \cite[Ch. 0, (2.1.8)]{EGAI}, there is a bijective map 
\begin{align}\label{map1}
\left\{ Y\in \SX^{(0)}  \mid  \tilde{\tau}(Y)=Z \right\}
\longrightarrow \widetilde{\XXX}_{\eta_Z}^{(0)},\quad Y\longmapsto  Y_{\eta_Z}.
\end{align}
The special fiber of $\BBB$ is reduced and hence, applying \cite[2.4.4(ii)]{Ber} and using the compatibility of reduction and algebraic closure, we get $\widetilde{\KK}_{v,\xi_Z}=\widetilde{{\overline{{\ensuremath{\mathscr{H}\!\left(\xi_Z\right)}}}}}=\overline{\kappa(\eta_Z)}$.
Thus, \mbox{$\widetilde{\XXX}_{\xi_Z}=\SX\!\times_{\SB}\spec\widetilde{\KK}_{v,\xi_Z}$} is the base change of the fiber $\widetilde{\XXX}_{\eta_Z}=\SX\times_{\SB}\spec\kappa(\eta_Z)$ by 
$\kappa(\eta_Z)\rightarrow \overline{\kappa(\eta_Z)}$.
Thus, by \cite[Lemma 32.6.10]{stacks}, 
we obtain a surjective map
\begin{align}\label{map2}
\widetilde{\XXX}_{\xi_Z}^{(0)}\longrightarrow \widetilde{\XXX}_{\eta_Z}^{(0)}.
\end{align}
Composing the  maps \eqref{map1} and \eqref{map2}, we get a canonical surjective map
\[\widetilde{\XXX}_{\xi_Z}^{(0)}\longrightarrow \left\{ Y\in \SX^{(0)}\mid \tilde{\tau}(Y)=Z   \right\}\]
with finite fibers. More precisely, for each irreducible component $Y$ in $\SX$ with $\tilde{\tau}(Y)=Z$, the scheme
$Y_{\xi_Z}=Y\times_Z \spec\widetilde{\KK}_{v,\xi_Z}$ is a finite union of (non-necessarily reduced) irreducible components of $\widetilde{\XXX}_{\xi_Z}^{(0)}$. 
Since $i_{\xi_Z}(\xi_W)=\xi_Y$ for $W\in{Y}_{\xi_Z}^{(0)}$, we deduce
\begin{align}\label{eq:2}
\sum_{W\in \widetilde{\XXX}_{\xi_Z}^{(0)}}\!\log\|&s_n(\xi_{j_{\xi_Z}(W)})\|_{n,v}^{-1}\,\m\bigl(W,\widetilde{\XXX}_{\xi_Z}\bigr)\,\deg_{{\bigl(j_{\xi_Z}^*\widetilde{\LLL}_{k}\bigr)_{k=0,\dotsc,n-1}}}(W)\nonumber\\
=\sum_{\substack{Y\in\SX^{(0)} \\ \tilde{\tau}(Y)=Z}}\!\log\|s_n(\xi_Y)&\|_{n,v}^{-1}\,\deg_{{\bigl(j_{\xi_Z}^*\widetilde{\LLL}_{k}\bigr)_{k=0,\dotsc,n-1}}}(Y_{\xi_Z}).
\end{align}

Let $Y$ be an irreducible component of $\SX$ with generic point $\eta_Y$ such that $\tilde{\tau}(Y)=Z$. It follows from the definitions in algebraic intersection theory that we 
have
\begin{align*}
\deg_{\LLL_{0},\dotsc,\LLL_{n-1},\tau^*\HHH_{1},\dotsc,\tau^*\HHH_{b}}(Y)
=\,&\deg_{j_{\eta_Z}^*\widetilde{\LLL}_{0},\dotsc,j_{\eta_Z}^*\widetilde{\LLL}_{n-1}}(Y_{\eta_Z})\deg_{\HHH_{1},\dotsc,\HHH_{b}}(Z).
\end{align*}
Since the degree is stable unter base change, we deduce 
\begin{align}\label{eq:3}
\deg_{\LLL_{0},\dotsc,\LLL_{n-1},\tau^*\HHH_{1},\dotsc,\tau^*\HHH_{b}}(Y)
=\,&\deg_{j_{\xi_Z}^*\widetilde{\LLL}_{0},\dotsc,j_{\xi_Z}^*\widetilde{\LLL}_{n-1}}(Y_{\xi_Z})\deg_{\HHH_{1},\dotsc,\HHH_{b}}(Z).
\end{align}
%
%
Combining the equations \eqref{eq:1}, (\ref{eq:0}), \eqref{eq:2} and (\ref{eq:3}), we obtain
\begin{align*}
&\int_{\Bgen_v}\biggl(\ \int_{\Xan_{v,p}}{\log i_p^*\|s_n\|_{n,v}^{-1}\bigwedge_{j=0}^{n-1}{\c1(i_p^*\overline{\LL}_{j,v})}}\biggr)\bigwedge_{i=1}^{b}\c1(\MH_{i,v})(p)\nonumber\\
=&\sum_{Z\in\SB^{(0)}}\sum_{\substack{Y\in\SX^{(0)} \\ \tilde{\tau}(Y)=Z}}{\!\log\|s_n(\xi_Y)\|_{n,v}^{-1}\,\deg_{\LLL_{0},\dotsc\LLL_{n-1},\tau^*\HHH_{1},\dotsc,\tau^*\HHH_{b}}(Y)}\nonumber\\
=&\sum_{Y\in \SX^{(0)}}{\log\|s_n(\xi_Y)\|_{n,v}^{-1}\,\deg_{\LLL_{0},\dotsc\LLL_{n-1},\tau^*\HHH_{1},\dotsc,\tau^*\HHH_{b}}(Y)}\nonumber \\
=&\int_{\XX^{\an}_{v}}{\log\|s_n\|_{n,v}^{-1}\bigwedge_{j=0}^{n-1}{\c1(\overline{\LL}_{j,v})}\wedge\bigwedge_{i=1}^{b}{\c1(\pi^*\MH_{i,v})}}\ ,\nonumber 
\end{align*}
using in the next-to-last equality that, for an irreducible component $Y$ of $\SX$ with $\dim(\tilde{\tau}(Y))\leq b-1$, the degree is zero.
This proves equation (\ref{genint}) in the algebraic case.

\medskip

In a next step, we assume that, for each $j=0,\dotsc,n$, the metric $\|\cdot\|_{j,v}$ on $\LL_{j,v}$ is algebraic, but that the metrics on $H_{i,v}$, $i=1,\dotsc,b$,
are not necessarily algebraic.
For this case, we once again show that $\rho_2$ is $\mu_v$-integrable and that the equality (\ref{genint}) holds.

As in the previous case, we may assume that, for each $j=0,\dotsc,n$, there is a proper model
$(\LLL_j,\XXX)$ of $(\LL_{j,v},\XX_v)$ over $\FF^\circ_v$ inducing the corresponding metric.
We choose any projective model $\BBB$ over $\FF_v^\circ$ of the projective variety $B_v$ and suppose, as in the previous case,
that $\pi_v\colon\XX_v\rightarrow B_v$ extends to a proper morphism $\tau\colon\XXX\rightarrow \BBB$.
Because $\XX_v$ is projective over $\FF_v$ and by \cite[Proposition 10.5]{GuPisa}, we may assume that $\XXX$ is projective over $\FF^\circ_v$
and thus, $\tau$ is projective.
Using Serre's theorem (see \cite[Theorem 13.62]{GW}), the line bundle $\LLL_j$ is the difference of two
very ample line bundles relative to $\tau$. By multilinearity of the height, we reduce to the case where $\LLL_j$ is very ample relative to $\tau$.
Because $\BBB$ is projective over $\FF^\circ_v$, we deduce by \cite[Summary 13.71\,(3)]{GW} that
there is a closed immersion $f_j\colon \XXX\hookrightarrow \PP_{\BBB}^{N_j}$ such
that $\LLL_j\simeq f^*_j\OO_{\PP_{\BBB}^{N_j}}(1)$.

For projective spaces $\PP^{N_j}$, $j=0,\dotsc,n$, let
$\PP\coloneq \PP^{N_0}\times\dotsb\times\PP^{N_n}$ be the multiprojective space and
let $\OO_{\PP}(e_j)$ be the pullback of $\OO_{\PP^{N_j}}(1)$ by the $j$-th projection.
Since $B$ is geometrically integral, we have the function field ${K_v}=\FF_v(B_v)$ and we define \mbox{$X_v=\XX_v\times_{B_v}\spec{K_v}$} 
and $L_{j,v}=\LL_{j,v}\otimes {K_v}$. 
The product of $f_0, \dots , f_n$ is a closed immersion $f:\XXX\hookrightarrow \PP_{\BBB}$ with $\LLL_j\simeq f^*\OO_{\PP_{\BBB}}(e_j)$ for $j=0, \dots, n$.
We obtain the following commutative diagram
\[
  \xymatrix @!=0.95cm@R=0.3cm{
   &  \XXX\times_{\BBB}\spec \KK_{v,p}^\circ \  \ar@{^{(}->}[rr]^(0.6){f_p} \ar'[d]^(0.6){j_p}[dd]  &  &\  \PP_{\KK^\circ_{v,p}} \ar[dd]  \\
 \ X_{v,p}\ \ \ar@{^{(}->}[rr]^(0.35){g_p} \ar[dd]_{h_p} \ar[ru]  &  &\  \ \PP_{\KK_{v,p}}  \ar[dd] \ar[ru]  &  \\ 
   &  \ \XXX \ \ar@{^{(}->}'[r]^(0.65)f[rr]  &  &  \  \PP_{\BBB}\  .\\
   \ \ X_v\ \ \ar@{^{(}->}[rr]^g \ar[ru]  &  & \ \PP_{{K_v}} \ar[ru]  &\hspace{3em} 
  }
\]
Note that each horizontal arrow is a closed immersion because $f$ is a closed immersion and the other morphisms are obtained by base change.

Let $p\in \Bgen_v$. Then the metric $\|\cdot\|_{v,p}=i_p^*\|\cdot\|_v$ on $L_{j,v,p}=g_p^*\OO_{\PP_{\KK_{v,p}}}(e_j)$ is induced by
\[j_p^*\LLL_j=j_p^*f^*\OO_{\PP_{\BBB}}(e_j)=f_p^*\OO_{\PP_{\KK^\circ_{v,p}}}(e_j).\]
Hence, $\ML_{j,v,p}=g_p^*\overline{\OO}_{\PP_{\KK_{v,p}}}(e_j)$, where $\overline{\OO}_{\PP_{\KK_{v,p}}}(e_j)$ is endowed with the
canonical metric.
By Proposition \ref{prodform}, the field ${K_v}$ together with $(\Bgen_v, \mu_v)$ is a $\Bgen_v$-field
in the sense of Definition \ref{definition of M-field}. 
Therefore, 
\cite[Proposition 5.3.7(d)]{GuHab} says that
every $n$-dimensional cycle on $\PP_{{K_v}}$ is $\mu_v$-integrable on $\Bgen_v$
with respect to $\overline{\OO}_{\PP_{{K_v}}}(e_0),\dotsc,\overline{\OO}_{\PP_{{K_v}}}(e_n)$. Since integrability is closed under tensor product and pullback (see \ref{propint}),
the local height $\rho$ is $\mu_v$-integrable on $\Bgen_v$.
By the induction hypothesis, we deduce that $\rho_2=\rho-\rho_1$ is also $\mu_v$-integrable on $\Bgen_v$.

For proving the equality (\ref{genint}), we study $\rho$ in more detail.
We may always replace $s_n$ by  $\lambda s_n$ for a non-zero $\lambda \in \CC_v$ and  hence we may assume that $s_n \in \Gamma(\XXX,\LLL_n)$. Using $\LLL_j\simeq f^*\OO_{\PP_{\BBB}}(e_j)$ and by possibly changing the closed immersion $f$, we may assume that $s_n = g^*t_n$ for a global section of $\OO_{\PP_{{K_v}}}(e_n)$.   
We  choose global sections 
$t_j$ of $\OO_{\PP_{{K_v}}}(e_j)$, $j=0,\dotsc,n-1$, such that 
\[\left|\dvs(t_{0})\right|\cap \dotsb\cap \left|\dvs(t_{n})\right| \cap X_{v}=\emptyset.\]
Note that the original $s_0, \dots, s_{n-1}$ do not play a role anymore and so we may set $s_j:=g^*t_{j}$ for $j=0,\dots,n-1$. By Proposition \ref{proplocal2}, we get
\begin{align}\label{chform}
\rho(p)&=\lambda_{(\ML_0,s_0),\dotsc,(\ML_{n},s_{n})}(X_v,p)\nonumber\\
&=\lambda_{(\overline{\OO}_{\PP_{\KK_{v}}}(e_0),t_{0}),\dotsc,(\overline{\OO}_{\PP_{\KK_{v}}}(e_n),t_{n})}\left(X_v,p\right). 
\end{align}

We can express $\rho$ in terms of the Chow form
of the $n$-dimensional subvariety $X_v$ of the multiprojective space $\PP_{{K_v}}$. This is a multihomogenous polynomial $F_{X_v}(\boldsymbol\xi_0,\dotsc,\boldsymbol\xi_n)$ 
with coefficients in ${K_v}$ and in
the variables $\boldsymbol\xi_j=(\xi_{j0},\dotsc,\xi_{jN_j})$ viewed as dual coordinates on $\PP^{N_j}_{{K_v}}$
(see \cite[Remark 2.4.17]{GuHab} for details).
By (\ref{chform}) and \cite[Example 4.5.16]{GuHab}, we obtain
\begin{align}\label{chform2}
\rho(p)=\log|F_{X_v}|_{v,p}-\log|F_{X_v}(\mathbf{t}_0,\dotsc,\mathbf{t}_n)|_{v,p},
\end{align}
where in the first term we use the Gauss norm and in the second term $\mathbf{t}_j$ denotes the dual coordinates of $t_j$.

For each $i=1,\dotsc,n$, we choose a sequence of algebraic semipositive metrics $(\|\cdot\|_{i,v,k})_{k\in\NN}$ on $H_{i,v}$ that converges to the semipositive metric $\|\cdot\|_{i,v}$ on $H_{i}$.
Denote $\MH_{i,v,k}=\left(H_{i,v},\|\cdot\|_{i,v,k}\right)$ and set
\[\mu_{v,k}= \mu(v) \cdot \c1(\MH_{1,v,k})\wedge \dotsb\wedge \c1(\MH_{b,v,k}).\]
By Corollary \ref{max and integrability}, we obtain
\begin{align}\label{limitb}
\lim_{k\to \infty}{\int_{\Bgen_v}{\rho(p)\dint\mu_{v,k}(p)}}=\int_{\Bgen_v}{\rho(p)\dint\mu_v(p)}.
\end{align}
Analogously we can show this for the local height $\rho_1$ and hence we get
\begin{align}\label{limitrho}
\lim_{k\to \infty}{\int_{\Bgen_v}{\rho_2(p)\dint\mu_{v,k}(p)}}=\int_{\Bgen_v}{\rho_2(p)\dint\mu_v(p)}.
\end{align}
On the other hand, Corollary \ref{max and integrability} again shows   
\begin{align}\label{limith}
&\lim_{k\to \infty}{\int_{\XX^{\an}_{v}}{\log\|s_n\|_{n,v}\bigwedge_{j=0}^{n-1}{\c1(\overline{\LL}_{j,v})}\wedge\bigwedge_{i=1}^{b}{\c1(\pi^*\MH_{i,v,k})}}}\nonumber\\
=&\int_{\XX^{\an}_{v}}{\log\|s_n\|_{n,v}\bigwedge_{j=0}^{n-1}{\c1(\overline{\LL}_{j,v})}\wedge\bigwedge_{i=1}^{b}{\c1(\pi^*\MH_{i,v})}}.
\end{align}
Thus, the equality (\ref{genint}) for semipositive metrics on $H_{i,v}$ and algebraic metrics on $\LL_{j,v}$ follows from 
(\ref{limitrho}), (\ref{limith}) and the algebraic case.

\medskip

In the last step, we assume that the metrics on $H_{i,v}$ and $\LL_{j,v}$ are semipositive and not necessarily algebraic.
We choose, for each $j=0,\dotsc,n-1$,
a sequence of algebraic semipositive metrics $(\|\cdot\|_{j,v,k})_{k\in\NN}$ on $\LL_{j,v}$ that converges to $\|\cdot\|_{j,v}$.
For $p\in \Bgen_v$, we set
\[\rho_{2,k}(p)\coloneq\int_{\Xan_{v,p}}{\log i_p^*\|s_n\|_{n,v,k}^{-1}\bigwedge_{j=0}^{n-1}{\c1(i_p^*\overline{\LL}_{j,v,k})}}.\]
By the induction formula \ref{ifformal} and Proposition \ref{proplocal}\,(iii), we obtain for each $k,l\in \NN$,
\begin{align*}
&\left|\rho_{2,k}(p)-\rho_{2,l}(p)\right|\\
=&\,\Bigl|\, \lambda_{(\ML_{0,k},s_0),\dotsc,(\ML_{n,k},s_n)}(X_v,p)
-\lambda_{(\ML_{0,k},s_0),\dotsc,(\ML_{n-1,k},s_{n-1})}\left(\cyc(s_n|_{X_v}),p\right)\\
&-\lambda_{(\ML_{0,l},s_0),\dotsc,(\ML_{n,l},s_n)}(X_v,p)
+\lambda_{(\ML_{0,l},s_0),\dotsc,(\ML_{n-1,l},s_{n-1})}\left(\cyc(s_n|_{X_v}),p\right)\,\Bigr |\\
\leq &\,\sum_{j=0}^n{\dist(\|\cdot\|_{j,v,k},\|\cdot\|_{j,v,l})\deg_{L_0,\dotsc,L_{j-1},L_{j+1},\dotsc,L_n}(X_v)}\\
&+\sum_{j=0}^{n-1}{\dist(\|\cdot\|_{j,v,k},\|\cdot\|_{j,v,l})\deg_{L_0,\dotsc,L_{j-1},L_{j+1},\dotsc,L_{n-1}}(\cyc(s_n|_{X_v}))}.
\end{align*}
Hence, the sequence $(\rho_{2,k})_{k\in\NN}$ converges uniformly to $\rho_2$ on $\Bgen_v$.
Because the measure $\mu_v$ has finite total mass and, by the previous case, the functions $\rho_{2,k}$ 
are $\mu_v$-integrable,
we deduce that $\rho_2$ is $\mu_v$-integrable and that\[\lim_{k\to \infty}{\int_{\Bgen_v}{\rho_{2,k}(p)\dint\mu_v(p)}}=\int_{\Bgen_v}{\rho_2(p)\dint\mu_v(p)}.\]
Thus, using (\ref{genint}) for the functions $\rho_{2,k}$ and the induction formula \ref{ifDSP},
the equality (\ref{genint}) also holds in the case when all the metrics are semipositive.
\end{proof}


%
%

\subsection{Global heights of toric varieties over finitely generated fields}\label{GlobalToric}

Following \cite[§\,3]{BPS3} closely, we use our preceding results
to get some combinatorial formulas for the height of a fibration  with generic toric fiber.
Indeed, our non-discrete non-archimedean toric geometry 
is necessary since the measure space $\Mf$ from \eqref{Mfield} contains arbitrary non-archimedean absolute values, in contrast
to the measure space considered in \mbox{\cite[§\,1]{BPS3}}.

As usual, we fix a lattice $M\simeq \ZZ^n$ with dual $M^\vee=N$ and use the respective notations from the sections on toric geometry.
At first, we consider an arbitrary $M$-field $K$ with associated positive measure $\mu$.
Let $\Sigma$ be a complete fan in $\NR$ and let $\XS$ be the associated proper toric variety over $K$ with torus $\TT=\spec K[M]$.

\begin{art}
Let $L$ be a toric line bundle on $\XS$. An $M$-metric $\|\cdot\|=(\|\cdot\|_v)_{v\in M}$ on $L$ is \emph{toric}\index{M-metric@($M$-)metric!toric}
if, for each $v\in M$, the metric $\|\cdot\|_v$ on $L_v$ is toric (see Definition \ref{tordef}).
The \emph{canonical $M$-metric}\index{M-metric@($M$-)metric!canonical!of a toric line bundle} on $L$, denoted $\|\cdot\|_{\can}$, is given,
for each $v\in M$, by the canonical metric on $L_v$ (see Definition \ref{candef}).
We will write $\MLcan=(L,\|\cdot\|_{\can})$.

Let $s$ be a toric section on $L$ and $\Psi$ the associated virtual support function. Then a toric $M$-metric $(\|\cdot\|_v)_{v}$ on $L$ induces a family $\bigl(\psi_{\ML,s,v}\bigr)_{v\in M}$ of 
real-valued functions on $\NR$ as in Definition \ref{torfunc}.
If $\|\cdot\|$ is semipositive, then each $\psi_{\ML,s,v}$ is concave and we obtain a family $\bigl(\vartheta_{\ML,s,v}\bigr)_{v\in M}$ of concave functions on $\Delta_\Psi$
called \emph{$v$-adic roof functions}\index{roof function!$v$-adic} (cf. Definition \ref{vartheta} which works the same way in the archimedean case).
When $\ML$ and $s$ are clear from the context, we also denote $\psi_{\ML,s,v}$ by $\psi_v$ and $\vartheta_{\ML,s,v}$ by $\vartheta_v$.
%
\end{art}

\begin{prop}\label{toricheightzero}
For each $i=0,\dotsc,t$, let $\ML_i$ be a toric line bundle on $\XS$ equipped with a DSP toric $M$-metric 
and denote by $\MLcan_i$ the same 
toric line bundle endowed with the canonical $M$-metric.
Let $Y$ be either the closure of an orbit or the image of a proper toric morphism of dimension $t$.
Then $Y$ is integrable with respect to $\MLcan_0\!,\dotsc,\MLcan_t$ and 
\begin{align}\label{can-zero} 
\h_{\MLcan_0\!,\dotsc,\MLcan_t}(Y)=0.
\end{align}
Furthermore, if $Y$ is integrable with respect to $\ML_0,\dotsc,\ML_t$, then the global height is given by
\begin{align}\label{tor-integra}
\h_{\ML_0,\dotsc,\ML_t}(Y)=\int_M{\lambdator_{\ML_0,\dotsc,\ML_t}(Y,v)\dint\mu(v)},
\end{align}
where $\lambdator_{\ML_0,\dotsc,\ML_t}(Y,v)=\lambdator_{\ML_{0,v},\dotsc,\ML_{t,v}}(Y_v)$ is the toric local height from Definition \ref{toriclocaldef}.
\end{prop}
\begin{proof}
The first statement and equation (\ref{can-zero}) can be shown using the same arguments as in \cite[Proposition 5.2.4]{BPS}.
The proof is based on an inductive argument using  the local induction formula from Theorem \ref{ifformal}. 
The second equation follows easily from the first one.
\end{proof}

\begin{cor}\label{height-theta}
Let $\ML=(L,(\|\cdot\|_v)_v)$ be a toric line bundle on $\XS$ equipped with a semipositive toric $M$-metric.
Choose any toric section $s$ of $L$ and denote by $\Psi$ the corresponding support function on $\Sigma$.
If $\XS$ is integrable with respect to $\ML$, then
\[\h_{\ML}(\XS)=(n+1)!\int_M{\int_{\Delta_\Psi}\vartheta_{\ML,s,v}\dint\vol_M}\dint\mu(v).\]
\end{cor}
\begin{proof}
This is a direct consequence of Proposition \ref{toricheightzero} and 
Theorem \ref{thetheorem} which holds also in the archimedean case by \cite[Theorem 5.1.6]{BPS}.
\end{proof}

Now we consider the particular case of an $\Mf$-field which is induced by a variety over a global field as in §\,\ref{section3.2}.
Let $B$ be a $b$-dimensional normal proper variety over a global field $F$ and let $\MH_1,\dotsc,\MH_b$ be nef quasi-algebraic metrized line bundles on $B$.
This provides the function field $K=F(B)$ with the structure $(\Mf,\mu)$ of an $\Mf$-field as in (\ref{Mfield}).
Let $X$ be an $n$-dimensional normal proper toric variety over $K$ with torus $\TT=\spec K[M]$, corresponding to  a complete fan $\Sigma$ in $\NR$.
We choose a base-point-free toric line bundle $L$ on $X$ together with a toric section $s$ and denote by $\Psi$ the associated support function on $\Sigma$.

Let $\pi\colon \XX\rightarrow B$ be a dominant morphism of proper varieties over $F$ such that $X$ is the generic fiber of $\pi$.
We equip $L$ with a toric $\Mf$-metric  $\|\cdot\|$ such that $\ML=(L,\|\cdot\|)$ is induced by a semipositive quasi-algebraic $M_F$-metrized line bundle $\overline{\LL}$ on $\XX$
as in (\ref{M-metrized}).
Then it follows easily that $\ML$ is also semipositive and so, for each $v\in \Mf$, the function $\psi_v$ is concave.

The following result generalizes Corollary 3.1 in \cite{BPS3}, where the global field is $\QQ$ and the metrized line bundles are induced by models over $\ZZ$.
It is based on our main theorems \ref{thetheorem} and \ref{THM}.

\begin{cor}\label{theta-prop}
Let notation be as above. Then the function
\begin{align}\label{theta-int}
\Mf\longrightarrow\RR,\quad w\longmapsto \int_{\Delta_\Psi}{\vartheta_{\ML,s,w}(m)\dint\vol_M(m)}
\end{align}
is integrable with respect to $\mu$ and 
\begin{align}\label{theta-integral}
\h_{\pi^*\MH_1,\dotsc,\pi^*\MH_b,\overline{\LL},\dotsc,\overline{\LL}}(\XX)=\h_{\ML}(X)=(n+1)!\int_{\Mf}\int_{\Delta_\Psi}\!\vartheta_w(m)\dint\vol_M(m)\dint\mu(w).
\end{align}
\end{cor}
\begin{proof}By Theorem \ref{thetheorem} (non-archimedean case) and \cite[Theorem 5.1.6]{BPS} (archi\-me\-dean case),
we have
\begin{align*}
(n+1)!\int_{\Delta_\Psi}\vartheta_w \dint\vol_M=\lambdator_{\ML_{0,w},\dotsc,\ML_{n,w}}(X_w).
\end{align*}
Hence, Theorem \ref{THM} implies the $\mu$-integrability of the function (\ref{theta-int}).
The first equality of (\ref{theta-integral}) is Theorem \ref{THM}.
The second follows readily from (\ref{tor-integra}) and (\ref{theta-int}).
\end{proof}

\begin{prop}
We use the same notation as above. 
\begin{enumerate}
\item For each $m\in \Delta_\Psi$, the function $\Mf\longrightarrow\RR$, $w\longmapsto \vartheta_w(m)$ is $\mu$-integrable.
\item The function \[\vartheta_{\ML,s}\colon \Delta_\Psi\longrightarrow \RR, \quad m\longmapsto \int_{\Mf}\vartheta_{\ML,s,w}(m)\dint\mu(w)\]
is continuous and concave.
\item The function $\Mf\times \Delta_\Psi\longrightarrow \RR$, $(w,m)\longmapsto \vartheta_w(m)$ is integrable with respect to the measure $\mu\times \vol_M$.
\item We have
\[\h_{\pi^*\MH_1,\dotsc,\pi^*\MH_b,\overline{\LL},\dotsc,\overline{\LL}}(\XX)=\h_{\ML}(X)=(n+1)!\int_{\Delta_\Psi}{\vartheta_{\ML,s}(m)\dint\vol_M(m)},
\]
where $\vartheta_{\ML,s}$ is the function in (ii).
\end{enumerate}
\end{prop}
\begin{proof}
The proof of (i)--(iii) respectively (iv) is analogous to \cite[Theorem 3.2 respectively Corollary 3.4]{BPS3} using
Corollary \ref{theta-prop} in place of \cite[Corollary 3.1]{BPS3}. It utilizes in an essential way that $\vartheta_w$ is concave (see Theorem \ref{sem-con}). 
\end{proof}

\subsection{Heights of projectively embedded toric varieties over the function field of an elliptic curve} \label{application example}

Similarly as in \cite[§\,4]{BPS3}, we consider the formulas in §\,\ref{GlobalToric} in the case where $X$ is the normalization of a translated subtorus in a projective space using 
canonical metrics. Then 
we  illustrate the resulting formulas in the case of the function field $K$ of an elliptic curve which is a natural example where the canonical polarizations at a place of bad reduction lead to non-discrete valuations on $K$. 

Let $B$ be a $b$-dimensional normal proper variety over a global field $F$ and let $\MH_1,\dotsc,\MH_b$ be nef quasi-algebraic $M_F$-metrized line bundles on $B$.
We equip $K=F(B)$ with the structure $(\Mf,\mu)$ of an $\Mf$-field as in (\ref{Mfield}). 

For $r\geq 1$, let us consider the projective space $\PP^r_B=\PP^r_F\times_F B$ over $B$ and Serre's twisting sheaf $\OO_{\PP^r_B}(1)$. We equip $\OO_{\PP^r_B}(1)$ with the metric
obtained by pulling back the canonical $M_F$-metric of $\OO_{\PP^r_F}(1)$ and denote this metrized line bundle by $\overline{\OO(1)}=\overline{\OO_{\PP^r_B}(1)}$.

For $\bfm_{j}\in \ZZ^{n}$ and $f_{j}\in K^{\times}$, $j=0,\dots,r$, we regard the morphism
\[
\GG^{n}_{\text{m},K}\longrightarrow \PP^{r}_{K}, \quad 
  \bft\longmapsto (f_{0}\bft^{\bfm_{0}}:\dotsb:f_{r}\bft^{\bfm_{r}}),
\]
where $f_{j}\bft^{\bfm_{j}}=f_{j}t_{1}^{m_{j,1}}\dotsm t_{n}^{m_{j,n}}$.
For simplicity, we suppose that $\bfm_{0}= \bf 0$, $f_0=1$ and that $\bfm_{0},\dotsc,\bfm_r$ generate $\ZZ^{n}$ as an abelian group.
Let $Y$ be the closure of the image of this morphism.
Then $Y$ is a  toric  variety over $K$, but  not necessarily normal.

Let $\YY$ be the closure of $Y$ in $\PP^r_B$ and let $\pi\colon \YY\rightarrow B$ be the 
restriction of $\PP^r_B\rightarrow B$.
Our goal is to compute the height $\h_{\pi^*\MH_1,\dotsc, \pi^*\MH_b,\overline{\OO(1)},\dotsc, \overline{\OO(1)}}(\YY)$
using formula (\ref{theta-integral}).
Since $Y$  is not necessarily normal, we consider the normalization $\XX$ of $\YY$ and the induced dominant morphism $\XX\rightarrow B$ which we also denote by $\pi$.
Then the generic fiber $X=\XX\times_B\spec K$ is a normal $\GG^n_{\text{m},K}$-toric variety over $K$.
Let $\overline{\LL}$ be the pullback of $\overline{\OO(1)}$ via $\XX\rightarrow \PP_{B}^r$ and let $\ML$ be the induced $\Mf$-metrized line bundle on $X$ as in (\ref{M-metrized}).
Then $\ML$ is a toric semipositive $\Mf$-metrized line bundle on $X$.

Analogously to \cite[Proposition 4.1]{BPS3}, we can explicitly describe the associated $w$-adic roof functions as follows:
\begin{prop}\label{height-envelope}
We keep the above notations and let $s$ be the toric section of $L$ induced by the global section $x_0$ of $\OO(1)$. Then the polytope $\Delta$ associated to $(L,s)$ is determined by  
\[\Delta=\conv(\bfm_0,\dotsc,\bfm_r ).\]
For $w\in \Mf$, the graph of the $w$-adic roof function $\vartheta_w\colon \Delta\rightarrow\RR$ is the upper envelope of the polytope
$\Delta_w\subseteq \RR^n\times \RR$ which is given by
\[
  \Delta_{w}=\begin{cases}
   \conv\bigl((\bfm_{j},-\h_{ \MH_{1},\dots,\MH_{b}}(V)\ord_{V}(f_{j}))_{j=0,\dots, r}\bigr),
&\text    { 
      if }w=V\in B^{(1)}, \\
    \conv\bigl((\bfm_{j},\log|f_{j}(p)|_v)_{j=0,\dots, r}\bigr),&\text{
      if }w=p\in \Bgen_v, v\in M_F.
  \end{cases}   
  \]
\end{prop} 

Now we differ from the setting in \cite[§\,4]{BPS3} and consider the special case of the function field of an elliptic curve
equipped with a canonical metrized line bundle. 
Note that in this case non-discrete non-archimedean absolute values naturally occur. 

\begin{art}\label{ellcurve}
From now on, we assume that  $B$ is an elliptic curve $E$ over the global field $F$ and let $H$ be an ample symmetric line bundle on $E$.
We choose any rigidification $\rho$ of $H$, i.\,e. $\rho\in H_0(F)\setminus\{0\}$.
By the theorem of the cube, we have, for each $m\in \ZZ$, a canonical identification $[m]^*H=H^{\otimes m^2}$ of rigidified line bundles.
Then there exists a unique $M_F$-metric $\|\cdot\|_\rho=(\|\cdot\|_{\rho,v})_v$ on $H$ such that, for all $v\in M_F$, $m\in \ZZ$,
\[[m]^*\|\cdot\|_{\rho,v}=\|\cdot\|_{\rho,v}^{\otimes m^2}.\]
For details, see \cite[Theorem 9.5.7]{BG}.
We call such an $M_F$-metric \emph{canonical}\index{M-metric@($M$-)metric!canonical!over an elliptic curve} because
it is canonically determined by $H$ up to $(|a|_v)_{v\in M_F}$ for some $a\in F^\times$.
By \cite[3.5]{GuBo}, the canonical metric $\|\cdot\|_\rho$ is quasi-algebraic and, since $H$ is ample and symmetric, it is semipositive.

The global height associated to $\MH=(H,\|\cdot\|_\rho)$ is equal to the Néron-Tate height $\hat{\h}_H$ (see \cite[Corollary 9.5.14]{BG}). 
In particular, it does not depend on the choice of the rigidification $\rho$.
Since $H$ is ample, we have $\h_{\MH}=\hat{\h}_H\geq 0$.

For each $v\in M_F$, the canonically metrized line bundle $\MH$ induces the \emph{canonical measure}\index{canonical measure!over an elliptic curve} $\c1(\MH_v)=\c1(H_v,\|\cdot\|_{\rho,v})$
which does not depend on the choice of the rigidification (cf. \cite[3.15]{GuTrop}) and which is positive.
It has the properties \[ \c1(\MH_v)(E_v^{\an})=\deg_H(E)\quad \text{and} \quad [m]^*\c1(\MH_v)=m^2\c1(\MH_v) \text{ for all }  m\in \ZZ.\]
For a detailed description of these measures, we have to consider three kinds of places $v\in M_F$.

(i) The set of archimedean places is denoted by $M_F^\infty$. For $v$ archimedean, $E_v^{\an}=E(\CC)$ is a complex analytic space which is biholomorphic to a complex torus $\CC/(\ZZ+\ZZ\tau)$, $\Im\tau >0$. We have 
 $\c1(\MH_v)=\deg_H(E)\,\muH$ for the Haar probability measure $\muH$ on this torus.

(ii) The set of non-archimedean places $v$ with $E$ of good reduction at $v$ is denoted by $M_F^{\text{\upshape g}}$. For such a  $v$, the canonical measure $\c1(\MH_v)$ is a Dirac measure at a single point of $E_v^{\an}$.
Since $E$ has good reduction at $v$, there is a smooth proper scheme $\EEE_v$ over $\FF_v^\circ$ with generic fiber $E_v$. The special fiber $\widetilde{\EEE}_v$ is an elliptic curve over $\widetilde{\FF}_v$.
Let $\xi_v$ be the unique point of $\Ean_v$ such that $\red(\xi_v)$ is the generic point of $\widetilde{\EEE}_v$. Then $\c1(\MH_v)=\deg_H(E)\,\delta_{\xi_v}$.

(iii) The set of non-archimedean places $v$ with $E$ of bad reduction at $v$ is denoted by $M_F^{\text{\upshape b}}$. Then $\Ean_v$ is a Tate elliptic curve over $\FF_v$, i.\,e. $\Ean_v$ is isomorphic as an analytic
group to $\GG_{\text{m},v}^{\an}/q^{\ZZ}$, where $\GG_{\text{m},v}$ is the multiplicative group over $\FF_v$ with fixed coordinate $x$ and
$q$ is an element of $\GG_{\text{m},v}(\FF_v)=\FF_v^\times$ with $|q|_v<1$ (see, for example, \cite[9.7.3]{BGR}).
Denote by $\trop\colon \GG_{\text{m},v}^{\an}\rightarrow \RR$, $p\mapsto -\log p(x)$, the tropicalization map and set $\Lambda_v\coloneq -\log|q|_v\ZZ$. Then we obtain a commutative diagram
\[
\begin{xy}
  \xymatrix{
       \GG_{\text{m},v}^{\an} \ar[r]^\trop \ar[d]    & \RR \ar[d]  \\
     \Ean_v            \ar[r]^{\overline{\trop}\ }   &\, \RR/\Lambda_v  .
  }
\end{xy}
\]
Consider the continuous section $\rho\colon\RR\rightarrow \GG_{\text{m},v}^{\an}$ of $\trop$, where $\rho(u)$ is given by
\begin{align}\label{skeletmap}
\sum_{m\in \ZZ}{\alpha_m x^m}\longmapsto \max_{m\in \ZZ}|\alpha_m|\exp(-m\cdot u)
\end{align}
as in \ref{torfunc2}.
Using $\Ean_v=\GG_{\text{m},v}^{\an}/q^{\ZZ}$, this section descends to a continuous section $\bar\rho:\RR/\Lambda_v \to \Ean_v$ of $\overline{\trop}$. The image of $\bar\rho$ 
is a canonical subset $S(\Ean_v)$ of $\Ean_v$ which is called the \emph{skeleton}\index{skeleton} of $\Ean_v$.
By \cite[Ex. 5.2.12 and Thm. 6.5.1]{Ber}, this is a closed subset of $\Ean_v$ and $\overline{\trop}$ restricts to a homeomorphism from $S(\Ean_v)$ onto $\RR/\Lambda_v$.
By \cite[Corollary 9.9]{GuTrop}, the canonical measure $\c1(\MH_v)$ on $\Ean_v$ is supported on the skeleton $S(\Ean_v)$ and we have $\c1(\MH_v)=\deg_H(E) \, \muH$ for the 
Haar probability measure $\muH$ on $\RR/\Lambda_v$. 

Let $\bfm_0= {\bf 0} \in \ZZ^n$ and $\bfm_1 \dots, \bfm_r \in \ZZ^n$ generating $\ZZ^n$ as a group and let $f_0, \dots f_r \in K^\times=F(E)^\times$ with $f_0=1$. Recall that we consider the morphism
\[
\GG_{\text{m},K}^n\longrightarrow \PP^{r}_{K}, \quad 
  \bft\longmapsto (f_{0}\bft^{\bfm_{0}}:\dotsb:f_{r} \bft^{\bfm_{r}}).
\]
The closure of the image of this morphism  in $\PP^r_E$ is denoted by $\YY$.
\end{art}

\begin{cor}\label{elliptich}
With the assumptions and notations from \ref{ellcurve}, we have
\begin{align*}
&\frac{1}{(n+1)! \cdot\deg_H(E)}\h_{\pi^*\MH,\overline{\OO(1)},\dotsc,\overline{\OO(1)}}(\YY)\\
= \ &\frac{1}{\deg_H(E)}\! \sum_{P\in C}\int_\Delta\vartheta_P(x)\dint\vol(x)
+\!\sum_{v\in M_F^{\infty}}\int_{E(\CC_v)}\int_\Delta\vartheta_p(x)\dint\vol(x)\dint\muH(p)\\
&+  \sum_{v\in M_F^{\text{\upshape g}}}\int_\Delta\vartheta_{\xi_v}(x)\dint\vol(x)
+ \sum_{v\in M_F^{\text{\upshape b}}}\int_{\RR/\Lambda_v}\int_\Delta\vartheta_{\bar\rho(\overline{u})}(x)\dint\vol(x)\dint\muH(\overline{u})  \ ,
\end{align*}
where $C\!\subset \! E^{(1)}$ is the set of irreducible components of the divisors $\cyc(f_0),\dotsc,\cyc(f_r)$ and $\vol$ is the standard measure on $\RR^n$.  
\end{cor}
\begin{proof}
We have $\h_{\pi^*\MH,\overline{\OO(1)},\dotsc,\overline{\OO(1)}}(\YY)=\h_{\pi^*\MH,\overline{\LL},\dotsc,\overline{\LL}}(\XX)$ because the global height is invariant under normalization.
We get the result by Theorem \ref{THM}, Corollary \ref{height-theta}, Proposition \ref{height-envelope} and the description in \ref{ellcurve}.
\end{proof}



\numberwithin{thm}{section}

\appendix 
\section{Convex geometry}\label{A}

In this appendix, we collect the notions of convex geometry that we need for the study of toric geometry.
We follow the notation of \cite[§\,2]{BPS} and \cite{Guide} which is based on the classical book \cite{Roc}.

Let $M$ be a free abelian group of rank $n$ and $N\coloneq M^\vee\coloneq\Hom(M,\ZZ)$ its dual group. The natural pairing between $m\in M$ and $u\in N$ is denoted by 
$\la m,u \ra \coloneq u(m)$.
If $G$ is an abelian group, we set $N_G\coloneq  N\otimes_{\ZZ} G=\Hom(M,G)$. In particular, $\NR = N\otimes_{\ZZ} \RR$ is an $n$-dimensional real vector space with dual space $\MRR=\Hom(N,\RR)$.
We denote by $\Gamma$ a subgroup of $\RR$.
\begin{art}\label{a1}
A \emph{polyhedron} $\Lambda$ in $\NR$ is a non-empty set defined as the intersection of finitely many closed half-spaces, i.\,e.
\begin{align}\label{polyhedron}
\Lambda=\bigcap_{i=1}^r{\left\{u\in \NR\mid\la m_i,u\ra\geq l_i\right\}} \ \text{ where } m_i\in \MRR, l_i\in \RR.
\end{align}
A \emph{polytope} is a bounded polyhedron. 
A \emph{face} $\Lambda'$ of a polyhedron $\Lambda$, denoted by $\Lambda'\preceq\Lambda$, is either $\Lambda$ itself or of the form $\Lambda\cap H$ where $H$ is the boundary of a closed half-space  containing $\Lambda$.
A face of $\Lambda$ of codimension $1$ is called a \emph{facet}, a face of dimension $0$ is a \emph{vertex}.
The \emph{relative interior} of $\Lambda$, denoted by $\ri{\Lambda}$, is the interior of $\Lambda$ in its affine hull.
\end{art}

\begin{art}\label{a2}
Let $\Lambda$ be a polyhedron in $\NR$.
We call $\Lambda$ \emph{strongly convex} if it does not contain any affine line.
We say that $\Lambda$ is \emph{$\Gamma$-rational} if there is a representation as (\ref{polyhedron}) with $m_i\in M$ and $l_i\in \Gamma$.
If $\Gamma=\QQ$, we just say $\Lambda$ is \emph{rational}.
We say that a polytope in $\MRR$ is \emph{lattice} if its vertices lie in $M$.
\end{art}

\begin{art}\label{a3}
A \emph{polyhedral cone} in $\NR$ is a polyhedron $\sigma$ such that $\lambda\sigma=\sigma$ for all $\lambda\geq 0$.
Its \emph{dual} is defined as
\[\sigmad\coloneq \left\{m\in \MRR \mid  \la m,u\ra\geq 0 \ \forall u \in \sigma\right\}.\]
A polyhedral cone is strongly convex if and only if $\dim(\sigmad)=0$.
We denote by $\sigma^\bot$ the set of $m\in \MRR$ with $\la m,u\ra=0$ for all $u\in \sigma$.
The \emph{recession cone} of a polyhedron $\Lambda$ is defined as
\[\rec(\Lambda)\coloneq \{u\in \NR\mid u+\Lambda\subseteq \Lambda\}.\]
If $\Lambda$ has a representation as (\ref{polyhedron}), the recession cone can be written as
\[\rec(\Lambda)=\bigcap_{i=1}^r\{u\in \NR\mid\la m_i,u\ra\geq 0\}.\]
\end{art}

\begin{art}\label{a4}
A \emph{polyhedral complex} $\Pi$ in $\NR$ is a non-empty finite set of polyhedra such that
\begin{enumerate}
\item every face of $\Lambda\in \Pi$ lies also in $\Pi$;
\item if $\Lambda,\Lambda'\in \Pi$, then $\Lambda\cap\Lambda'$ is empty or a face of $\Lambda$ and $\Lambda'$.
\end{enumerate}
A polyhedral complex $\Pi$ is called \emph{$\Gamma$-rational} (resp. \emph{rational}, resp. \emph{strongly convex}) if each $\Lambda\in \Pi$ is $\Gamma$-rational (resp. rational, resp. strongly convex).
The \emph{support} of $\Pi$ is defined as the set $|\Pi|\coloneq \bigcup_{\Lambda\in \Pi}\Lambda$.
We say that $\Pi$ is \emph{complete} if $|\Pi|=\NR$. 
We will denote by $\Pi^k$ the subset of $k$-dimensional polyhedra of $\Pi$.

A \emph{fan} in $\NR$ is a polyhedral complex in $\NR$ consisting of strongly convex rational polyhedral cones.
\end{art}

\begin{art}
Let $\Pi$ be a polyhedral complex in $\NR$. The \emph{recession} $\rec(\Pi)$ of $\Pi$ is defined as
\[\rec(\Pi)=\{\rec(\Lambda)\mid\Lambda\in\Pi\}.\]
If $\Pi$ is a complete $\Gamma$-rational strongly convex polyhedral complex, then $\rec(\Pi)$ is a complete fan in $\NR$.
\end{art}

\begin{art}
Let $C$ be a convex set in a real vector space. A function $f\colon C \rightarrow \RR$ is \emph{concave} if 
\begin{align}\label{concave} f(tu_1+(1-t)u_2)\geq tf(u_1)+(1-t)f(u_2)\end{align}
for all $u_1,u_2\in C$ and $0<t<1$.

Note that we use the same terminology as in convex analysis. In the classical books of toric varieties \cite{KKMS}, \cite{Fu2}, \cite{CLS}, our concave functions are called ``convex''.
\end{art}

\begin{art}\label{a5}
Let $f$ be a concave  function on $\NR$. We define the \emph{stability set} of $f$ as
\[\Delta_f\coloneq \{m\in \MRR \mid \la m,\cdot\ra - f \text{ is bounded below}\}.\]
This is a convex set in $\MRR$. 
The \emph{(Legendre-Fenchel) dual} of $f$ is the function
\[f^\vee\colon \Delta_f\longrightarrow \RR, \quad m \longmapsto \inf_{u\in \NR}(\la m, u\ra-f(u)). \]
It is a continuous concave function. 
\end{art}

\begin{art}
Let $f\colon \NR\rightarrow \RR$ be a concave function.
The \emph{recession function} $\rec(f)$ of $f$ is defined as
\[\rec(f):\NR\longrightarrow \RR,\quad u\longmapsto \lim_{\lambda\to\infty}{\frac{f(\lambda u)}{\lambda}}\ .\]
By \cite[Theorem 13.1]{Roc}, $\rec(f)$ is the support function of the stability set $\Delta_f$, i.\,e. it is given by 
\[ \rec(f)(u)=\inf_{m\in\Delta_f}\la m,u\ra \]
for $u\in \NR$.
\end{art}

\begin{prop}\label{con-func}
Let $\Sigma$ be a complete fan in $\NR$ and let $\Psi\colon \NR\rightarrow \RR$ be a 
 support function on $\Sigma$ 
(Definition \ref{SF-Def}).
Then the assignment $\psi\mapsto\psi^\vee$ gives a bijection between the sets of
\begin{enumerate}
\item concave functions $\psi$ on $\NR$ such that $|\psi-\Psi|$ is bounded,
\item continuous concave functions on $\Delta_\Psi$.
\end{enumerate}
\end{prop}
\begin{proof}
This follows from the propositions 2.5.20\,(2) and 2.5.23 in \cite{BPS}.
\end{proof}

\begin{art} \label{definition piecewise affine function}
A function $f\colon \NR\rightarrow \RR$ is \emph{piecewise affine} if there is a finite cover $\{\Lambda_i\}_{i\in I}$ of $\NR$ by closed subsets such that $f|_{\Lambda_i}$ is an affine function.

Let $\Pi$ be a complete polyhedral complex in $\NR$. We say that $f$ is a piecewise affine function \emph{on} $\Pi$ if $f$ is affine on each polyhedron of $\Pi$.
\end{art}

\begin{art} \label{rationality for piecewise affine functions}
Let $f\colon \NR\rightarrow \RR$ be a piecewise affine function on $\NR$. 
Then there is a complete polyhedral complex $\Pi$ in $\NR$ such that, for each $\Lambda\in\Pi$,
\begin{align}\label{dvectors}
f|_\Lambda(u) =\la m_\Lambda,u \ra+l_\Lambda\quad \text{with } (m_\Lambda,l_\Lambda) \in\MRR\times \RR\ .
\end{align}
The set  $\{(m_\Lambda,l_\Lambda)\}_{\Lambda\in\Pi}$  is called \emph{defining vectors} of $f$.
We call $f$ a \emph{$\Gamma$-lattice} function if it has a representation as (\ref{dvectors}) with $(m_\Lambda,l_\Lambda)\in M\times \Gamma$ for each $\Lambda \in\Pi$.
We say that $f$ is a \emph{$\Gamma$-rational} piecewise affine function if there is an integer $e>0$ such that $ef$ is a $\Gamma$-lattice function.
\end{art}

\begin{art}\label{a12}
Let $f$ be a concave piecewise affine function  on $\NR$. Then there are $m_i\in \MRR$, $l_i\in \RR$, $i=1,\dotsc,r$, such that $f$ is given by
\begin{align*}
f(u)=\min_{i=1,\dotsc,r}{\la m_i,u \ra}+l_i \quad \text{for } u\in\NR. 
\end{align*}
The stability set $\Delta_f$ is a polytope in $\MRR$ which is the convex hull of $m_1,\dotsc,m_r$.

The recession function $\rec(f)$ has integral slopes 
if and only if the stability set $\Delta_f$ is a lattice polytope.
\end{art}

\begin{art}
Let $f$ be a piecewise affine function on $\NR$.
Then we can write $f=g-h$, where $g$ and $h$ are concave piecewise affine functions on $\NR$.
The \emph{recesssion function} of $f$ is defined as $\rec(f)=\rec(g)-\rec(h)$.
%
%
%
\end{art}

In Theorem \ref{sem-con} we need the following assertion.
\begin{prop} \label{sup-conc}
Let $\Gamma$ be a non-trivial subgroup of $\RR$.
Let $\Psi$ be a support function on a complete fan in $\NR$ (Definition \ref{SF-Def}) and $\psi$ a concave function on $\NR$ such that $|\psi-\Psi|$ is bounded.
Then there is a sequence of $\Gamma$-rational piecewise affine concave functions $(\psi_k)_{k\in\NN}$ with $\rec(\psi_k)=\Psi$,
that uniformly converges to $\psi$.
\end{prop}
\begin{proof}
Since $\Psi$ is a support function with bounded $|\psi-\Psi|$, the stability set $\Delta_\Psi$ is a lattice polytope in $\MRR$ with  $\Delta_\Psi=\Delta_\psi$.
Thus, by Proposition \cite[Proposition 2.5.23\,(2)]{BPS}, there is a sequence of piecewise affine concave functions $(\psi_k)_{k\in \NN}$ with $\Delta_{\psi_k}=\Delta_\Psi$,
that converges uniformly to $\psi$.
Because the divisible hull of $\Gamma$ lies dense in $\RR$,  we may assume that the $\psi_k$'s are $\Gamma$-rational.
Finally, Proposition 2.3.10 in \cite{BPS} says that $\Delta_{\psi_k}=\Delta_\Psi$ implies $\rec(\psi_k)=\Psi$.
\end{proof}

\begin{art}\label{sup-dif}
Let $f$ be a concave function on $\NR$. The \emph{sup-differential} of $f$ at $u\in \NR$ is defined as
\[\partial f(u)\coloneq \{m\in \MRR\mid \la m,v-u\ra\geq f(v)-f(u)\text{ for all }v\in \NR\}.\]
For each $u\in \NR$, the sup-differential $\partial f(u)$ is a non-empty compact convex set.
For a subset $E$ of $\NR$, we set
\[\partial f(E)\coloneq \bigcup_{u\in E}{\partial f (u)}.\]
\end{art}

\begin{rem} \label{piecewise affine and sup-differential}
Let $f$ be a concave function on $\NR$ which is piecewise affine on a complete fan. Then the stability set $\Delta_f$ is equal to the sup-differential $\partial f(0)$. This follows easily from the definitions. 
\end{rem}

\begin{art}\label{a8}
We denote by $\vol_M$ the Haar measure on $\MRR$ such that $M$ has covolume one.
Let $f$ be a concave function on $\NR$.
The \emph{Monge-Ampère measure} of $f$ with respect to $M$ is defined, for any Borel subset $E$ of $\NR$, as
\[\M_M(f)(E)\coloneq \vol_M(\partial f(E)).\]
We have for the total mass $\M_M(f)(\NR)=\vol_M(\Delta_f)$.
\end{art}

\begin{prop}\label{a16}
Let $(f_k)_{k\in\NN}$ be a sequence of concave functions on $\NR$ that converges uniformly to a function $f$.
Then the Monge-Ampère measures $\M_M(f_k)$ converge weakly to $\M_M(f)$.
\end{prop}
\begin{proof}This follows from \cite[Proposition 2.7.2]{BPS}.
\end{proof}

\begin{prop}\label{a17}
Let $f$ be a piecewise affine concave function on a complete polyhedral complex $\Pi$ in $\NR$.
Then
\[\M_M(f)=\sum_{v\in \Pi^0}\vol_M(\partial f(v))\,\delta_v,\]
where $\delta_v$ is the Dirac measure supported on $v$.
\end{prop}
\begin{proof}
This is \cite[Proposition 2.7.4]{BPS}.
\end{proof}

\begin{art}\label{pol-fan}
Let $\Delta$ be an $n$-dimensional lattice polytope in $\MRR$ and let $F$ be a face of $\Delta$. Then we set
\[\sigma_F\coloneq \left\{u\in \NR\mid \la m- m',u\ra\geq 0\text{ for all }m\in\Delta,\, m'\in F \right\}.\]
This is a strongly convex rational polyhedral cone which is normal to $F$.
By setting $\Sigma_\Delta\coloneq \{\sigma_F\mid F \preceq \Delta\}$, we obtain a complete fan in $\NR$.
We call $\Sigma_\Delta$ the \emph{normal fan} of $\Delta$.
The assignment $F\mapsto \sigma_F$ defines a bijective order reversing correspondence between faces of $\Delta$ and cones of $\Sigma_\Delta$.
The inverse map sends a cone $\sigma$ to the face 
\begin{align}\label{Fsigma}
F_\sigma\coloneq \{m\in \Delta\mid \la m'-m,u\ra\geq 0\text{ for all }m'\in \Delta,\, u\in\sigma\}.
\end{align}
For details, we refer to \cite[§\,2.3]{CLS}.

We also use the notation $F_\sigma$ in the following situation.
Let $\Sigma$ be a fan in $\NR$ and $\Psi$ a support function on $\Sigma$ with associated lattice polytope $\Delta_\Psi$.
For $\sigma\in \Sigma$, we denote by $F_\sigma$ the face of $\Delta_\Psi$ given as in (\ref{Fsigma}).
\end{art}

\begin{art}\label{a9}
Let $\Delta$ be a lattice polytope in $\MRR$. We denote by $\aff(\Delta)$ the affine hull of $\Delta$ and by $\Lb_\Delta$ the linear subspace of $\MRR$ associated to $\aff(\Delta)$.
Then $M(\Delta)\coloneq M\cap \Lb_\Delta$ defines a lattice in $\Lb_\Delta$.
The measure $\vol_{M(\Delta)}$ on $\Lb_\Delta=M(\Delta)_\RR$ (see \ref{a8}) induces a measure on $\aff(\Delta)$ which we also denote by $\vol_{M(\Delta)}$.

If $\Delta$ is $n$-dimensional and $F$ is a facet of $\Delta$, we denote by $v_F\in N$ the unique minimal generator of the ray
$\sigma_F\in \Sigma_\Delta$ (see \ref{pol-fan}).
We call $v_F$ the \emph{primitive inner normal vector} to $F$.
\end{art}

\begin{prop}\label{a21}
Let $f$ be a concave function on $\NR$ such that the stability set $\Delta_f$ is a lattice polytope of dimension $n$.
With the notations in \ref{a9} we have
\[-\int_{\NR}{f\dint\M_M(f)}=(n+1)\int_{\Delta_f}{f^\vee\dint \vol_M}+\sum_F \la F, v_F\ra\int_F {f^\vee \dint\vol_{M(F)}},\]
where the sum is over the facets $F$ of $\Delta_f$.
\end{prop}
\begin{proof}
This is \cite[Corollary 2.7.10]{BPS}.
\end{proof}

\section{Strictly semistable models} \label{B}

In this appendix, we will see that the theory of strictly semistable models has strong similarities to the theory of toric schemes and we will use that to prove a semipositivity statement on a toric scheme of relative dimension $1$ which will be useful in §\,\ref{Semipositive metrics and measures on toric varieties}. On the way, we will prove some new results for formal models related to regular subdivisons of the skeleton of a given strictly semistable formal scheme.

In this appendix, 
$K$ is an algebraically closed field endowed with a complete non-trivial non-archimedean absolute value $|\phantom{a}|$, valuation $\val \coloneq -\log|\phantom{a}|$ and corresponding valuation ring $K^\circ$. 

\begin{art} \label{strictly semistable}
 A {\it  strictly semistable formal scheme} $\XXXX$  is a
  connected quasi-compact 
  admissible formal scheme $\XXXX$ over ${K^\circ}$ which is
  covered by formal open subsets $\UUUU$ admitting an \'etale morphism 
  \begin{equation}
    \label{eq:formal.sss.pair}
    \psi ~:~ \UUUU \longrightarrow 
    \spf\big( K^\circ\la x_0, \dotsc, x_d \ra / \la x_0 \dotsb x_r - \pi \ra \big) 
  \end{equation}
  for $r \leq d$ and $\pi \in K^\times$ with $|\pi|<1$. 

A {\it strictly semistable formal model over ${K^\circ}$} of a proper algebraic variety $X$ is a strictly semistable formal scheme which is a formal $K^\circ$-model of $X$.
\end{art}

\begin{art} \label{stratification}
Let $V_1, \dotsc, V_R$ be the irreducible components of the special fiber of a strictly semistable formal scheme $\XXXX$ over ${K^\circ}$. Then the special fiber $\widetilde{\XXXX}$
has  a stratification, where a {\it stratum} $S$ is given
as an irreducible component of 
$\bigcap_{i \in I} V_i \setminus \bigcup_{i \not \in I} V_i$ for any 
$I \subset \{1, \dotsc, R\}$. We get a partial order on the set of strata by using $S \leq T$ if and only if $\overline{S} \subset \overline{T}$.

A formal open subset $\UUUU$ as in \ref{strictly semistable} is called a {\it building block} if the special fiber $\widetilde{\UUUU}$ has a smallest stratum. This stratum is called the {\it distinguished} stratum of the building block $\UUUU$. It is given  on $\widetilde{\UUUU}$ 
    by $\psi^{-1}(x_0=\dotsb=x_{r}=0)$ in terms of the \'etale morphism
    $\psi$ in \ref{strictly semistable}. We note from \cite[Proposition 5.2]{GuCan} that the building blocks form a basis of topology for the strictly semistable formal scheme $\XXXX$. 
\end{art}

\begin{art} \label{skeleton}
Berkovich showed in \cite{BerContra1} and in \cite{BerContra2} that for a formal strictly semistable scheme $\XXXX$ over ${K^\circ}$ there is a canonical deformation retraction $\tau\colon \XXXX_\eta \to S(\XXXX)$ to a canonical piecewise linear subspace $S(\XXXX)$ of $\XXXX_\eta$ called the {\it skeleton} of $\XXXX$. We sketch the construction referring to \cite[\S 5]{GuCan} for details.

In the construction of $S(\UUUU)$ for a building block $\UUUU$, one uses the skeleton $S(D)$ of the closed affine torus $D$ in ${\mathbb G}_m^{d+1}$ given by the equation $x_0 \dotsb x_r = \pi$. Note that the skeleton $S(D)$ can be canonically identified with $\{u \in \RR^{r+1} \mid u_0+ \dotsb + u_r=\val(\pi)\}$ and hence it contains $\Delta(r,\pi) \times \{{\bf 0}\} \subset \RR^{r+1} \times \RR^{d-r}$ for the $\Gamma$-rational standard simplex
\begin{equation} \label{standard simplex}
\Delta(r,\pi)\coloneq\bigl\{u \in \RR_{\geq 0}^{r+1} \mid u_0 + \dotsb + u_r = \val(\pi) \bigr\}
\end{equation}
associated to $\UUUU$. Then we define $S(\UUUU)\coloneq \psi^{-1}(\Delta(r,\pi) \times \{{\bf 0}\}  )$. One further notices that $S(\UUUU)$ depends only on the stratum $S$ of $\widetilde{\XXXX}$ which contains  the distinguished stratum of the building block $\UUUU$. We call $S(\UUUU)$ the {\it canonical simplex associated to $S$} and  we  denote it by $\Delta_S$. We note that any stratum of $\widetilde{\XXXX}$ contains the distinguished stratum of a suitable building block as a dense subset. We use the canonical homeomorphism 
$$\Val\colon \Delta_S \longrightarrow \Delta(r,\pi), \quad x \longmapsto \left(\val(x_0), \dotsc, \val(x_r)\right)$$
to see $\Delta_S$ as a $\Gamma$-rational simplex. Now the skeleton $S(\XXXX)$ is defined as the union of all $S(\UUUU)$ and hence it is a compact subset of $\XXXX_\eta$. The piecewise linear structure is given by the closed faces $\Delta_S$ and we have the order-reversing correspondence $\Delta_S \subset \Delta_T$ if and only if $T \leq S$ for strata $S,T$ of $\widetilde{\XXXX}$. It is clear that the integral  structure is preserved on overlappings. The retraction map $\tau\colon \XXXX_\eta \to S(\XXXX)$ is given on $\UUUU_\eta$ by $\Val$ and the natural identification $\Delta_S \cong \Delta(r,\pi)$. It is shown in \cite[Theorem 5.2]{BerContra1} that $\tau$ is a proper strong deformation retraction. The stratum--face  correspondence is an order-reversing bijective correspondence between the strata $S$ of $\widetilde{\XXXX}$ and the canonical simplices $\Delta_S$  of $S(\XXXX)$ given by 
$$\ri(\Delta_S) = \tau(\red^{-1}(S)), \quad \red(\tau^{-1}(\ri(\Delta_S))) = S.$$
Moreover, we have $\dim(S)=d-\dim(\Delta_S)$. In particular, we get a bijective correspondence between the irreducible components of $\widetilde{\XXXX}$ and the vertices of $S(\XXXX)$. A vertex $u$ means a canonical simplex of dimension $0$ and we will denote the associate irreducible component by $Y_u$. 
\end{art}

\begin{art} \label{semistable curve}
For every smooth projective curve $X$ over $K$  and every admissible formal $K^\circ$-model $\XXXX_0$ of $X$, there is a strictly semistable formal model $\XXXX$ over ${K^\circ}$ such that the canonical isomorphism on the generic fibers extends to a morphism $\XXXX \to \XXXX_0$. This is proved in \cite[§\,7]{BLStable1}. In the case of a curve, the skeleton is also called the {\it dual graph} of $\XXXX$. 
\end{art}

\begin{art} \label{functoriality}
Let $\XXXX$ and $\XXXX'$ be strictly formal schemes over ${K^\circ}$. A morphism $\varphi\colon\XXXX' \to \XXXX$ of  formal schemes over ${K^\circ}$ induces the  map $\tau \circ \varphi\colon S(\XXXX') \to S(\XXXX)$ which is integral $\Gamma$-affine on each canonical simplex (see \cite[Corollary 6.1.3]{BerContra2}). Here, {\it integral $\Gamma$-affine} means that the map is a translate of a linear homomorphism of the underlying integral structure by a $\Gamma$-rational vector.
\end{art}

Let $L$ be a line bundle on the proper algebraic variety $X$ over $K$. We consider a strictly semistable formal $K^\circ$-model $\XXXX$ of $X$ with a line bundle $\LLLL$ which is a formal model of $L$.  This model induces a formal metric $\metr_\LLLL$ on $L$ (see §\,\ref{localsec}).

\begin{prop} \label{formal metrics and skeleton}
Let $s$ be an invertible meromorphic section of $L$ and let  $\psi$ be the restriction of $-\log\|s\|_{\LLLL}$ to the skeleton $S(\XXXX)$. 
\begin{itemize}
\item[(i)] Then $\psi$ is a well-defined function on $S(\XXXX)$.
\item[(ii)]  Any canonical simplex $\Delta_S$ is covered by finitely  $\Gamma$-rational polyhedra $\Delta_j$ such that $\psi|_{\Delta_j}$ is a $\Gamma$-lattice function.
\item[(iii)] If $s$ is a nowhere vanishing global section, then the restriction of $\psi$ to $\Delta_S$ is even a $\Gamma$-lattice function and we have $\psi \circ \tau =  -\log\|s\|_{\LLLL}$ on $\Xan=\XXXX_\eta$. 
\end{itemize}
\end{prop}

\begin{proof}
Since the building blocks form a basis, we may assume that $\XXXX$ is a building block $\UUUU$ and that  $\LLLL=\OO_{\XXXX}$. Then $s$ is a meromorphic function on $X$ and $\psi$ is the restriction of $-\log|f|$ to $S(\XXXX)$. Then the claim follows from the proofs of propositions 5.2 and 5.6 in \cite{GRW}. 
\end{proof}

\begin{prop} \label{formal metrics vs piecewise linear}
For a strictly semistable formal $K^\circ$-model $\XXXX$ of the proper algebraic variety $X$, there are canonical isomorphisms between the following groups:
\begin{itemize}
 \item[(a)] the group $\left\{\metr_{\LLLL} \mid \text{$\LLLL$ formal model of $\OO_X$ on $\XXXX$}\right\}$ endowed with $\otimes$;
 \item[(b)] the additive group of functions $\psi\colon S(\XXXX) \to \RR$ which are  $\Gamma$-lattice functions on every canonical simplex $\Delta_S$;
 \item[(c)] the group of Cartier divisors on $\XXXX$ which are trivial on the generic fiber $\XXXX_\eta$. 
\end{itemize}
\end{prop}

\begin{proof}
Since the special fiber $\widetilde{\XXXX}$ of a strictly semistable formal scheme is reduced, it follows from \cite[Proposition 7.5]{GuLocal} that the map $D \mapsto \metr_{\OO(D)}$ is an isomorphism from the group in (c) onto the group in (a). 

By Proposition \ref{formal metrics and skeleton} and using the canonical invertible global section $s\coloneq 1$ of $\OO_X$, we get a homomorphism $\metr_{\LLLL} \mapsto \psi_{\LLLL}\coloneq -\log\|s\|_{\LLLL}$ from the group in (a) to the group in (b). 

Next, we will define a map $\psi \mapsto D_\psi$ from the group in (b) to the group in (c).  Let $\UUUU$ be a building block given as in \eqref{strictly semistable}. Let $S$ be the unique stratum of $\widetilde{\XXXX}$ containing  the distinguished stratum of $\UUUU$ in $\widetilde{\XXXX}$. Then there is $m_\Delta \in \ZZ^{r+1}$ and $\alpha_\Delta \in K^\times$ such that $\psi(u)=\langle m, u \rangle +\val(\alpha_\Delta)$ for all $u \in \Delta\coloneq \Delta_S = \Delta(r,\pi)$. Using the coordinates $x_0,\dotsc, x_r$ from \ref{strictly semistable}, we define the equation $\alpha_\Delta \psi^*(x_0)^{m_0} \dotsb \psi^*(x_r)^{m_r}$ on $\UUUU_\eta$. We claim that this defines a Cartier divisor $D_\psi$ on $\XXXX$. 
By construction, the absolute values of the equations agree on overlappings and hence the equations are equal up to units as claimed. 
It is clear from the construction that $D_\psi$ is trivial on the generic fiber. 

Our goal is to show that all these homomorphisms are isomorphisms. 
Let us consider the formal metric $\metr_{\LLLL}$ for a line bundle $\LLLL$ as in (a) and let $\psi\coloneq \psi_{\LLLL}$. We claim that $\LLLL= \OO(D_\psi)$ as formal models for $\OO_X$. We cover $\XXXX$ by building blocks $\UUUU$. Then the Cartier divisor $\dvs(1)$ associated to the meromorphic section $1$ of $\LLLL$ is given on $\UUUU$ by some $\gamma \in \OO(\UUUU_\eta)^\times$. Let $S$ be the unique stratum of $\widetilde{\XXXX}$ containing  the distinguished stratum of $\UUUU$ and let $\Delta\coloneq\Delta_S$. We use the same equations for the Cartier divisor $D_\psi$ as above. It follows from  \cite[Proposition 2.11]{GuTrop} that $\gamma$ agrees with $\alpha_\Delta \psi^*(x_0)^{m_0} \cdots \psi^*(x_r)^{m_r}$ up to multiplication by a unit in $K^\circ$. This proves $\dvs(1)=D_\psi$ and hence  $\LLLL= \OO(D_\psi)$. 

Conversely, if $D$ is a Cartier divisor on $\XXXX$ which is trivial on $\XXXX_\eta$, then we always find a covering of $\XXXX$ by building blocks $\UUUU$ such that $D$ is given on $\UUUU$ by $\gamma \in \OO(\UUUU_\eta)^\times$. Let $S$ be the unique stratum of $\widetilde{\XXXX}$ containing the distinguished stratum of $\UUUU$ and let $\Delta \coloneq\Delta_S$. Then as above, we may assume that $\gamma = \alpha_\Delta \psi^*(x_0)^{m_0} \dotsb \psi^*(x_r)^{m_r}$ for suitable  $m_\Delta \in \ZZ^{r+1}$ and $\alpha_\Delta \in K^\times$. Let $\psi\colon S(\XXXX) \to \RR$ be the restriction of $-\log\|s_D\|_{\OO(D)}$ to $S(\XXXX)$. By construction, we have $D_\psi = D$ on $\UUUU$ and hence all maps between the groups in (a)--(c) are isomorphisms.
\end{proof}

\begin{art} \label{strictly semistable models and subdivision}
Let $\XXXX$ be a strictly semistable formal model of $X$ over $K^\circ$. We consider a {\it regular subdivision} $\Dcal$ of the skeleton $S(\XXXX)$ which means the following:
\begin{itemize}
\item[(a)] Every $\Delta \in \Dcal$ is a subset of a canonical face of $S(\XXXX)$ and is integral $\Gamma$-affine isomorphic to a standard simplex of the form \eqref{standard simplex}. 
\item[(b)] For every canonical face $\Delta_S$ of $S(\XXXX)$, the set 
  $\{\Delta \in \Dcal \mid \Delta \subset \Delta_S\}$ is a polyhedral complex with support equal to $\Delta_S$. 
\end{itemize}
We will show that the regular subdivision $\Dcal$ induces a canonical strictly semistable formal model $\XXXX''$ of $X$ over $\XXXX$ with skeleton $S(\XXXX'') = S(\XXXX)$ as a subset of $\Xan$, with $\tau''=\tau$ for the canonical retractions and with canonical simplices of $S(\XXXX'')$ agreeing with the subdivision $\Dcal$. 

We recall the construction of $\XXXX''$ from \cite[5.5, 5.6]{GuCan}. Let $\UUUU$ be a building block of $\XXXX$ as in \ref{strictly semistable} and let $S$ be the unique stratum of $\widetilde{\XXXX}$ containing  the distinguished stratum of $\widetilde{\UUUU}$. Then $\Delta_S$ denotes the associated canonical simplex of $S(\XXXX)$. Let $\Delta \in \Dcal$ with $\Delta \subset \Delta_S$. We identify $\Delta_S=\Delta(r,\pi)$ as in \eqref{skeleton}. Using the relation $u_0=\val(\pi)-u_1-\dotsb-u_r$,  the simplex $\Delta(r,\pi)$ leads to the simplex
$$\Delta_0(r,\pi)=\{u \in \RR_{\geq 0}^r \mid 0 \leq u_1 + \dotsb + u_r \leq \val(\pi)\}$$
and $\Delta$ induces a $\Gamma$-rational simplex $\Delta_0 \subset \Delta_0(r,\pi)$. 
This allows us to work with the $\TT_S^r$-toric schemes $\UUU_{\Delta_0}$ and $\UUU_{\Delta(r,\pi)}$, where $\TT_S^r$ is the split affine torus over $S$ of rank $r$ with coordinates $x_1, \dotsc ,x_r$. Let $\UUUU_{\Delta_0}$ and $\UUUU_{\Delta(r,\pi)}$ be the associated formal schemes obtained by $\rho$-adic completion for some non-zero $\rho \in \Kmax$.
Since $\Delta_0 \subset \Delta_0(r,\pi)$, we have a canonical morphism $\iota_1\colon \UUUU_{\Delta_0} \to \UUUU_{\Delta_0(r,\pi)}$. Let  $\BBBB^{d-r}\coloneq \spf(K^\circ\langle x_{r+1}, \dotsc, x_d\rangle)$ be the formal ball of dimension $d-r$.  
We may skip the coordinate $x_0$ in \eqref{eq:formal.sss.pair} by the relation $x_0 \dotsb x_r = \pi$ and then $\psi$ may be seen as a morphism $\UUUU \to \UUUU_{\Delta_0(r,\pi)} \times \BBBB^{d-r}$.  We form the cartesian square
\begin{equation} \label{cartesian diagram for subdivision}
\begin{CD} 
\UUUU'' @>{\psi'}>> \UUUU_{\Delta_0} \times \BBBB^{d-r}  @>{p_1'}>>  \UUUU_{\Delta_0} \\
@VV{\iota'}V  @VV{\iota}V  @VV{\iota_1}V \\
\UUUU  @>{\psi}>> \UUUU_{\Delta_0(r,\pi)} \times \BBBB^{d-r}   @>{p_1}>>  \UUUU_{\Delta_0(r,\pi)}
\end{CD}
\end{equation}
where $p_1$ is the first projection. This defines the building block $\UUUU''$ of $\XXXX''$. We glue the building blocks $\UUUU''$ along common faces of the subdivision $\Dcal$ and then along overlappings of the building blocks $\UUUU$ which leads to a formal model $\XXXX''$ of $X$ over ${K^\circ}$. By construction, we have a canonical morphism $\varphi_0\colon \XXXX'' \to \XXXX$ extending the identity on $X$. Since $\Delta_0$ is isomorphic to a standard simplex of the form \eqref{standard simplex} and since $\psi'$ is \'etale, we conclude that $\XXXX''$ is strictly semistable. It follows from the above construction and the definitions in \ref{skeleton} that $S(\XXXX'')=S(\XXXX)$ as a set, that the canonical faces of $S(\XXXX'')$ agree with the regular subdivision $\Dcal$ and that $\tau''=\tau$. 
\end{art}

\begin{art} \label{skeleton containement}
Let $\varphi\colon \XXXX' \to \XXXX$ be a morphism of strictly semistable models of the proper algebraic variety $X$ over $K$ extending the identity on $X$. Berkovich gave in \cite[Theorem 4.3.1]{BerContra2} an intrinsic characterization of the points of $S(\XXXX)$ as the maximal points of a certain  partial order on $X^{\an}=\XXXX^{\an}$ depending canonically on the strictly semistable formal model $\XXXX$. He mentioned after the definition  at the beginning of \cite[§\,4.3.]{BerContra2} that this partial order is compatible with morphisms of formal schemes which implies immediately that $S(\XXXX) \subset S(\XXXX')$ in our situation above and that the contractions agree on $S(\XXXX)$. 

It follows from the projection formula in \cite[Proposition 4.5]{GuLocal} that every irreducible component $Y$ of $\widetilde{\XXXX}$ is dominated by exactly one irreducible component $Y'$ of $\widetilde{\XXXX}'$. Moreover, in this case the induced morphism $Y' \to Y$ is a proper birational morphism. We conclude from the stratum--face correspondence in \ref{skeleton} that the vertices of $S(\XXXX)$ are vertices of $S(\XXXX')$. It follows from  \ref{functoriality} that every canonical simplex of $S(\XXXX')$ with relative interior intersecting $S(\XXXX)$ is in fact contained in a canonical simplex of $S(\XXXX)$. Putting these two facts together, we get a regular subdivision $\Dcal \coloneq \{\Delta' \subset S(\XXXX) \mid \text{$\Delta'$ canonical simplex of $S(\XXXX')$}\}$ of $S(\XXXX)$. 
\end{art}

\begin{art} \label{push-forward of metric}
Let $\varphi:\XXXX' \to \XXXX$ be a morphism of strictly semistable formal schemes over $\Kval$ extending the identity on $X$ and let $\Dcal$ be the regular subdivision of $S(\XXXX)$ constructed in \ref{skeleton containement}.   We apply \ref{strictly semistable models and subdivision} to this subdivision $\Dcal$ and we get an associated strictly semistable formal scheme $\XXXX''$ and a canonical morphism $\varphi_0:\XXXX'' \to \XXXX$ extending the identity on $X$. Now let us consider a line bundle $L$ on $X$ which has a formal model $\LLLL$ on $\XXXX$. In this situation and for $\LLLL'$ a formal model of $L$ on $\XXXX'$, we define the following formal metric $\metr_{\LLLL',S(\XXXX)}$ on $L$: By Proposition \ref{formal metrics vs piecewise linear}, the metric $\metr_{\LLLL'}/\metr_{\LLLL}$ corresponds to a piecewise linear function $\psi'$ on $S(\XXXX')$. Let $\psi$ be the restriction of $\psi'$ to $S(\XXXX)$. Then $\psi$ is also piecewise linear function  satisfying the requirements of Proposition \ref{formal metrics vs piecewise linear}(b) for the strictly semistable model $\XXXX''$ from \ref{strictly semistable models and subdivision}. By Proposition \ref{formal metrics vs piecewise linear}, we get an associated vertical Cartier divisor $D_\psi$ on $\XXXX''$ 
and we define 
$$\metr_{\LLLL',S(\XXXX)}\coloneq \metr_{\OO(D_\psi)} \otimes \metr_{\LLLL}.$$
To get a geometric idea of this definition, we note that the value of $\psi$ in a vertex $u$ of $S(\XXXX'')$ is the multiplicity of the Weil divisor associated to $D_\psi$ in the corresponding irreducible component $Y_u$ of $\widetilde{\XXXX''}$. We have a similar description for the Weil divisor of $D_{\psi'}$ in  the vertices of the skeleton $S(\XXXX')$ and 
hence we may view $D_\psi$ as some sort of push-forward of $D_{\psi'}$. This could be made more precise if the identity on $X$ extends to a morphism $\XXXX' \to \XXXX$, but we will not need it.
\end{art}

\begin{prop} \label{properties of push-forward metric}
Under the hypotheses from \ref{push-forward of metric}, we have the following properties:
\begin{itemize}
 \item[(i)] The definition of $\metr_{\LLLL',S(\XXXX)}$ is independent of the choice of $\LLLL$.
 \item[(ii)] There is a formal model $\LLLL''$ of $L$ on $\XXXX''$ with $\metr_{\LLLL''}=\metr_{\LLLL',S(\XXXX)}$.
 \item[(iii)] The line bundle $\LLLL''$ in (ii) is unique up to isomorphism.
 \item[(iv)] The construction of the metric $\metr_{\LLLL',S(\XXXX)}$ is additive in $\LLLL'$. 
 \item[(v)] If there is a morphism $\varphi_1:\XXXX' \to \XXXX''$ extending the identity on $X$ and factorizing through  $\varphi$ with $\LLLL' \cong \varphi_1^*(\LLLL_1)$  for a line bundle $\LLLL_1$ on $\XXXX''$, then $\LLLL'' \cong \LLLL_1$. 
\end{itemize}
\end{prop}

\begin{proof}
Let $\LLLL_1$ and $\LLLL_2$ be two formal models of $L$ on $\XXXX$. For $i=1,2$, we denote by $\psi_i$  
the corresponding piecewise linear function on $S(\XXXX)$  constructed in \ref{push-forward of metric}. Then the canonical vertical Cartier divisor $D$ on $\XXXX$ with $\OO(D) = \LLLL_1 \otimes \LLLL_2^{-1}$ corresponds in Proposition \ref{formal metrics vs piecewise linear} to the piecewise linear function $\psi_1-\psi_2$ on $S(\XXXX)$. By linearity, this proves (i). 
By construction, we get (ii). 
Property (iii)  follows from \cite[Proposition 7.5]{GuLocal} using that a strictly semistable formal scheme has reduced special fiber. Property (iv) and (v) are obvious from the construction. 
\end{proof}

\begin{rem} \label{transitivity of push-forward metric}
Let $\varphi:\XXXX' \to \XXXX$ and  $\psi\colon \XXXX \to \YYYY$ be morphisms of strictly semistable formal models of $X$ extending the identity on $X$. We assume that the line bundle $L$ on $X$ has a formal model on $\YYYY$. For the line bundle $\LLLL''$ on $\XXXX''$ with $\metr_{\LLLL''}=\metr_{\LLLL',S(\XXXX)}$, we have the transitivity property
$$\metr_{\LLLL',S(\YYYY)} = \metr_{\LLLL'',S(\YYYY)}.$$
This follows easily from $S(\XXXX'')=S(\XXXX)$ and the construction in \ref{push-forward of metric}. 
\end{rem}

\begin{art} \label{divisor of piecewise linear function}
Now we return to the case of a proper algebraic curve $X$ over $K$ with strictly semistable formal $K^\circ$-model $\XXXX$. Then the skeleton $S(\XXXX)$ is the dual graph which is canonically a metrized graph using lattice length on each segment. For a function $\psi\colon S(\XXXX) \to \RR$ which is piecewise affine on each segment, we define the formal sum
$$\dvs(\psi) \coloneq \sum_{v \in S(\XXXX)} \sum_{e \ni v} d_e\psi(v) [v]$$
on $S(\XXXX)$, where $e$ ranges over all edges containing $v$ and where $d_e(\psi)(v)$ is the slope of $\psi$ at $v$ along the edge $e$. We view $\dvs(\psi)$ as a divisor on the dual graph. For any $P \in X(K)$, we define $\tau_*([P])\coloneq [\tau(P)]$ and we extend the map $\tau_*$ linearly to all cycles of $X$. 
\end{art}

The following result is due to Katz, Rabinoff and Zureick-Brown. Note that they use another sign in the definition of $\dvs(\psi)$. 

\begin{thm} \label{slope formula}
Let $X$ be a smooth proper curve over $K$ and let $\XXXX$ be a strictly semistable formal model of $X$ over ${K^\circ}$. Let $s$ be an invertible meromorphic section of the line bundle $L$ on $X$ with formal model $\LLLL$ on  $\XXXX$. Then $ -\log\|s\|_\LLLL$ restricts to a piecewise linear function $\psi$ on $S(\XXX)$ and 
we have the slope formula  $$\tau_*(\dvs(s)) + \dvs(\psi)= \sum_v \deg(\LLLL|_{Y_v})[v],$$
where $v$ ranges over the vertices of the dual graph $S(\XXXX)$.
\end{thm}

\begin{proof} 
See \cite[Proposition 2.6]{KRZ}.
\end{proof}

\begin{prop} \label{push-forward of measures}
Let $\varphi\colon \XXXX' \to \XXXX$ be a morphism of strictly semistable formal $K^\circ$-models of the proper curve $X$ extending the identity. We assume that the line bundle $L$ has a formal model on $\XXXX$ and that $\LLLL'$ is a formal model of $L$ on $\XXXX'$. For the canonical retraction $\tau\colon X^{\an} \to S(\XXXX)$, we have 
$$\tau_*(c_1(L,\metr_{\LLLL'})=c_1(L,\metr_{\LLLL',S(\XXXX)}).$$\end{prop}

\begin{proof} We first assume that $\LLLL' \cong \varphi^*(\LLLL)$ for a formal model $\LLLL$ of $L$ on $\XXXX$. Then the projection formula (Proposition \ref{propmeasure}\,(ii)) shows that $c_1(L,\metr_{\LLLL'}) = c_1(L,\metr_{\LLLL})$. By Proposition \ref{properties of push-forward metric}(v), we have $\metr_{\LLLL',S(\XXXX)} = \metr_{\LLLL}$ and since the Chambert-Loir measure $c_1(L,\metr_{\LLLL})$ is supported in the vertices of $S(\XXXX)$, we get the claim in this special case.

Now we handle the general case. Using the above and the linearity of the constructions, we may assume that $L=\OO_X$. 
Moreover, the above special case and Proposition \ref{properties of push-forward metric}(v) show that we may replace $\XXXX$ and $\XXXX'$ by models associated to compatible $\Gamma$-rational subdivisions of the dual graphs $S(\XXXX)$ and $S(\XXXX')$. Hence we may assume that the piecewise linear map $\tau\colon S(\XXXX') \to S(\XXXX)$ from \ref{functoriality} maps vertices to vertices. 

Now we consider a vertex $u'$ of $S(\XXXX')$ which is not contained in $S(\XXXX)$. We claim that any edge $e'$ of $S(\XXXX')$ with vertex $u'$ is contracted by $\tau$. To prove that, we note that the projection formula shows that the irreducible component $Y_{u'}$ of $\widetilde{\XXXX}'$ corresponding to $u'$ is mapped to a closed point by $\varphi$. By the stratum--face correspondence in \ref{skeleton}, this point is contained in the dense open stratum of the irreducible component $Y_u$ of $\widetilde{\XXXX}'$ corresponding to the vertex $u=\tau(u')$. In particular, we deduce that the double point corresponding to the edge $e'$ is mapped to this open stratum and hence the whole edge is mapped to $u$.

Since the dual graph $S(\XXXX')$ is connected, we conclude that 
$u'$ is connected by an edge-path to the vertex $u=\tau(u')$ in $S(\XXXX)$. Since $\tau$ is a retraction, we conclude that $\tau^{-1}(u) \cap S(\XXXX')$ is a tree with finitely many edges $e_i$ for $i=0,\dotsc,r$. 
We denote the vertices of $e_i$ by $v_i$ and $w_i$ using some orientation.
Let $\mu=c_1(L,\metr_{\LLLL',S(\XXXX)})$ and let $\mu'=c_1(L,\metr_{\LLLL'})$. 

Let $\psi'$ be the restriction of $-\log\|1\|_{\LLLL'}$ to $S(\XXXX')$. We have seen in Theorem \ref{slope formula} that $\psi'$ is affine on each edge of $S(\XXXX')$. Moreover,   
for a vertex $w \neq u$ in the tree, we have 
$$\mu'(\{w\})= \sum_{v_i=w} d_{e_i}\psi'(w) + \sum_{w_i=w}  d_{e_i}\psi'(w)$$
and 
$$\mu'(\{u\}) = \sum_{v_i=u} d_{e_i}\psi'(u) + \sum_{w_i=u} d_{e_i}\psi'(u)+ \sum_{e \ni u} d_e\psi'(u),$$
where $e$ ranges over all edges of $S(\XXXX)$ with vertex $u$. On the other hand, we have 
$$\mu(\{u\})= \sum_{e \ni u}d_e\psi'(u).$$
Since $\psi'$ is affine on each edge $e_i$, we have $d_{e_i}\psi'(v_i)=-d_{e_i}\psi'(w_i)$.  
Using that we deal with a tree, we deduce for the multiplicity $(\tau_*(\mu'))(\{u\})$ of the discrete measure $\tau_*(\mu')$ in $u$ that 
$$(\tau_*(\mu'))(\{u\})=\sum_{w } \mu'(\{w\}) = \sum_{e \ni u} d_e\psi'(u)=\mu(\{u\})$$
where $w$ ranges over all vertices of the tree contracted to $u$. 
\end{proof}

\begin{cor} \label{push-forward of semipositiv}
Let $L$ be a line bundle on the proper smooth curve $X$ over $K$ which has a formal model $\LLLL$ on $\XXXX$ and let $\varphi\colon \XXXX' \to \XXXX$ be a morphism of strictly semistable models of $X$ extending the identity. If $\LLLL'$ is a formal model of $L$ on $\XXXX'$ such that the formal metric $\metr_{\LLLL'}$ is semipositive, then $\metr_{\LLLL',S(\XXXX)}$ is a semipositive formal metric.
\end{cor}

\begin{proof}
It follows from Proposition \ref{push-forward of measures} that $c_1(L,\metr_{\LLLL',S(\XXXX)})$ is a positive measure and hence $\metr_{\LLLL',S(\XXXX)}$ is a semipositive formal metric.  
\end{proof} 

\begin{prop} \label{strictly semistable toric schemes}
Let $\XXX_\Pi$ be the toric scheme of relative dimension $n$ over $K^\circ$ associated to a complete $\Gamma$-rational polyhedral complex $\Pi$ (see \ref{fanpol}). Then $\XXX_\Pi$ is strictly semistable if  every maximal polyhedron of $\Pi$ has an integral $\Gamma$-affine isomorphism onto a polyhedron of the form $\Delta(r,\pi) \times \RR_{\geq 0}^{n-r}$ for suitable $r \in \NN_{\leq n}$ and non-zero $\pi$ in $K^{\circ\circ}$. 
\end{prop}

\begin{proof} 
This follows easily from the definitions. 
\end{proof}


In \ref{torification}, we have introduced the torification $\metr_{\SSS}$ of the metric on a toric line bundle over a toric variety over $K$. In the next result, we apply the above theory to the toric variety $\PP_K^1$. 

\begin{cor} \label{torification of semipositive}
The torification of a semipositive formal metric on $\PP_K^1$ is semipositive.
\end{cor}

\begin{proof} Note that algebraic metrics and formal metrics are the same (see Proposition \ref{algmod}). We choose the toric model $\PP_{K^\circ}^1$. Let $L$ be the underlying line bundle of the semipositive formal metric $\metr$ in question. Then $L$ is isomorphic to $\OO_{\PP_{K}^1}(k)$ for some $k \geq 0$ and hence we find a canonical model $\OO_{\PP_{K^\circ}^1}(k)$ on $\PP_{{K^\circ}}^1$. Let $s$ be a non-trivial global section and let $\psi$ be the restriction of $-\log( \|s\|/\|s\|_{\OO(k)})$ to the skeleton $S(\TT)=N_\RR$ of the dense torus $\TT = \PP_K^1 \setminus \{0,1\}$ (see \ref{torfunc2}). Then $\psi$ is a piecewise affine function on the skeleton $S(\TT)$ and there is a  $\Gamma$-rational polyhedral subdivision of the fan of $\PP_K^1$ such that $\psi$ is a $\Gamma$-lattice function on the segments and halflines (see Proposition \ref{alg-pa}). Let $\XXXX$ be the associated toric model of $\PP_K^1$. By Proposition \ref{strictly semistable toric schemes}, we note that $\XXXX$ is a strictly semistable formal scheme and it is obvious that $S(\XXXX)$ is the bounded part of the polyhedral subdivision. Note also that $\psi$ is constant on the two halflines as the recession function is associated to the trivial line bundle (see Proposition \ref{tor-conc}). 

The formal metric $\metr$ is given by a formal model $\LLLL'$ on a  formal model $\XXXX'$ of $\PP_K^1$. By \ref{semistable curve}, we may assume that $\XXXX'$ is strictly semistable and that we have a canonical morphism $\varphi\colon \XXXX' \to \XXXX$ extending the identity. By construction, the metric $\metr_{\LLLL',S(\XXXX)}$ 
from \ref{push-forward of metric} is the torification of $\metr$. Then the claim follows from Corollary \ref{push-forward of semipositiv}. 
\end{proof}

\bibliographystyle{plain}
\bibliography{Literatur}

\begin{thebibliography}{10}

\bibitem{Ber}
V.G. Berkovich.
\newblock {\em Spectral theory and analytic geometry over non-Archimedean
  fields}, volume~33.
\newblock American Mathematical Society, 1990.

\bibitem{Ber2}
V.G. Berkovich.
\newblock \'{E}tale cohomology for non-{A}rchimedean analytic spaces.
\newblock {\em Publ. Math. IH\'ES}, 78(78):5--161, 1993.

\bibitem{BerContra1}
V.G. Berkovich.
\newblock Smooth {$p$}-adic analytic spaces are locally contractible.
\newblock {\em Invent. Math.}, 137(1):1--84, 1999.

\bibitem{BerContra2}
Vladimir~G. Berkovich.
\newblock Smooth {$p$}-adic analytic spaces are locally contractible. {II}.
\newblock In {\em Geometric aspects of {D}work theory. {V}ol. {I}, {II}}, pages
  293--370. Walter de Gruyter GmbH \& Co. KG, 2004.

\bibitem{BG}
E.~Bombieri and W.~Gubler.
\newblock {\em {Heights in Diophantine Geometry}}.
\newblock Cambridge University Press, 2006.

\bibitem{BGR}
S.~Bosch, U.~G{\"u}ntzer, and R.~Remmert.
\newblock {\em Non-Archimedean analysis}.
\newblock Springer, 1984.

\bibitem{BL2}
S.~Bosch and W.~L{\"u}tkebohmert.
\newblock Formal and rigid geometry. {II}. {F}lattening techniques.
\newblock {\em Math. Ann.}, 296(3):403--429, 1993.

\bibitem{BLR}
S.~Bosch, W.~L{\"u}tkebohmert, and M.~Raynaud.
\newblock Formal and rigid geometry. {IV}. {T}he reduced fibre theorem.
\newblock {\em Invent. Math.}, 119(2):361--398, 1995.

\bibitem{BLStable1}
Siegfried Bosch and Werner L{\"u}tkebohmert.
\newblock Stable reduction and uniformization of abelian varieties. {I}.
\newblock {\em Math. Ann.}, 270(3):349--379, 1985.

\bibitem{BGS}
J.B. Bost, H.~Gillet, and C.~Soul{\'e}.
\newblock {Heights of projective varieties and positive Green forms}.
\newblock {\em Journal of the American Mathematical Society}, 7(4):903--1027,
  1994.

\bibitem{BPS}
J.I. Burgos~Gil, P.~Philippon, and M.~Sombra.
\newblock Arithmetic geometry of toric varieties. metrics, measures and
  heights.
\newblock {\em Ast{\'e}risque}, 360, 2014.

\bibitem{BPS3}
J.I. Burgos~Gil, P.~Philippon, and M.~Sombra.
\newblock Height of varieties over finitely generated fields.
\newblock {\em Kyoto J. Math.}, 56(1):13--32, 2016.

\bibitem{BS}
J.I. Burgos~Gil and M.~Sombra.
\newblock When do the recession cones of a polyhedral complex form a fan?
\newblock {\em Discrete \& Computational Geometry}, 46(4):789--798, 2011.

\bibitem{Cha}
A.~Chambert-Loir.
\newblock Mesures et {\'e}quidistribution sur les espaces de berkovich.
\newblock {\em Journal f{\"u}r die reine und angewandte Mathematik (Crelles
  Journal)}, 595:215--235, 2006.

\bibitem{CT}
A.~Chambert-Loir and A.~Thuillier.
\newblock Mesures de mahler et {\'e}quidistribution logarithmique.
\newblock {\em Annales de l'institut Fourier}, 59(3):977--1014, 2009.

\bibitem{CLS}
D.A. Cox, J.B. Little, and H.K. Schenck.
\newblock {\em Toric Varieties}.
\newblock Graduate studies in mathematics. American Mathematical Society, 2011.

\bibitem{EGAI}
J.A. Dieudonn{\'e} and A.~Grothendieck.
\newblock \'{E}l\'ements de g\'eom\'etrie alg\'ebrique: {I}. {L}e langage des
  sch\'emas.
\newblock {\em Publ. Math. IH{\'E}S}, 4:5--228, 1960.

\bibitem{Duc}
A.~Ducros.
\newblock Les espaces de {B}erkovich sont excellents.
\newblock {\em Ann. Inst. Fourier (Grenoble)}, 59(4):1443--1552, 2009.

\bibitem{Fal}
G.~Faltings.
\newblock Diophantine approximation on abelian varieties.
\newblock {\em Ann. of Math. (2)}, 133(3):549--576, 1991.

\bibitem{Fu2}
W.~Fulton.
\newblock {\em Introduction to Toric Varieties}.
\newblock Annals of mathematics studies. Princeton University Press, 1993.

\bibitem{Fu}
W.~Fulton.
\newblock {\em {Intersection Theory}}.
\newblock Springer, 1998.

\bibitem{GS}
H.~Gillet and C.~Soul{\'e}.
\newblock Arithmetic intersection theory.
\newblock {\em Inst. Hautes \'Etudes Sci. Publ. Math.}, 72(72):93--174 (1991),
  1990.

\bibitem{GW}
U.~G{\"o}rtz and T.~Wedhorn.
\newblock {\em Algebraic Geometry I: Schemes With Examples and Exercises}.
\newblock Vieweg+Teubner Verlag, 2010.

\bibitem{GuM}
W.~Gubler.
\newblock Heights of subvarieties over {$M$}-fields.
\newblock In {\em Arithmetic geometry ({C}ortona, 1994)}, Sympos. Math.,
  XXXVII, pages 190--227. Cambridge Univ. Press, Cambridge, 1997.

\bibitem{GuLocal}
W.~Gubler.
\newblock Local heights of subvarieties over non-archimedean fields.
\newblock {\em Journal f{\"u}r die reine und angewandte Mathematik},
  498:61--113, 1998.

\bibitem{GuHab}
W.~Gubler.
\newblock {\em Basic Properties of Heights of Subvarieties}.
\newblock ETH Z{\"u}rich (Mathematikdepartement), 2002.
\newblock Habilitation thesis.

\bibitem{GuPisa}
W.~Gubler.
\newblock Local and canonical heights of subvarieties.
\newblock {\em Annali della Scuola Normale Superiore di Pisa-Classe di
  Scienze-Serie V}, 2(4):711--760, 2003.

\bibitem{GuBo}
W.~Gubler.
\newblock The {B}ogomolov conjecture for totally degenerate abelian varieties.
\newblock {\em Inventiones Mathematicae}, 169(2):377--400, 2007.

\bibitem{GuTrop}
W.~Gubler.
\newblock Tropical varieties for non-archimedean analytic spaces.
\newblock {\em Inventiones mathematicae}, 169(2):321--376, 2007.

\bibitem{GuEq}
W.~Gubler.
\newblock Equidistribution over function fields.
\newblock {\em Manuscripta Mathematica}, 127(4):485--510, 2008.

\bibitem{GuCan}
W.~Gubler.
\newblock Non-archimedean canonical measures on abelian varieties.
\newblock {\em Compositio Mathematica}, 146:683--730, 4 2010.

\bibitem{Guide}
W.~Gubler.
\newblock A guide to tropicalizations.
\newblock In {\em Algebraic and combinatorial aspects of tropical geometry},
  volume 589 of {\em Contemp. Math.}, pages 125--189. Amer. Math. Soc.,
  Providence, RI, 2013.

\bibitem{GuKue}
W.~Gubler and K.~K{\"u}nnemann.
\newblock Positivity properties of metrics and delta-forms.
\newblock {\em arXiv preprint arXiv:1509.09079}, 2015.

\bibitem{GRW}
W.~Gubler, J.~Rabinoff, and A.~Werner.
\newblock Skeletons and tropicalizations.
\newblock {\em Adv. Math.}, 294:150--215, 2016.

\bibitem{GuSo}
W.~Gubler and A.~Soto.
\newblock Classification of normal toric varieties over a valuation ring of
  rank one.
\newblock {\em Doc. Math.}, 20:171--198, 2015.

\bibitem{Kaj}
T.~Kajiwara.
\newblock Tropical toric geometry.
\newblock In {\em Toric topology}, volume 460 of {\em Contemp. Math.}, pages
  197--207. American Mathematical Society, Providence, RI, 2008.

\bibitem{KRZ}
E.~Katz, J.~Rabinoff, and D.~Zureick-Brown.
\newblock Uniform bounds for the number of rational points on curves of small
  mordell--weil rank.
\newblock {\em Duke Math. J.}, 165(16):3189--3240, 2016.

\bibitem{KKMS}
G.~Kempf, F.~Knudsen, D.~Mumford, and B.~Saint-Donat.
\newblock {\em Toroidal Embeddings 1}.
\newblock Lecture Notes in Mathematics. Springer, 1973.

\bibitem{Kle}
S.L. Kleiman.
\newblock Toward a numerical theory of ampleness.
\newblock {\em Ann. of Math. (2)}, 84:293--344, 1966.

\bibitem{Mor}
A.~Moriwaki.
\newblock Arithmetic height functions over finitely generated fields.
\newblock {\em Invent. Math.}, 140(1):101--142, 2000.

\bibitem{Pay}
S.~Payne.
\newblock Analytification is the limit of all tropicalizations.
\newblock {\em Mathematical Research Letters}, 16(3):543--556, 2009.

\bibitem{Roc}
R.~T. Rockafellar.
\newblock {\em Convex analysis}.
\newblock Princeton Mathematical Series, No. 28. Princeton University Press,
  1970.

\bibitem{stacks}
{Stacks Project Authors}.
\newblock \itshape {S}tacks project.
\newblock \normalfont\ttfamily{http://stacks.math.columbia.edu}\normalfont,
  2015.

\bibitem{ZhaS}
S.-W. Zhang.
\newblock Small points and adelic metrics.
\newblock {\em J. Algebraic Geom.}, 4(2):281--300, 1995.

\bibitem{ZhaB}
S.-W. Zhang.
\newblock Equidistribution of small points on abelian varieties.
\newblock {\em Ann. of Math. (2)}, 147(1):159--165, 1998.

\end{thebibliography}

\bigskip

{\small Walter Gubler, Fakult{\"a}t f{\"u}r Mathematik,  Universit{\"a}t Regensburg,
Universit{\"a}tsstra{\ss}e 31, D-93040 Regensburg, walter.gubler@mathematik.uni-regensburg.de}

{\small Julius Hertel, Fakult{\"a}t f{\"u}r Mathematik,  Universit{\"a}t Regensburg,
Universit{\"a}tsstra{\ss}e 31, D-93040 Regensburg, julius.hertel@mathematik.uni-regensburg.de}

\end{document}